\pdfoutput=1
\documentclass[10pt]{book} 
\usepackage[reqno]{amsmath} % equation numbers on the right
\usepackage{amssymb, amsthm}
\usepackage{enumerate}  % lets you do \begin{enumerate}[(a)] and such
\usepackage{url} % formats url's nicely using \url{}
\usepackage{makeidx} % indexing
\usepackage[refpage,intoc]{nomencl} % list of symbols
\usepackage{verbatim} % allows \begin{comment}

\usepackage[nottoc]{tocbibind}

\usepackage{fancyhdr}
\usepackage[overload]{textcase} % fixes MakeUppercase so it doesn't affect stuff in math mode
\fancypagestyle{plain}{
  \fancyhf{} % clear all headers and footers
  \fancyfoot[CE,CO]{\thepage}
}
\fancypagestyle{content}{
  \fancyhf{} % clear all headers and footers
  \fancyhead[LE]{\hspace{3em}\footnotesize{\leftmark}}
  \fancyhead[RO]{{\footnotesize\rightmark}\hspace{3em}}
  \fancyfoot[CE,CO]{\thepage}
}

\makeatletter
\def\cleardoublepage{
  \clearpage
  \if@twoside\ifodd\c@page\else
  \hbox{}
  \thispagestyle{empty}
  \newpage
  \if@twocolumn\hbox{}\newpage\fi
  \fi\fi
  }
\makeatother
\newtheoremstyle{plainsl}%
	{\topsep}
	{\topsep}
	{\slshape} % only non-default setting
	{}
	{\normalfont\bfseries}
	{.}
	{ }
	{}

\swapnumbers

\theoremstyle{plainsl}
\newtheorem{theorem}{Theorem}[section]
\newtheorem{lemma}[theorem]{Lemma}
\newtheorem{corollary}[theorem]{Corollary}

\newcommand\cref[1]{Corollary~\ref{cor:#1}}

\newcommand\sqr[2]{{\vbox{\hrule height.#2pt
    \hbox{\vrule width.#2pt height#1pt \kern#1pt
        \vrule width.#2pt}\hrule height.#2pt}}}
\renewcommand\qed{%
	\ifmmode\eqno\sqr53
	\else\nolinebreak\ \hfill\sqr53\medbreak\fi}

\numberwithin{equation}{section}

\def\cfudot#1{\ifmmode\setbox7\hbox{$\accent"5E#1$}\else
    \setbox7\hbox{\accent"5E#1}\penalty 10000\relax\fi\raise 1\ht7
    \hbox{\raise.1ex\hbox to 1\wd7{\hss.\hss}}\penalty 10000
    \hskip-1\wd7\penalty 10000\box7}

\newcommand\defn[1]{\textsl{#1}\index{#1}}
\newcommand\named[2]{#1\nomenclature{#1}{#2}}

\newcommand\lf{\lfloor}
\newcommand\rf{\rfloor}

\newcommand\lc{\lceil}
\newcommand\rc{\rceil}
\newcommand\dde{\partial}

\newcommand\conj[1]{\overline{#1}}
\newcommand\bz{\conj{z}}
\newcommand\Hom{\mathop{\mathrm{Hom}}}
\newcommand\Harm{\mathop{\mathrm{Harm}}}
\newcommand\Pol{\mathop{\mathrm{Pol}}}
\newcommand\Cay{X} %\mathop{\mathrm{Cay}}}
\newcommand\inqed{\tag*{\sqr53}}
\newcommand\charac[1]{\chi\left(#1\right)}
\newcommand\slnc{sl_n(\cx)}
\newcommand\leg[2]{\Big(\frac{#1}{#2}\Big)}
\newcommand\Ima{\mathop{\mathrm{Im}}}
\newcommand\Div{\mathop{\mathrm{div}}}

\newcommand\twovec[2]{\left(\begin{matrix}#1\\#2\end{matrix}\right)}
\newcommand\twomat[4]{\left(\begin{array}{cc}#1 & #2\\#3 & #4\end{array}\right)}

\newcommand\lb{\langle}
\newcommand\rb{\rangle}

\newcommand\al{\alpha}
\newcommand\be{\beta}
\newcommand\de{\delta}
\newcommand\De{\Delta}

\newcommand\Ga{\Gamma}
\newcommand\la{\lambda}
\newcommand\La{\Lambda}
\newcommand\om{\omega}
\newcommand\Om{\Omega}
\newcommand\sg{\sigma}

\newcommand\ep{\epsilon}
\newcommand\tha{\theta}

\newcommand\scr[1]{{\mathcal #1}}

\newcommand\cA{\scr{A}}
\newcommand\cH{\scr{H}}

\newcommand\cx{{\mathbb C}}     % complexes
\newcommand\ints{{\mathbb Z}}
\newcommand\zz[1]{\ints_{#1}}
\newcommand\re{{\mathbb R}}     % reals

\newcommand\opk[1]{\mathop{\hbox{\rm #1}}\nolimits} % text as a function
\newcommand\sbs{\subseteq}

\newcommand\abs[1]{\left|#1\right|}

\newcommand\ip[2]{\left\langle#1,#2\right\rangle} %inner product
\newcommand\scp[2]{#1^T\!#2} %scalar product
\newcommand\zero{{\bf0}}
\newcommand\one{{\bf1}}
\newcommand\rk{\opk{rk}}
\newcommand\tr{\mathop{\mathrm{Tr}}}

\newcommand\spn{\opk{span}}
\newcommand\cip[2]{#1^*#2}
\newcommand\msum{\mathop{\mathrm{sum}}}
\newcommand\col{\mathop{\mathrm{col}}}
\newcommand\ad{\mathop{\mathrm{ad}}}

\makenomenclature     % use \nomenclature{$f$}{foo}
\include{thesisarxiv.ind}

\begin{document}

%% start numbering in roman
\pagenumbering{roman}
%\pagestyle{empty} % no headers and footers (for now)

%% title, declaration, abstract, and thanks
\begin{titlepage}
\vspace*{3cm}
\begin{center}
\Large
Complex Lines with Restricted Angles \\
\vspace{2cm}
\large
by \\
Aidan Roy \\
\vspace{2cm}
A thesis \\
presented to the University of Waterloo \\
in fulfilment of the \\
thesis requirement for the degree of \\
Doctor of Philosophy \\
in \\
Combinatorics and Optimization \\
\vspace{2cm}
Waterloo, Ontario, Canada, 2005 \\
\copyright Aidan Roy 2005
\end{center}
\end{titlepage}

\cleardoublepage

I hereby declare that I am the sole author of this thesis. 

I authorize the University of Waterloo to lend this thesis to other institutions or individuals for the purpose of scholarly research. \\

\vspace{5cm}

I further authorize the University of Waterloo to reproduce this thesis by photocopying or by other means, in total or in part, at the request of other institutions or individuals for the purpose of scholarly research. \\

\cleardoublepage

\begin{center}
\section*{Abstract}
\end{center}
This thesis is a study of large sets of unit vectors in $\cx^n$ such that the absolute value of their standard inner products takes on only a small number of values. 

We begin with bounds: what is the maximal size of a set of lines with only a given set of angles? We rederive a series of upper bounds originally due to Delsarte, Goethals and Seidel, but in a novel way using only zonal polynomials and linear algebra. In the process we get some new results about complex $t$-designs and also some new characterizations of tightness.

Next we consider constructions. We describe some generic constructions using linear codes and Cayley graphs, and then move to two specific instances of the problem: mutually unbiased bases and equiangular lines. Both cases are motivated by problems in quantum computing, although they have applications in digital communications as well.

Mutually unbiased bases are collections of orthonormal bases with a constant angle between vectors from different bases. We construct some maximal sets in prime-power dimensions, originally due to Calderbank, Cameron, Kantor and Seidel, but again in a novel way using relative difference sets or distance-regular antipodal covers. We also detail their numerous relations to other combinatorial objects, including symplectic spreads, orthogonal decompositions of Lie algebras, and spin models. Peripherally, we discuss mutually unbiased bases in small dimensions that are not prime powers and in real vector spaces.

Equiangular lines are collections of vectors with only one angle between them. We use difference sets from finite geometry to construct equiangular lines: these sets do not have maximal size, but they are maximal with respect to having all entries of the same absolute value. We also include some negative results about constructions of maximal sets in large dimensions.

\cleardoublepage

\begin{center}
\section*{Acknowledgements}
\end{center}
My experience as a graduate student at the University of Waterloo has been thoroughly enjoyable and enlightening; many friends, colleagues, and professors are responsible. Most significantly, my supervisor Chris Godsil has offered an enormous amount of time, effort, and knowledge to further my development as a mathematician. His contribution to this thesis is immeasurable. My work has also benefitted substantially from several discussions with Martin R\"otteler, and I am grateful for funding from the Natural Sciences and Engineering Research Council of Canada. Last but not least, I would like to thank Mom, Dad, Giles, and Claire for all their love and support.

%\blankevenpage
\cleardoublepage

\pagestyle{plain} % do the table of contents in plain headers and footers
\tableofcontents\cleardoublepage

%% restart numbering, in arabic; adjust page style
\pagenumbering{arabic}
\pagestyle{content}

%% give your chapters names - you might change the order at some point
\chapter{Introduction}\label{chap:intro}

The field of quantum information has seen enormous growth in the last five years, as the concept of a quantum computer inches closer to reality. This growth has produced a variety of new and interesting combinatorial problems. At the same time, some of the mathematics behind these problems, particularly the combinatorics of quantum measurements, is not very well studied. For this reason, the study of complex vectors with only a few angles has become active and relevant.

Let $u$ and $v$ be unit vectors in ${\mathbb C}^n$. For the purposes of this work, the \defn{angle} between $u$ and $v$ is 
\[
\abs{u^*v}^2.
\] 
An \defn{$s$-distance set} is a set of vectors in which only $s$ angles occur. This thesis examines $s$-distance sets of maximal size. In particular, we consider $1$-distance sets, also called \defn{equiangular lines}, and $2$-distance sets which can be partitioned into orthonormal bases, which are called \defn{mutually unbiased bases}.

\begin{comment}
The term "angle" comes from equiangular lines, which dates back to at least 1973 in the real case. In the complex case, where it is not clear what it means geometrically, it has still become standard terminology. Neumaier defines the angle $\tha$ as $\cos^2 \tha = \abs{u*v}^2$. The \defn{degree} of an $s$-distance set is $s$; it is also an \defn{$s$-code}. The "distance" here is $d(u,v) = \sqrt{1 - \abs{u^*v}^2}$ as proved by Neumaier. In the work of Delsarte, the \defn{degree} of $X$ is the smallest $s$ such that $X$ is an $s$-distance set, while the \defn{strength} is the largest $t$ such that $X$ is a $t$-design.
\end{comment}

\section*{Background}\label{sec:history}

Historically, the study of unit vectors has been tied closely to information theory. In the theory of communication sequences, $u^*v$ is called the \defn{cross-correlation} between $u$ and $v$, and the objective is to minimize its absolute value (see Golomb and Gong \cite{gg1}). For example, suppose vectors $u$ and $v$ (called \defn{signals}) are sent down the same communications channel at the same time in the form $x = u + v$. If their cross-correlation is $0$, then the receiver can decompose $x$ into its original parts by projecting onto the subspaces spanned by $u$ and $v$. More generally, the probability of error in decomposing $x$ is a function of $\abs{u^*v}$. \textsl{Matched filter detection}\index{matched filter detection} is one example of this process: see Proakis \cite{pro1} for details. In the same way that finding large sets with large minimal distance is the fundamental problem in coding theory, finding large sets with small maximal cross-correlation is one of the fundamental problems in signal design.

\begin{comment}
More precisely, bits $a$ and $b$ can be sent at the same time in the form $x = au + bv$, and both $a$ and $b$ can be recovered. Code division multiple access communication does this. Perhaps most useful in dealing with noise: transmit your signals so they are orthogonal to the presumed noise.

Cross-correlation also seems to be used in steganography, where information is encoded into a picture so as to appear to be noise. Correlation is used to detach the information at the end.
\end{comment}

In fact, many results about angles between unit vectors are essentially translations from coding theory. Given a codeword $x = (x_1,\ldots,x_d)$ in $\zz{2}^d$, we can construct a vector $x' \in \re^d$ via the map 
\[
x'_i := (-1)^{x_i}.
\]
Then the cross-correlation between $x'$ and $y'$ in $\re^d$ is a function of the Hamming distance between $x$ and $y$. More generally, using $p$-th primitive roots of unity, codewords in $\zz{p}^d$ translate into vectors in $\cx^d$. 

The study of real unit vectors is older than that of complex vectors. Haantjes \cite{haa1} first considered the problem of real equiangular lines in 1948, under the guise of elliptic geometry. Seidel and others made important advances, culminating in the characterization of real equiangular lines in terms of regular two-graphs in the early 1970's (see Seidel \cite{sei1}). At the same time, sequence analysts were considering cross-correlation of binary sequences, beginning as early as 1953 with Barker \cite{bar1}. Golomb for example (see \cite{gg1}) constructed binary sequences with low correlation from cyclic difference sets. Real mutually unbiased bases have barely been studied, but they are closely related to Hadamard matrices, which date back as far as Sylvester in 1867.

\begin{comment} J.J. Sylvester, ``Thoughts on inverse orthogonal...''
\end{comment}

When complex lines were studied historically, it was typically either as an afterthought to the real case or with the goal of having low cross-correlation regardless of the number of angles. Delsarte, Goethals and Seidel \cite{dgs} developed some important bounds in 1974 as an extension of their work on real vectors; Welch developed relevant bounds for signal sets in the same year. Constructions were investigated only sporadically: Lerner \cite{ler1} might have been the first in 1961. Interest in mutually unbiased bases started in 1981, when Ivanovic \cite{iva1} found maximal sets of bases and showed their usefulness in quantum applications. The development of complex equiangular lines was even later; Zauner \cite{zau1} introduced them in the quantum setting in 1999. 

\begin{comment}
Regular two-graphs introduced by Higman \cite{sei2}
Much of the characterization done in 1966 by van Lint and Seidel \cite{vls1}.
Schwinger: mub's in the 60's?
\end{comment}

\section*{Quantum mechanics}\label{sec:quantum}

Since quantum information plays such a large role in the applications for the problems at hand, we briefly review the postulates of quantum mechanics, focusing on measurements. For a more detailed review, see Nielsen and Chuang \cite[Chapter 2]{nc1}.

The first postulate of quantum mechanics says that the state of any isolated physical system is described by a \defn{state vector} $v$, which is a unit vector in a fixed complex Hilbert space. More precisely, the phase of the vector does not matter: it suffices to consider the vector projectively or up to a complex scalar unit. Often we represent a state by its projection matrix $\rho = vv^*$, which is called a \defn{density matrix}. This is a Hermitian matrix with rank $1$ and trace $1$. 
\nomenclature[\p]{$\rho$}{density matrix}%

Of course in practice no system is isolated, and a state from one system can be entangled with others. A unit vector $v$ is known as a \defn{pure state}, while a \defn{mixed state} is a collection of pure states, each occurring with a certain probability. Suppose $v_i$ occurs with probability $p_i$. Then this mixed state is represented by the density matrix
\[
\rho = \sum_i p_i v_iv_i^*.
\]
Thus an arbitrary mixed state is represented by a positive semidefinite Hermitian matrix with trace of $1$. 

The second postulate of quantum mechanics states that evolution in a closed quantum system is given by a unitary transformation. That is, if \named{$U$}{unitary matrix} is unitary, then $v \mapsto Uv$ represents a change in the system (or, equivalently, $\rho \mapsto U\rho U^*$). Note that if $v$ is a unit vector then so is $Uv$. Unitary evolutions do not play a role in this thesis.

The third postulate describes how a quantum system is measured or observed using a collection of matrices $\scr{M} = \{M_1,\ldots,M_m\}$ such that 
\[
\sum_{i=1}^m M_i^*M_i = I.
\]
Each matrix in $\scr{M}$ is called a \defn{measurement operator}, and the collection is called a \defn{measurement}. Each $M_i$ is assigned a certain probability: given a state $\rho$, we say that \textsl{outcome}\index{measurement outcome} $i$ occurs with probability  
\[
p_i = \tr(M_i\rho M_i^*).
\]
\nomenclature[a$T]{$\tr$}{trace}%
The fact that the matrices $M_i^*M_i$ sum to the identity implies that the probabilities sum to $1$. The state $M_i\rho M_i^*$, suitably normalized, is the \defn{resulting state} after measurement with outcome $i$. If $\rho = vv^*$, then $p_i$ is the length of the vector $M_iv$, and that vector is the resulting state vector after normalization.

In situations when the measurement probabilities are more important than the resulting state, the matrices $M_i$ are sometimes replaced by $E_i := M_i^*M_i$. This is called the \textsl{Positive Operator-Valued Measurement} or \textsl{POVM} formalism. A \defn{POVM} is a collection of Hermitian, positive semidefinite matrices which sum to the identity. 

One type of measurement of particular interest is the \textsl{projective}\index{projective measurement} or \textsl{von Neumann}\index{von Neumann measurement} measurement. In such a measurement, each $M_i$ is a projection onto an orthogonal subspace. Since the matrices of $\scr{M}$ sum to $I$, the direct sum of the corresponding subspaces is all of $\cx^n$. Suppose $M_i$ is the projection onto a one-dimensional subspace, say $M_i = uu^*$, where $u$ is a unit vector. If $\rho = vv^*$ is a pure state, then the probability of outcome $i$ is
\[
p_i = \tr(M_i\rho M_i^*) = \tr(M_i\rho) = \tr(uu^*vv^*) = \abs{u^*v}^2.
\]
This connection between measurements and complex angles is the prime motivation for the problems in this thesis.

The fourth postulate of quantum mechanics states that quantum systems are composed using tensor products. If $v_1$ and $v_2$ are states in Hilbert spaces $\scr{H}_1$ and $\scr{H}_2$, then the state of the composite system is $v_1 \otimes v_2$. This postulate allows for quantum entanglement: a state $v$ in $\scr{H}_1 \otimes \scr{H}_2$ is \textsl{entangled}\index{entanglement} if it cannot be written as a tensor product of states in $\scr{H}_1$ and $\scr{H}_2$. For the purposes of quantum computation, the most common states are \defn{qubits}: vectors in a $2$-dimensional space. However, we will work with systems in arbitrary dimensions.
\nomenclature[\]{$v_1 \otimes v_2$}{tensor product}
 
\begin{comment}
Entanglement is related to mixed states in that if one was observing an entangled state $v$ only in the context of $\scr{H}_1$, then it would appear as a mixed state in that system.
\end{comment}

\section*{Outline}\label{sec:outline}

Our study of maximal sets of complex unit vectors has two components: bounds, and constructions. 

We begin with a brief review of association schemes (Chapter \ref{chap:configs}), as schemes and distance-regular graphs hide underneath all of the work in this thesis. We then consider upper bounds on the size of an $s$-distance set in Chapter \ref{chap:bounds}. These are mostly due to Delsarte, Goethals, and Seidel \cite{dgs}, but we develop them in a different manner using zonal harmonics. Along the way, we get some new characterizations of equality and some new results about complex $t$-designs, which are closely related to $s$-distance sets. The theory of bounds on complex lines has a well-developed, unified structure and it fits into the larger theories of both Delsarte spaces and polynomial spaces. There is also a nice symmetry between $s$-distance sets and $t$-designs.

On the other hand, actually constructing maximal $s$-distance sets seems to be difficult, and no general technique is known. In Chapter \ref{chap:constructions} we develop some general constructions for sets of lines using error-correcting codes and Cayley graphs with few eigenvalues. These constructions work better for mutually unbiased bases (Chapter \ref{chap:mubs}) than for equiangular lines (Chapter \ref{chap:eals}): they produce maximal sets for the former. In fact more progress has been made with mutually unbiased bases overall; numerous connections to combinatorics have been discovered. Conversely, maximal sets of equiangular lines are really only understood on an algebraic level at the present time. In both cases, the problem of finding maximal sets remains open for most dimensions. 

\begin{comment}
Complex equiangular lines have not been around long enough (1999) to say if the problem is hard; mutually unbiased bases have been around for over 20 years. 
\end{comment}

\chapter{Weighted Matrix Algebras}\label{chap:configs}

In this chapter we give a brief introduction to homogeneous weighted adjacency algebras, and Hermitian algebras in particular. These algebras are a slight generalization of association schemes: the larger framework is needed to describe some results in Chapter \ref{chap:bounds}. At the end of the chapter we specialize to association schemes and distance-regular graphs, both of which will arise frequently in later chapters. 

\section{Weighted adjacency algebras}\label{sec:cfgalg}

A \defn{weighted adjacency matrix} of a graph $G$ is a matrix $A$ indexed by the vertices of $G$ with entries satisfying
\[
\abs{A_{a,b}} = \begin{cases}
1, & ab \text{ is an arc}; \\
0, & \text{otherwise}.
\end{cases}
\]
We will always assume our adjacency matrices are Hermitian. Let $A \circ B$ denote the \defn{Schur product} of $A$ and $B$:
\nomenclature[\]{$A \circ B$}{Schur product}%
\[
(A \circ B)_{a,b} = A_{a,b} B_{a,b}.
\]
A \defn{coherently-weighted configuration} is a set of weighted adjacency matrices $\cA = \{A_0, \ldots, A_d\}$ such that 
\begin{enumerate}[(a)]
\item $A_i \circ A_j = 0$ for $i \neq j$, \label{item:cschur}
\item $A_iA_j$ is in the span of $\cA$, and \label{item:cspan}
\item \named{$I$}{identity matrix} is a sum of elements of $\cA$.
\end{enumerate}
\nomenclature[a$A]{$\cA$}{coherently-weighted configuration}%
\nomenclature{$A_i$}{adjacency matrix}%
Because of \eqref{item:cschur} we say the matrices $A_i$ are \defn{Schur orthogonal}. By \eqref{item:cspan} the span of $\cA$ is closed under multiplication: we call $\spn(\cA)$ a \defn{weighted adjacency algebra}. A configuration is \textsl{homogeneous}\index{configuration!homogeneous} if $I$ is an element of $\cA$. In this case we always take $A_0 = I$.

This terminology is not standard. That is, $\cA$ is not a coherent configuration in the sense of Cameron \cite{cam2}; it is the weighting which is coherent and not the configuration. Higman \cite{hig1} calls $\cA$ a \textsl{configuration with a coherent weight}. Throughout this chapter, a configuration will refer to a homogeneous, Hermitian, coherently-weighted configuration.

Since products are in the span of $\cA$, there are constants \named{$p_{ij}(k)$}{intersection number} such that
\[
A_iA_j = \sum_{k=0}^d p_{ij}(k) A_k.
\]
These constants are called the \defn{intersection numbers}. Every matrix in the algebra is Hermitian, so the matrices commute and $p_{ij}(k) = p_{ji}(k)$. In a homogeneous configuration, $p_{ii}(0) = (A_i^2)_{a,a}$ for any vertex $a$. This is the number of vertices adjacent to $a$ in the graph of $A_i$. For this reason $p_{ii}(0)$ is called the \defn{valency} of $A_i$.

The algebra generated by $\cA$ has dimension $d+1$. Since $A_i^k$ is in this algebra for each $k$, the minimal polynomial of $A_i$ has degree at most $d+1$, and $A_i$ has at most $d+1$ distinct eigenvalues. Since the matrices of $\cA$ are Hermitian and commute, they are simultaneously diagonalizable. Let $E_0,E_1,\ldots$ be the projection matrices onto the distinct eigenspaces of $\cA$, with $p_{ij}$ the eigenvalue of $A_i$ for $E_j$. By ``distinct eigenspaces", we mean that each $E_j$ projects on a subspace of an eigenspace for $A_i$, and for each $j \neq k$, there is at least one $i$ such that $E_j$ and $E_k$ do not project onto the same eigenspace of $A_i$.
\nomenclature{$E_j$}{matrix idempotent}%

\begin{comment}
Every Hermitian matrix is diagonalizable $H = UDU^{-1}$. Commuting, diagonalizable matrices are simultaneously diagonalizable using the $E_j$'s: find the minimal projections.
\end{comment}

\begin{theorem} If $\cA = \{A_0,\ldots,A_d\}$ is a Hermitian coherently-weighted configuration, then there are orthogonal idempotents $E_0,\ldots,E_d$ in the span of $\cA$ such that 
\begin{enumerate}[(a)]
\item $\sum_j E_j = I$, \label{idema}
\item $A_iE_j = p_{ij}E_j$, and \label{idemb}
\item $\{E_0,\ldots,E_d\}$ is a basis for $\spn(\cA)$.
\end{enumerate}
\end{theorem}

\begin{proof}
Since the matrices $E_0,E_1,\ldots$ are projection matrices onto distinct eigenspaces, they are orthogonal idempotents. Since $A_i$ is Hermitian, it has a spanning set of eigenvectors; hence the projection matrices sum to the identity. 
From \eqref{idema} and \eqref{idemb} it follows that
\[
A = AI = A\sum_j E_j = \sum_j p_{ij} E_j,
\]
and so the idempotents span $\cA$. Thus there are at least $d+1$ idempotents. 

We claim that each projection matrix $E_j$ is in $\spn(\cA)$. Since $A_iE_j = p_{ij}E_j$, we get that for any polynomial $p(x)$, 
\[
p(A_i)E_j = p(p_{ij})E_j.
\]
Fix $j$, let $m(x)$ be the minimal polynomial of $A_i$, and let $f_i(x) = m(x)/(x-p_{ij}).$ Then
\[
f_i(A_i) = f_i(A_i)I = \sum_k f_i(A_i)E_k = \sum_k f_i(p_{ik})E_k = \sum_{k:p_{ik}=p_{ij}}f_i(p_{ij})E_k.
\]
Since $E_j$ and $E_k$ correspond to different eigenvalues for some $A_i$, and $E_jE_k = 0$, we find that $f_0(A_0)\cdots f_d(A_d)$ is a multiple of $E_j$. Thus $E_j$ is a polynomial in $\cA$. 

Since the projection matrices are orthogonal, they are linearly independent. Each one is in $\spn(\cA)$, so there are exactly $d+1$ of them and we have a basis. \qed
\end{proof}

The projection matrices $E_j$ are called the \textsl{idempotents}\index{configuration!idempotents} of the configuration, and the constants \named{$p_{ij}$}{eigenvalue of a scheme} are the \textsl{eigenvalues}\index{configuration!eigenvalues}. Let $v$ be the number of vertices in the configuration (that is, the matrices are $v \times v$). Since $\cA$ spans the configuration, there are constants \named{$q_{ij}$}{dual eigenvalue} such that
\[
E_j = \frac{1}{v} \sum_i q_{ji}A_i.
\]
The constants $q_{ji}$ are called the \defn{dual eigenvalues}. Note that in a homogeneous configuration,
\[
\tr(E_j) = \frac{1}{v} \sum_i q_{ji} \tr(A_i) = q_{j0}.
\]
Since $E_j$ is a projection matrix, $q_{j0}$ also equals $\rk(E_j)$. 
\nomenclature[a$r]{$\rk$}{rank}%

\begin{comment}
A projection matrix satisfies $E^2-E = 0$, therefore using the minimal polynomial we see that all eigenvalues are $0$ or $1$. Diagonalizing with eigenvalues $\la_i$, we get
\[
\tr(E) = \sum_i \la_i = \sum_{\la_i = 1} 1 = \rk(E).
\]
\end{comment}

Define a \defn{matrix of eigenvalues} $P$ and a \defn{matrix of dual eigenvalues} $Q$ as follows:
%\nomenclature{$P$}{matrix of eigenvalues}%
%\nomenclature{$Q$}{matrix of dual eigenvalues}%
\[
P_{ij} = p_{ji}, \quad Q_{ij} = q_{ji}.
\]
Also let $\De_n$ denote the diagonal matrix with entries $p_{ii}(0)$, and let $\De_m$ be diagonal with entries $q_{j0}$. 

\begin{lemma}\label{lem:pqvi}
If $\cA$ is a Hermitian configuration, then
\[
PQ = vI.
\]
If $\cA$ is also homogeneous, then
\[
P^T\De_m = \De_nQ.
\]
\end{lemma}

\begin{proof}
The matrix $P$ is the change of basis matrix from $A_i$ to $E_j$, and up to a constant $v$, its inverse is $Q$. For the second equation, taking the trace of $A_iE_j = p_{ij}E_j$, we get
\[
\tr(A_iE_j) = p_{ij}\tr(E_j) = p_{ij}q_{j0}.
\]
Now writing $E_j$ as $\frac{1}{v}\sum_k q_{jk}A_k$, 
\[
\tr(A_i E_j) = \frac{1}{v} \sum_k q_{jk} \tr(A_i A_k).
\]
But $A_i$ and $A_k$ are orthogonal for $i \neq k$, so this simplifies to
\[
p_{ij}q_{j0} = \frac{1}{v} \sum_k q_{ji} \tr(A_i^2) = q_{ji} p_{ii}(0).
\]
Entry-wise, this is the second matrix equation. \qed
\end{proof}

In fact, Lemma \ref{lem:pqvi} can be extended to the non-homogeneous case without too much difficulty. If $\cA$ is Hermitian but not homogeneous, then $\cA$ decomposes into a direct sum of homogeneous subalgebras.

\begin{corollary}\label{cor:interms}
The intersection numbers and dual eigenvalues of a Hermitian coherently-weighted configuration can be written in terms of the eigenvalues.
\end{corollary}

\begin{proof}
From the previous lemma it is clear that the the dual eigenvalues can be written in terms of the eigenvalues. For the intersection numbers, begin with
\[
A_iA_j = \sum_k p_{ij}(k)A_k.
\]
Multiplying both sides by $I = \sum_l E_l$, we get
\begin{align*}
\sum_k p_{ij}(k)A_k & = \sum_l A_iA_j E_l \\
& = \sum_l p_{il}p_{jl} E_l. 
\end{align*}
Now writing $E_l$ as $\frac{1}{v}\sum_k q_{lk}A_k$, we get
\[
\sum_k p_{ij}(k)A_k = \frac{1}{v}\sum_{l,k} p_{il}p_{jl}q_{lk}A_k. 
\]
However, the $A_i$'s are linearly independent. Therefore,
\[
p_{ij}(k) = \frac{1}{v}\sum_{l} p_{il}p_{jl}q_{lk}. \qed
\]
\end{proof}

There is a standard matrix inner product for configurations:
\[
\ip{A}{B} := \tr(A^*B) = \msum(\bar{A} \circ B),
\]
\nomenclature[\]{$\ip{A}{B}$}{matrix inner product}%
where $\msum(A)$ is the sum of all the entries of $A$. Both the matrices of $\cA$ and the projection matrices are orthogonal with respect to this inner product.

\begin{comment}
From the second definition,
\[
\ip{A_i}{A_j} = \de_{ij}vp_{ii}(0),
\]
where $A_i$ is a $v \times v$ matrix. From the first definition,
\[
\ip{E_i}{E_j} = \de_{ij} \rk(E_i).
\]
\end{comment}

\begin{comment}
\begin{theorem}
If $\cA$ is a Hermitian homogeneous coherently-weighted configuration, then for any matrix $M$,
\[
\sum_i \frac{\ip{A_i}{M}}{vp_{ii}(0)}A_i = \sum_j \frac{\ip{E_j}{M}}{q_{j0}}E_j.
\]
\end{theorem}

\begin{proof}
Since both the idempotents and the Schur idempotents are orthogonal with respect to the matrix inner product, the projection of $M$ onto the span of $\cA$ is
\[
\sum_i \frac{\ip{A_i}{M}}{\ip{A_i}{A_i}}A_i = \sum_j \frac{\ip{E_j}{M}}{\ip{E_j}{E_j}}E_j. \qed
\]
\end{proof}
\end{comment}

\section{Example: Seidel matrices}\label{sec:Seidel}

A \defn{Seidel matrix} is a symmetric matrix with $0$ on the diagonal and off-diagonal entries of $\pm 1$. Each Seidel matrix $S$ can be considered a type of adjacency matrix for a graph, where $u$ and $v$ are adjacent if and only if $S_{uv} = -1$. If $A$ is the standard adjacency matrix of the graph, then
\[
S = J-I-2A.
\]

Suppose $S$ is a Seidel matrix with only two eigenvalues. Then the minimal polynomial of $S$ has degree $2$, and so
\[
S^2 = aS + bI
\]
for some constants $a$ and $b$. Moreover, $S$ is a weighted adjacency matrix and Hermitian, so it follows that $\cA = \{I,S\}$ is a Hermitian homogeneous configuration. 

Seidel matrices with two eigenvalues come from sets of real equiangular lines of maximal size. Let $X = \{v_1,\ldots,v_n\}$ be unit vectors in $\re^d$ such that 
\[
v_i^Tv_j = \pm \al
\]
for some constant $\al$ and all $i \neq j$. Such lines are \textsl{equiangular} with angle $\al^2$. Then the Gram matrix of $X$ has the form $G = I + \al S$, where $S$ is a Seidel matrix. We call $S$ the Seidel matrix corresponding to $X$.

\begin{lemma}\label{lem:realealrelbnd}
Let $X$ be a set of $n$ equiangular lines in $\re^d$ with angle $\al^2$. Then
\[
|X| \leq \frac{d(1-\al^2)}{1-d\al^2}.
\]
Equality holds if and only if
\[
\sum_{i=1}^n v_iv_i^T = \frac{n}{d}I. \qed
\]
\end{lemma}

Lemma \ref{lem:realealrelbnd} is called the \textsl{relative bound}, and in Chapter 3 we prove an analogous result for complex lines. For a more direct proof, see Godsil and Royle \cite[Lemma 11.3.2]{god2}. 

\begin{corollary}
If the relative bound holds with equality, then the Seidel matrix of $X$ has two eigenvalues. Conversely, any Seidel matrix with two eigenvalues corresponds to a set of equiangular lines satisfying the relative bound with equality.
\end{corollary}

\begin{proof}
Let $M$ be the $d \times n$ matrix with columns $\{v_1,\ldots,v_n\}$. Then 
\[
M^TM = G = I + \al S,
\]
and 
\[
MM^T = \sum_{i=1}^n v_iv_i^T = \frac{n}{d}I.
\]
Now $M^TM$ and $MM^T$ have the same nonzero eigenvalues and multiplicities. Letting $\la^{(m)}$ denote an eigenvalue $\la$ with multiplicity $m$, it follows that the spectrum of $S$ is 
\[
\La = \left\{ -\frac{1}{\al}^{(n-d)}, \frac{n-d}{d\al}^{(d)} \right\}.
\]

For the converse, suppose $S$ has order $n$ and eigenvalues $\la_1^{(m_1)}$ and $\la_2^{(m_2)}$. Without loss of generality, $\la_1 < 0$. If we let $\la_1 = -1/\al$ and $m_1 = n-d$, then $m_2 = d$ and the spectrum of $S$ is $\La$. Therefore $G = I + \al S$ is positive semidefinite with rank $d$, so it is the Gram matrix of a set of lines $\{v_1,\ldots,v_n\}$ in $\re^d$. Those lines are equiangular because all off-diagonal entries of $G$ have the same absolute value. Again using the fact that $M^TM$ and $MM^T$ have the same nonzero eigenvalues, we see that $MM^T$ has exactly one eigenvalue, $n/d$. Thus 
\[
MM^T = \sum_{i=1}^n v_iv_i^T = \frac{n}{d}I,
\]
which implies that Lemma \ref{lem:realealrelbnd} is satisfied with equality. \qed
\end{proof}

\begin{comment}
Same eigenvalues and multiplicities: $M^TMx = \la x$ implies $MM^T(Mx) = \la Mx$. $G$ positive semidefinite of rank $d$ implies $G = M^TM$ with $M$ $d \times n$: this is the characterization of positive semidefinite, for example by $G = Q^TDQ$ with $Q$ orthogonal and $M = D^{1/2}Q$.
\end{comment}

Since multiplying the unit vector $v_i$ by $-1$ will not affect its angle with any other vector, two Seidel matrices are considered equivalent if one can be obtained from the other by multiplying row $i$ and column $i$ by $-1$. The corresponding operation on the graph, which consists of replacing the neighbourhood of a vertex by its complement, is called \defn{switching}. An equivalence class of graphs under this operation is called a \defn{switching class} or \defn{two-graph}. If $\Ga$ and $\Ga'$ are graphs from the same switching class, then their Seidel matrices have the same eigenvalues. A two-graph whose Seidel matrix has only two eigenvalues is called a \textsl{regular two-graph}\index{two-graph!regular}. In this way, maximal sets of real equiangular lines are characterized graph-theoretically. For more details about real equiangular lines, see Seidel \cite{sei1} or Godsil and Royle \cite{god2}.

\section{Example: monomial groups}

A matrix is \textsl{monomial}\index{monomial matrix} if exactly one entry is non-zero in every row and column. Every monomial matrix is of the form
\[
M = DP,
\]
where $D$ is diagonal, and $P$ is a permutation matrix. We call $P$ the \defn{underlying permutation} of $M$. Suppose $G$ is a group of monomial matrices, with $D_1P_1$ and $D_2P_2$ in $G$. Then
\[
D_1P_1D_2P_2 = D_1(P_1D_2P_1^{-1})P_1P_2
\]
is also in $G$. Since $P_1D_2P_1^{-1}$ is diagonal, so is $D_1(P_1D_2P_1^{-1})$. Thus $P_1P_2$ is the underlying permutation, and it follows that the underlying permutations of $G$ also form a group.

The \defn{centralizer} of a matrix group $G$ is the set of matrices
\[
C(G) = \{M :\; Mg = gM \text{ for all } g \in G\}.
\]
\nomenclature{$C(G)$}{centralizer}%
The centralizer is a matrix algebra in that it is closed under addition and multiplication and contains $I$.

\begin{lemma} \label{lem:monmat}
Let $M$ be a matrix and $G$ a monomial group. Then $M$ is in $C(G)$ if and only if for each $DP$ in $G$ with $D$ diagonal and $P$ the matrix of permutation $\pi$,
\[
M_{\pi(x),\pi(y)} = \frac{D_{yy}}{D_{xx}}M_{x,y} \quad \text{(all $x,y$).}
\]
\end{lemma}

\begin{proof}
Each $DP$ commutes with $M$ if and only if the entries
\[
(DPM)_{x,\pi(y)} = D_{x,x}M_{\pi(x),\pi(y)}
\]
and
\[
(MDP)_{x,\pi(y)} = M_{x,y}D_{y,y}
\]
are equal for every $x$ and $y$. \qed
\end{proof}

Higman \cite{hig1} showed how to construct a homogeneous configuration from the centralizer of a monomial group using induced representations. Let $G$ act transitively on a set $X$, and set $H = G_a$ for some fixed $a \in X$. Also fix $R$ as a set of coset representatives for $H$, so that every $g \in G$ can be written uniquely in the form 
\[
g = rh, \quad r \in R, h \in H.
\]
Now associate each $x \in X$ with the unique coset representative $r_x \in R$ such that
\[
x = r_x(a).
\]
Then the action of $G$ on $X$ can be described as follows: if $gr_x = r_yh$ for some $r_y \in R$ and $h \in H$, then
\[
g(x) = gr_x(a) = r_yh(a) = r_y(a) = y.
\]
Finally, let $\la$ be a fixed linear character of $H$. For each $g \in G$, we define an $|X| \times |X|$ matrix $M(g)$ by its action on the standard basis $\{e_x : x \in X\}$. If $gr_x = r_yh$, then
\[
M(g): e_x \mapsto \la(h)e_y.
\]
It is not difficult to verify that $\{M(g): g \in G\}$ is a representation of $G$. Clearly each $M(g)$ is also monomial, so we have a monomial group. 

\begin{comment}
A linear character is a character of degree $1$: $\la(g) = \tr(M'(g))$, where matrix $M'(g)$ is $1 \times 1$.

Let $g_1$ and $g_2$ be elements of $G$. For a given $x \in X$, there is a unique $y$ in $X$ and $h_y$ in $H$ such that $g_1r_x = r_yh_y$. Similarly, $g_2r_y = r_zh_z$ for some $z$ in $X$ and $h_z$ in $H$. Then,
\[
M(g_2)M(g_1)e_x = M(g_1)\la(h_y)e_y = \la(h_y)\la(h_z)e_z = \la(h_zh_y)e_z.
\]
However, since $g_2g_1r_x = g_2r_yh_y = r_zh_zh_y$, we also have
\[
M(g_2g_1)e_x = \la(h_zh_y)e_z.
\]
\end{comment}

\begin{theorem}\label{thm:higmon}
The centralizer of $\{M(g): g \in G\}$ has a basis which is a homogeneous coherently-weighted configuration.
\end{theorem}

\begin{comment}
In fact, since non-homogeneous Hermitian algebras decompose into homogeneous subalgebras, the assumption of transitivity on the group $G$ may be removed.

\begin{proof}
This proof only works for the case below, namely when $H$ is a normal subgroup. Let $M_G$ denote the matrix group $\{M(g): g \in G\}$. Since $C(M_G)$ is an algebra, it suffices to show that it has a basis of Schur-orthogonal matrices, one element of which is $I$. Assume the rows and columns of the matrices of $G$ are indexed by a set $X$. Fix $a \in X$, and let $S$ be the elements $s \in X$ such that $M_{a,s} \neq 0$ for at least one $M$ in $C(M_G)$. Then for each $s$ in $S$, define a matrix $M_s$ with the following $a$-th row:
\[
(M_s)_{a,y} = \begin{cases}
1, & y = s; \\
0, & y \neq s.
\end{cases}
\]
Define the other rows of $M_s$ in accordance with Lemma \ref{lem:monmat}. Since the underlying group of permutations of $G$ acts regularly on $X$, this defines all rows of $M_s$ uniquely. Each $M_s$ is flat monomial, since the matrices of $G$ are flat monomial. The set $\{M_s: s \in S\}$ is Schur orthogonal and spans $C(M_G)$. Finally, since $I$ is in $C(M_G)$, it follows that $a$ is in $S$ and therefore $M_a = I$. \qed 
\end{proof} 

In general, the matrices in Theorem \ref{thm:higmon} will not even be normal (that is, satisfy $MM^* = M^*M$). 
\end{comment}

In general, the configuration in Theorem \ref{thm:higmon} will not be Hermitian. A monomial matrix is \defn{flat} if all its non-zero entries have the same absolute value; each $M(g)$ is flat. Also, in the case when $H$ is a normal subgroup, the quotient group $G/H$ acts regularly on $X$. In fact, if $G$ is any group of flat monomial matrices such that the underlying group of permutations is regular, then the centralizer of $G$ has a basis which is a homogeneous configuration.

\section{Association schemes}

\begin{comment}
Associations schemes seem to have been introduced formally by Bose and Shimamoto \cite{bos1} in 1952, although their ideas go back to at least 1939.
\end{comment}

A \textsl{symmetric association scheme}\index{association scheme} is a Hermitian homogeneous configuration $\cA$ such that every $A_i \in \cA$ is $0$-$1$, and 
\[
\sum_{i=0}^d A_i = J.
\]
\nomenclature{$J$}{all-ones matrix}%
Since $A_i$ is $0$-$1$, it is a \defn{Schur idempotent}:
\[
A_i \circ A_i = A_i.
\]
This implies that the span of $\cA$ is closed with respect to Schur multiplication, which is not true of coherent configurations in general. The weighted adjacency algebra of an association scheme is called a \defn{Bose-Mesner algebra}.

\begin{lemma} 
If $\cA$ is an association scheme, then 
\[
E_0 = \frac{1}{v}J
\]
is an idempotent of the scheme, and the corresponding eigenvalue for $A_i$ is the valency $p_{ii}(0)$.
\end{lemma}

\begin{proof}
Since $A_i$ is Hermitian and $0$-$1$, it is symmetric and therefore the adjacency matrix of a graph $G_i$. Since $(A_i)^2_{a,a} = p_{ii}(0)$, this graph is regular with valency $p_{ii}(0)$. Therefore $\one$, the all-ones vector, is an  eigenvector with eigenvalue $p_{ii}(0)$. Denote the idempotent matrix for this eigenspace by $E_0$.

In any connected regular graph, the valency is an eigenvalue of multiplicity $1$. Therefore if $E_0$ has rank $k$, each $G_i$ has at least $k$ components. More specifically, there is a $k$-cell partition $\pi$ of the vertex set such that the partition of components of $G_i$ is a refinement of $\pi$. But $\sum A_i = J$, so every pair of vertices is an edge in some $G_i$ and therefore $k=1$. Thus $E_0$ is the projection onto the space spanned by $\one$. \qed
\end{proof}

Since $E_i \circ E_j$ is in the span of $\cA$, there are also constants \named{$q_{ij}(k)$}{Krein parameter} such that
\[
E_i \circ E_j = \frac{1}{v}\sum_k q_{ij}(k) E_k.
\]
These constants are called the \defn{Krein parameters}. The proof of the following is similar to Corollary \ref{cor:interms}.

\begin{corollary}
The intersection numbers, Krein parameters, and dual eigenvalues of a scheme can all be written in terms of the eigenvalues. \qed
\end{corollary}

Suppose $\cA$ is an association scheme and $a$ is adjacent to $b$ in $G_k$, the graph corresponding to $A_k$. Then the $(a,b)$ entry of $A_iA_j$ is the number of vertices $c$ adjacent to $a$ in $G_i$ and adjacent to $b$ in $G_j$. It follows that the intersection number $p_{ij}(k)$ is a nonnegative integer. (Again this is not true of configurations in general.) The next theorem is slightly more difficult, but it is an important condition for proving that a scheme with a given set of parameters does not exist.

\begin{theorem}
The Krein parameters of a scheme are nonnegative. \qed
\end{theorem}

\begin{proof}
The parameter $q_{ij}(k)$ is the eigenvalue of $E_i \circ E_j$ for eigenspace $E_k$. Now $E_i \otimes E_j$ is a projection matrix, so it is positive semidefinite. But $E_i \circ E_j$ is a principal submatrix, so it is also positive semidefinite and therefore $q_{ij}(k) \geq 0$. \qed
\end{proof}

\section{Distance-regular graphs}

Let $d(a,b)$ denote the distance between two vertices $a$ and $b$ in a graph $G$, and let $\Ga_i(a)$ denote the \textsl{$i$-th neighbourhood} of $a$: the set of vertices at distance $i$ from $a$. Then $G$ is \defn{distance-regular} if, for every $a$ and $b$, the size of $\Ga_i(a) \cap \Ga_j(b)$ depends only on $i$, $j$, and $d(a,b)$.
\nomenclature[\c]{$\Ga_i(a)$}{$i$-th neighbourhood}%

Let $G_i$ denote the \defn{distance-$i$ graph} of $G$: $a$ and $b$ are adjacent in $G_i$ if they are at distance $i$ in $G$. Then $G = G_1$. Also let $A_i$ be the adjacency matrix of $G_i$, with $A_0 = I$. Then each $A_i$ is a symmetric $0$-$1$ matrix, and if $d$ is the diameter of $G$ then
\[
\sum_{i=0}^d A_i = J.
\]
If $G$ is distance-regular, then the $(a,b)$-entry of $A_iA_j$ depends only on the distance between $a$ and $b$. Therefore there are constants $p_{ij}(k)$ such that
\[
A_iA_j = \sum_k p_{ij}(k) A_k,
\]
and so $\{A_0,\ldots,A_d\}$ is an association scheme. 

Suppose $(A_1A_i)_{a,b}$ is nonzero. Then there is a vertex $c$ at distance $1$ from $a$ and distance $i$ from $b$, and so $a$ and $b$ must be at distance $i-1$, $i$ or $i+1$. Therefore if $G$ is distance-regular, the intersection numbers of the scheme can be simplified: there are constants $a_i$, $b_{i-1}$, and $c_{i+1}$ such that 
\begin{equation}\label{eqn:intarray}
A_1A_i = b_{i-1}A_{i-1} + a_iA_i + c_{i+1}A_{i+1}.
\end{equation}
If $a$ and $b$ are at distance $i$, then $b_i$ is the number of vertices $c$ at distance $1$ from $a$ and $i+1$ from $b$. Similarly $c_i$ is the number at distance $1$ from $a$ and $i-1$ from $b$. There is some redundancy here: the number of neighbours of $a$ is $a_i+b_i+c_i = k$, the valency of the graph. Also, $c_0 = 1$ and $b_d = 0$. The intersection numbers are often encapsulated in an \defn{intersection array}:
\[
\{b_0,b_1,\ldots,b_{d-1}; c_1,c_2,\ldots,c_d\}.
\]
\nomenclature[a$b]{$\{b_0,\ldots,b_{d-1}; c_1,\ldots,c_d\}$}{intersection array}%

When $i=2$ in \eqref{eqn:intarray}, we get
\[
c_2A_2 = A_1^2 - a_1A_1 - b_0I,
\]
which is a quadratic polynomial in $A_1$. More generally, induction shows that $A_i$ is a polynomial of degree $i$. A configuration in which each $A_i$ is a polynomial of degree $i$ in $A_1$ is called \defn{$P$-polynomial}.

\begin{theorem} 
An association scheme is $P$-polynomial if and only if its Schur idempotents are the distance matrices of a distance-regular graph. \qed
\end{theorem}

\begin{comment}
Let $A_i = p_i(A_1)$, and define an inner product on polynomials by $\ip{p}{q} = \ip{p(A_1)}{q(A_1)}$. Since $A_i$ and $A_j$ are orthogonal, so are $p_i$ and $p_j$. Thus the orthogonal sequence $\{p_i\}$ forms a three-term recurrence: $xp_i = c_{i+1}p_{i+1}+a_ip_i+b_{i-1}p_{i-1}$, which is equation \eqref{eqn:intarray}. This tells us that $A_2$ is the distance-$2$ matrix of $A_1$, and so on.
\end{comment}

Dually, a configuration is \defn{$Q$-polynomial} if each idempotent $E_j$ is Schur polynomial of degree $j$ in $E_1$. In other words, for some polynomial $q$ of degree $j$,
\[
(E_j)_{a,b} = q((E_1)_{a,b}).
\]
For more about association schemes and distance-regular graphs, see Brouwer, Cohen and Neumaier \cite{bcn} or Godsil \cite{blue}.

\section{Example: distance-regular covers}

Let $G$ be a distance-regular graph with diameter $d$. Then $G$ is \defn{antipodal} if any two vertices at distance $d$ from a given $x$ are also at distance $d$ from each other. Equivalently, there is a partition $\pi$ of the vertices such that $x$ and $y$ are in the same cell if and only if they are at maximum distance. The cells of $\pi$ are called \defn{fibres}.

Given a graph $G$ with antipodal partition $\pi$, the \defn{quotient graph} \named{$G/\pi$}{quotient graph} has the fibres of $\pi$ as vertices, with $\pi_i$ and $\pi_j$ adjacent if there are vertices in $\pi_i$ and $\pi_j$ that are adjacent in $G$. Assume $d \geq 3$; then from distance-regularity it follows that all the fibres have the same size, and if two fibres are adjacent in $G/\pi$, then there is a matching between them in $G$. If every fibre has size $n$, we call $G$ an \defn{$n$-fold cover} of $G/\pi$.

\begin{comment}
If $x$ and $y$ are in the same cell, then there are $p_{dd}(d)$ other vertices also in that cell. Therefore all cells have the same size. If $x$ and $z$ are adjacent, then each $w$ in the same cell as $z$ must be at distance $d-1$ from $x$ and is therefore adjacent to a unique $y$ in the same cell as $x$. This gives the matching.

\begin{lemma}
$G$ is bipartite if and only if $a_i = 0$ for all $i$. $G$ is antipodal if and only if $b_i = c_{d-i}$ for all $i \neq \lf d/2 \rf$.
\end{lemma}

\begin{proof}
Suppose $a_i \neq 0$. Then some $y$ and $z$ at distance $i$ from $x$ are adjacent, and we have a cycle of length $2i+1$. The converse is also straightforward. \qed
\end{proof}
\end{comment}

\begin{theorem}
Let $G$ be antipodal and distance-regular with intersection array 
\[
\{b_0,\ldots,b_{d-1};c_1,\ldots,c_d\}.
\]
Then $G/\pi$ has diameter $m = \lf d/2 \rf$. If $d = 2m$, and $G/\pi$ is an $n$-fold cover, then $G/\pi$ has intersection array
\[
\{b_0,\ldots,b_{m-1};c_1,\ldots,nc_m\}.
\]
If $d = 2m+1$, then $G/\pi$ has intersection array
\[
\{b_0,\ldots,b_{m-1};c_1,\ldots,c_m\}. \qed
\]
\end{theorem}

For a full proof, consult Brouwer et al.~\cite{bcn} or Gardiner \cite{gar1}. We prove a more specific case.

\begin{theorem}
If $G$ is an antipodal distance-regular cover of the complete bipartite graph $K_{k,k}$, then the intersection array of $G$ is
\begin{equation} \label{eqn:dracknnarray}
\{k,k-1,k-\la,1;1,\la,k-1,k\},
\end{equation}
where $\la$ divides $k$.
\end{theorem}

\begin{proof}
Suppose $G$ is a distance-regular $n$-fold antipodal cover of $K_{k,k}$. For any quotient graph, the natural mapping from $G$ to $G/\pi$ is a homomorphism; therefore since $K_{k,k}$ is bipartite, so is $G$. 

Now fix a vertex $x$, and suppose $G$ has diameter at least $6$. Then there are vertices at distance $3$ from every vertex in the cell of $x$, and so the quotient graph has diameter $3$. But $K_{k,k}$ has diameter $2$, so by contradiction, $G$ has diameter at most $5$. Now suppose $G$ has distance $5$. Let $y$ and $z$ be adjacent vertices at distance $2$ and $3$ from $x$ respectively. Then the fibres of $y$ and $z$ are adjacent and both at distance $2$ from the fibre of $x$, giving an odd cycle in the quotient graph. By contradiction, $G$ must have diameter $4$.

We can now build up the intersection array $\{b_0,b_1,b_2,b_3;c_1,c_2,c_3,c_4\}$. A distance-regular graph is bipartite if and only if $a_i = 0$ for all $i$, and since $a_i + b_i + c_i = k$, we know that $c_i = k-b_i$. Clearly $b_0 = k$, the valency of the graph, and since $c_1 = 1$, we know $b_1 = k-1$. Similarly, antipodality implies that $b_3 = 1$ and therefore $c_3 = k-1$.  Letting $c_2 = \la$, we get $b_2 = k-\la$, and so the intersection array has the form of \eqref{eqn:dracknnarray}. 

To see that $\la$ divides $k$, use the intersection array to count the number of vertices at each distance from $x$. The number at distance $4$ is $(k-\la)/\la$, which must be an integer. \qed
\end{proof}

It can be shown that if $G$ is an antipodal distance-regular cover, then every eigenvalue of $G/\pi$ is an eigenvalue of $G$ with the same multiplicity. Thus if $G$ is an $n$-fold cover of $K_{k,k}$, then $0$ and $\pm k$ are eigenvalues. However we can obtain all eigenvalues of a distance-regular graph from its intersection array: for an $n$-fold cover of $K_{k,k}$ they are $0$, $\pm k$ and $\pm \sqrt{k}$, with multiplicities $2(k-1)$, $1$, and $k(n-1)$ respectively.

\chapter{Bounds}\label{chap:bounds}

The goal of this chapter is to find upper bounds on the size of $s$-distance sets. Most significantly, in $\cx^d$ there can be at most $d^2$ equiangular lines and at most $d+1$ mutually unbiased bases. When equality holds, the lines can be characterized in terms of $t$-designs.

Most of the bounds in this chapter were first discovered by Delsarte, Goethals, and Seidel \cite{dgs} in 1974. Their approach relied heavily on the ``addition formula'' for harmonic polynomials due to Koornwinder \cite{koo1}; instead, we use zonal polynomials to obtain the same results, as well as some new ones.

Zonal polynomials can be described in the context of both polynomial spaces and Delsarte spaces. Polynomial spaces were introduced by Godsil \cite{god8} as a common framework for deriving results about block designs (due to Ray-Chaudhuri and Wilson \cite{rcw}) and real spherical designs (due to Delsarte, Goethals, and Seidel \cite{dgs2}). Delsarte spaces were formalized by Neumaier \cite{neu1} based on the work in Delsarte's thesis \cite{del1}, which covers block designs as well as bounds on error-correcting codes. 

Both polynomials spaces and Delsarte spaces provide a general framework; we provide the details on how they apply to complex lines. To do this, we rely on a treatment of harmonic polynomials due to Vilenkin and \v{S}apiro \cite{vs1}. As a result of this chapter, we obtain all of the results of Delsarte, Geothals, and Seidel without any difficult complex analysis; linear algebra is the major tool involved. The existence of certain weighted adjacency algebras (also found by Delsarte et al.) falls out of the analysis. 

We also get some new results about complex $t$-designs, which are a generalization of block designs to complex vector spaces. Neumaier characterized maximal $s$-distance sets as minimal $t$-designs in any Delsarte space. Our main result is a characterization of the same form but using a slightly different bound, one which is more appropriate when $0$ is one of the angles.

\begin{comment}
A \defn{Delsarte space} is a metric space $\Om$ with $c_{xy} = d(x,y)^2$ such that there is a polynomial $f_{ij}(t)$ of degree no larger than $i$ or $j$ satisfying
\[
\int_{\Om} c_{az}^i c_{bz}^j \om(z) = f_{ij}(c_{ab}).
\]
For our purposes, $c_{xy} = 1 - \abs{x^*y}^2$. (In the real case $1-x^Ty$ is the standard Euclidean metric, but not here.) It takes some work to show that complex projective space is a Delsarte space.

The \defn{Polynomial spaces} of Godsil on the other hand are defined as follows: $\Pol(\Om,1)$ is the set of zonal polynomials $\{f_a: a \in \Om, \deg(f) = 1\}$, and $\Pol(\Om,r)$ is defined recursively as the span of $\Pol(\Om,1)\Pol(\Om,r-1)$. 

Both Delsarte and Polynomial spaces are technically only defined for real inner product spaces.
\end{comment}

\section{Harmonic polynomials}\label{sec:harm}

Informally, a function $f$ is harmonic if it satisfies the Laplacian equation $\De f = 0$. In this section, we consider harmonic polynomials $f(z): \cx^d \rightarrow \cx$ which are homogeneous in both $z$ and $\bz$.

Let $\Hom(k,l)$ denote the polynomials $f: \cx^{d} \rightarrow \cx$ of the form
\nomenclature[a$H]{$\Hom(k,l)$}{homogeneous polynomials}
\[
f(z) = f(z_1,\ldots,z_d;\bz_1,\ldots,\bz_d),
\]
where $f$ is homogeneous of degree $k$ in $\{z_i\}$ and homogeneous of degree $l$ in $\{\bz_i\}$. In this context, the Laplacian is 
\[
\De := \frac{\dde^2}{\dde z_1 \dde \bz_1} + \ldots + \frac{\dde^2}{\dde z_d \dde \bz_d}.
\]
\nomenclature[\d]{$\De$}{Laplacian}%
For the purposes of partial differentiation, the variables $z_i$ and $\bz_i$ are considered independent. The Laplacian operator commutes with unitary transformations: if $U$ is a unitary mapping on $\cx^d$, then for any $f: \cx^d \rightarrow \cx$,
\[
\De (f \circ U) = (\De f) \circ U.
\]
Let $\nabla$ denote the gradient with respect to $z_1,\ldots,z_d$, namely
\[
\nabla f := \Big(\frac{\dde f}{\dde z_1}, \ldots,\frac{\dde f}{\dde z_d}\Big)^T,
\]
%\nomenclature{$\nabla$}{gradient}%
and $\conj{\nabla}$ the gradient with respect to $\bz_1,\ldots,\bz_d$. Then the following product rule for the Laplacian is easy to verify. 

\begin{lemma}[Product Rule]
\[
\De (fg) = f \De g + g \De f + \nabla f \cdot \conj{\nabla} g + \nabla g \cdot \conj{\nabla} f. \qed
\]
\end{lemma}

\begin{comment}
\begin{proof}
\begin{align*}
\frac{\dde^2 fg}{\dde z_1 \dde \bz_1} & = \frac{\dde}{\dde z_1}(f \frac{\dde g}{\dde \bz_1} + g\frac{\dde f}{\dde \bz_1}) \\
& = \frac{\dde}{\dde z_1}f \frac{\dde}{\dde \bz_1}g + f \frac{\dde^2 g}{\dde z_1 \dde \bz_1} + \frac{\dde}{\dde z_1}g \frac{\dde}{\dde \bz_1}f + g \frac{\dde^2 f}{\dde z_1 \dde \bz_1}
\end{align*}
and similarly for every other $z_i$ and $\bz_i$. Summing over all $i$, the result follows. \qed
\end{proof}
\end{comment}

Define the \defn{harmonic polynomials} $\Harm(k,l)$ as the kernel of $\De$ in $\Hom(k,l)$. Note that $\Harm(k,l)$ is a complex vector space. 
\nomenclature[a$H]{$\Harm(k,l)$}{harmonic polynomials}%
Let
\[
Z := z_1\bz_1 + \ldots z_d \bz_d,
\]
\nomenclature{$Z$}{squared norm}%
and let $[A,B]$ be the commutator
\nomenclature[\]{$[A,B]$}{commutator}%
\[
[A,B] := AB - BA.
\]

\begin{lemma} \label{lem:dez}
If $f$ is in $\Hom(k,l)$, then
\[
[\De,Z] f = (d+k+l) f.
\]
If $f$ is also harmonic, then
\[
\De^r Z^r f = \frac{(d+k+l+r-1)!r!}{(d+k+l-1)!}f.
\]
\end{lemma}

\begin{proof}
Note that $\De Z = d$. Then,
\begin{align*}
[\De,Z] f & = \De Zf - Z\De f \\
& = (Z \De f + f \De Z + \nabla f \cdot \conj{\nabla} Z + \conj{\nabla} f \cdot \nabla Z) - Z\De f \\
& = (d+k+l)f.
\end{align*}
A little work shows that $\De Z^r = r(d+r-1)Z^{r-1}$. Then when $f$ is harmonic,
\begin{align*}
\De^r Z^r f & = \De^{r-1}(\De Z^r f) \\
& = \De^{r-1}(f \De Z^r + Z^r \De f + \nabla f \cdot \conj{\nabla} Z^r + \conj{\nabla} f \cdot \nabla Z^r) \\ 
& = \De^{r-1}(f \De Z^r + r(k+l)Z^{r-1}f) \\
& = r(d+r+k+l-1)\De^{r-1}Z^{r-1}f.
\end{align*} 
The result follows by induction. \qed
\end{proof}

\begin{comment}
Proof of $\De Z^r$:
\[
\frac{\dde Z^r}{\dde z_i} = rZ^{r-1}\frac{\dde Z}{\dde z_i} = rZ^{r-1}\bz_i,
\]
and so
\[
\frac{\dde^2 Z^r}{\dde z_i\dde \bz_i} = \frac{\dde rZ^{r-1}\bz_i}{\dde \bz_i} = r((r-1)Z^{r-2}z_i\bz_i + Z^{r-1}).
\]
Therefore $\De Z^r = r((r-1)+d)Z^{r-1}$.
\end{comment}

\begin{corollary}
If $f \neq 0$, then $Zf$ is not harmonic.
\end{corollary}

\begin{proof}
From Lemma \ref{lem:dez} it is clear that if $f$ is harmonic then $Zf$ is not. Otherwise, let $q$ be the minimum such that $\De^qf = 0$. Taking $\De^{q-1}$ of the first equation in Lemma \ref{lem:dez}, it is a straightforward induction on $q$ to show that  
\[
\De^q Z f = c \De^{q-1} f,
\]
for some $c \neq 0$. This implies that $Zf$ is not harmonic, since $\De^{q-1}f$ is nonzero. \qed
\end{proof}

\begin{comment}
Let $f' = \De f$ and assume the result for $f'$. Taking $\De^{q-1}$ of the first equation in Lemma \ref{lem:dez},
\begin{align*}
(d+k+l)\De^{q-1}f & = \De^q Z f + \De^{q-1} Z f' \\
& = \De^q Z f + c' \De^{q-2} f' \\
& = \De^q Z f + c' \De^{q-1} f.
\end{align*}
Thus $\De^q Z f$ is a (nonzero) multiple of $\De^{q-1} f$.
\end{comment}

The following theorem is due to Vilenkin and \v{S}apiro \cite{vs1}. 

\begin{theorem} \label{thm:homdecomp}
\[
\Hom(k,l) = \Harm(k,l) \oplus Z \Hom(k-1,l-1).
\]
\end{theorem}
\nomenclature[\]{$V_1 \oplus V_2$}{direct product}%

\begin{proof}
The proof is by induction on the smallest value of $q$ such that $\De^q f = 0$, for $f \in \Hom(k,l)$.
Assume the decomposition holds for $f$ when $\De^q f = 0$, and consider $f$ such that $\De^{q+1} f = 0$. Then $\De^q f$ is harmonic, so by the previous lemma, 
\[
\De^q Z^q (\De^q f) = c\De^q f,
\]
where $c$ is the constant $(d+k+l-q-1)!q!/(d+k+l-2q-1)!$. Rearranging,
\[
\De^q [c - Z^q\De^q] f = 0.
\]
Since $(c - Z^q\De^q) f$ satisfies the induction hypothesis, there is some $g \in \Harm(k,l)$ and $h \in \Hom(k-1,l-1)$ such that
\[
(c - Z^q\De^q) f = g + Zh.
\]
Again rearranging,
\[
cf = g + Z(h+ Z^{q-1}\De^q f).
\]
We conclude that $cf$, and therefore $f$, can be decomposed appropriately. 

Next, we show the decomposition is unique. Suppose not; then subtracting two distinct decompositions, we see that $0 = g + Zh$ for some $g \in \Harm(k,l)$ and $h \in \Hom(k-1,l-1)$. Now applying the decomposition to $h$ and repeating, we get
\[
0 = g_0 + Zg_1 + \ldots + Z^sg_s,
\]
where each $g_i$ is harmonic and without loss of generality $g_s \neq 0$. Take $\De^s$ of both sides. Since $\De^iZ^ig_i$ is a nonzero multiple of $g_i$, say $c_ig_i$, we get
\[
0 = \De^sg_0 + \De^{s-1}c_1g_1 + \ldots + c_sg_s.
\]
But $\De g_i = 0$ for each $i$, so we conclude that $c_sg_s = 0$. By contradiction, the decomposition must be unique. \qed
\end{proof}

\begin{corollary} \label{cor:homdecomp}
Let $f \in \Hom(k,l)$, with $q = \min\{k,l\}$. Then
\[
f = f_0 + Zf_1 + \ldots + Z^qf_q,
\]
where $f_i$ is in $\Harm(k-i,l-i)$. \qed
\end{corollary}

From the proof of Theorem \ref{thm:homdecomp}, we get a formula for the orthogonal projection from $\Hom(k,l)$ onto its subspace $\Harm(k,l)$. 

\begin{corollary} \label{cor:harmproj}
Let $P$ denote the projection $\Hom(k,l) \rightarrow\Harm(k,l)$, and let $m = \min\{k,l\}$. Then
\[
P = \Big(1 - \frac{(d+k+l-2-1)!}{(d+k+l-1-1)!1!}Z\De\Big)\ldots\Big(1 - \frac{(d+k+l-2m-1)!}{(d+k+l-m-1)!m!}Z^m\De^m\Big).
\]
\end{corollary}

\begin{proof}
Consider $f \in \Hom(k,l)$ such that $\De^q f \neq 0$ but $\De^{q+1}f = 0$, and let $P_q$ denote the projection for this $f$. As noted in the theorem, $\De^q[c_q - Z^q\De^q]f = 0$,
where
\[
c_q = \frac{(d+k+l-q-1)!q!}{(d+k+l-2q-1)!}.
\]
Now let 
\[
\Big(1 - \frac{Z^q\De^q}{c_q}\Big)f = g + Zh
\]
with $g$ harmonic, so $P_{q-1}$ maps $(1 - Z^q\De^q/c_q)f$ to $g$. But then
\[
f = g + Z\Big(h + \frac{Z^{q-1}\De^qf}{c_q}\Big),
\]
so $P_q$ maps $f$ to $g$ also. Thus
\[
P_q = P_{q-1}\Big(1 - \frac{Z^q\De^q}{c_q}\Big).
\]
Note that if $\De^q f = 0$, then $P_{q+1}f = P_qf$. Therefore when $q = m = \min\{k,l\}$, $P_q$ applies to all of $\Hom(k,l)$ and so $P = P_m$. With the initial condition $P_0 = 1$, we get the formula above. \qed
\end{proof}
Vilenkin and \v{S}apiro reformulated this projection as
\[
P = \sum_{r=0}^{\min\{k,l\}} (-1)^r \frac{(d+k+l-r-2)!}{(d+k+l-2)!r!} Z^r \De^r.
\]

We can also use Theorem \ref{thm:homdecomp} to find the dimension of $\Harm(k,l)$. Since the number of monomials of degree $k$ in $d$ variables is ${d+k-1 \choose d-1}$, the dimension of $\Hom(k,l)$ is 
\[
\dim(\Hom(k,l)) = {d+k-1 \choose d-1}{d+l-1 \choose d-1}.
\]
\nomenclature[a$d]{$\dim$}{dimension}%
Then using the decomposition in Theorem \ref{thm:homdecomp},
\[
\dim(\Harm(k,l)) = {d+k-1 \choose d-1}{d+l-1 \choose d-1} - {d+k-2 \choose d-1}{d+l-2 \choose d-1}.
\]

\subsection{Inner product}

Define an inner product on complex functions as follows:
\[
\ip{f}{g} := \int_{\Om} \conj{f(z)}g(z) \; d\om(z).
\]
\nomenclature[\]{$\ip{f}{g}$}{inner product on functions}%
\nomenclature[\z]{$\Om$}{unit sphere}%
\nomenclature[\z]{$\om(z)$}{measure on the unit sphere}%
Here $\Om$ is the unit sphere in $\cx^d$, and $\om$ is the unique measure on $\Om$ which is invariant under unitary transformations and normalized so that
\[
\int_{\Om} \; d\om(z) = 1.
\]
This means that in addition to the usual properties of a complex inner product, for any unitary $U$ on $\cx^d$ we have
\[
\ip{f}{g} = \ip{f \circ U}{g \circ U}.
\]
In fact, this inner product is consistent with our previous direct sum of $\Hom(k,l)$ in Theorem \ref{thm:homdecomp}: the components of the direct sum are orthogonal. 

\begin{theorem} \label{thm:harmortho}
Let $f$ be in $\Harm(k,l)$ and let $g$ be in $\Hom(k-i,l-i)$ for some $i > 0$. Then
\[
\ip{f}{g} = 0. \inqed
\]
\end{theorem}

\begin{comment}
Add a \qed.

Using Corollary \ref{cor:homdecomp} and the fact that $Z = 1$ on $\Om$, it suffices to proof the case when $g \in \Harm(k-i,l-i)$. From Corollary \ref{cor:intdelt} below, it then suffices to show 
\[
\De^t(\conj{f}g) = 0,
\]
where $t = k+l-i$. One possibility is to prove that $f$ and $g$ are eigenfunctions with different eigenvalues for a self-adjoint operator. For example, prove that $\De Z$ is self-adjoint.

Alternatively, we could define $\Harm(k,l)$ as the space in $\Hom(k,l)$ orthogonal to $Z\Hom(k-1,l-1)$. But then we would not have explicit formulae for the Jacobi polynomials.
\end{comment}

See Rudin \cite[Chapter~12]{rud1} for a proof, or see Axler, Bourdon, and Ramey~\cite[Proposition~5.9]{abr} for the analogous result on the real sphere, which is standard in harmonic analysis. In particular, if $f$ is in $\Harm(k,k)$, then
\[
\int_{\Om} f(z) \; d\om(z) = \ip{1}{f} = 0.
\]

We now examine integration over the unit sphere in more detail. Consider a function of the form 
\[
f(z) = z_1g(z_2,\ldots,z_d;\bz_2,\ldots,\bz_d).
\]
For every point $a = (a_1,a_2,\ldots,a_d)$ on the unit sphere, there is a point $a' = (-a_1,a_2,\ldots,a_d)$ such that $f(a') = -f(a)$. By symmetry about zero, we conclude that $\ip{1}{f} = 0$. More generally, only monomials in $Z_i := z_i\bz_i$ can have nonzero inner product with $1$. In that case, integration is given by the following theorem. For a proof see Rudin \cite{rud1}, who attributes it to Bungart \cite{bun1}. 

\begin{theorem}
If
\[
f(z) = Z_1^{a_1}\ldots Z_d^{a_d},
\]
then
\[
\int_{\Om} f(z) \; d\om(z) = \frac{(d-1)!a_1!\ldots a_d!}{(d-1+a_1+\ldots+a_d)!}. \qed
\]
\end{theorem}

\begin{comment}
\begin{proof}
Instead of working with the unit ball, we work in the entire complex space $\cx^d$. As with $\Om$, there is a unique weight function $v(z)$ for integration, namely the Levesque measure. This measure is normalized by dividing by $\pi^d/d!$ so that 
\[
\int_{\cx^d} \; dv(z) = 1.
\]
The two forms of integration are related by polar coordinates: for any function $g$,
\[
\int_{\cx^d} g(z) \; dv(z) = 2d \int_0^\infty r^{2d-1} dr \int_\Om g(rz) d\om(z).
\]

Denote $Z_1^{a_1}\ldots Z_d^{a_d}$ by $Z^a$, where $a = (a_1,\ldots,a_d)$. Now consider the following integral:
\[
I = \int_{\cx^d} Z^a e^{-Z^{(1)}} \; dv(z).
\]
Here $(1) = (1,1,\ldots,1)$. The component of $I$ in the $i$-th coordinate is $Z_i^{a_i} e^{-Z_i}$, and this expression evaluates to $a_i!\pi$ when integrated over $\cx$. Hence
\[
I = a_1!\ldots a_d!\pi^d.
\]
Evaluating $I$ using polar coordinates, we also find that 
\[
I = \frac{2d \pi^d}{d!} \int_0^\infty r^{2(d+a_1+\ldots+a_d)-1}e^{-r^2} dr \int_\Om Z^\al \; d\om(z).
\]
Combining the two expressions for $I$ yields the desired result. \qed
\end{proof}
\end{comment}

\begin{corollary} \label{cor:intdelt}
For any $f \in \Hom(t,t)$,
\[
\int_{\Om} f(z) \; d\om(z) = \frac{(d-1)!}{t!(d-1+t)!} \De^t f .
\]
\end{corollary}

\begin{proof}
It is not difficult to verify that if $f$ is a monomial in $Z_i$, say $Z_1^{a_1}\ldots Z_d^{a_d}$, then 
\[
\De^t f = t!a_1!\ldots a_d!
\]
and if $f$ is a monomial in $\Hom(t,t)$ but not a monomial in $Z_i$ then $\De^t f$ is zero. \qed
\end{proof}

\begin{comment}
Since each term in $\De^tZ^a$ has the form $(a_1!\ldots a_d!)^2$, and there ${t \choose a_1,\ldots,a_d}$ terms, we get $\De^t f$ as written.
\end{comment}

\section{Zonal polynomials}\label{sec:zonal}

A \textsl{zonal} function on a set $\Om$ is a function $Z_a$, for $a \in \Om$, such that the value of $Z_a(z)$ depends only on the distance between $a$ and $z$ (the ``zone" of $z$ with respect to $a$). In this section we consider zonal polynomials on the unit sphere in $\cx^d$. Here, the \textsl{distance} between two points $a$ and $b$ on $\Om$ is defined to be $\abs{a^*b}$.  

If $f: \cx \rightarrow \cx$ is any univariate polynomial, then
\[
f_a(z) := f(\abs{a^*z}^2)
\]
is a function on $z \in \Om$ which depends only on $\abs{a^*z}$. Since $\abs{a^*z}^2 = (a^*z)(z^*a)$ is a polynomial in $z$ and $\bz$, so is $f_a$. If $f$ is not homogeneous, then terms in $f_a$ can be padded with powers of $Z = \sum z_i\bz_i$, which do not affect the value of the function on $\Om$. Thus $f_a$ defines a polynomial in $\Hom(r,r)$, where $r$ is the degree of $f$. Such functions are called the \defn{zonal polynomials} of $\Om$. The zonal polynomials of degree at most $r$ are denoted $Z(\Om,r)$.

We give one important example of zonal polynomials. For $a \in \Om$, define $Z_a \in \Harm(k,l)$ such that for every $p(x) \in \Harm(k,l)$,
\[
\ip{Z_a}{p} = p(a).
\]
\nomenclature{$Z_a$}{zonal polynomial}%
Since $\Harm(k,l)$ is a finite-dimensional inner-product space, $Z_a$ exists and is unique. Note that for any $a$ and $b$ in $\Om$,
\[
Z_a(b) = \ip{Z_b}{Z_a} = \conj{\ip{Z_a}{Z_b}} = \conj{Z_b(a)}.
\]

\begin{lemma}
In $\Harm(k,k)$,
\[
Z_a(b) = Z_b(a).
\]
\end{lemma}

\begin{proof}
It suffices to show that $Z_a$ is real-valued. Let $p = \Ima Z_a$, the imaginary part of $Z_a$, which is a real-valued homogeneous polynomial. Then
\[
0 = \Ima p(a) = \Ima \ip{Z_a}{p} = \ip{\Ima \conj{Z_a}}{p} = -\ip{\Ima Z_a}{\Ima Z_a}.
\]
\nomenclature[a$I]{$\Ima$}{imaginary part}%
For any inner product, $\ip{f}{f} \geq 0$, with equality only if $f = 0$. Thus $\Ima Z_a = 0$, and $Z_a$ is real. \qed
\end{proof}

\begin{comment}
Define $p = \Ima Z_a$ as the unique element of $\Harm(k,k)$ such that $\ip{p}{q} = \Ima q(a)$ to see that it is a polynomial. To see the middle equation, use the integral form of the inner product.
\end{comment}

Note that the set $\{Z_a : a \in \Om\}$ spans $\Harm(k,l)$. For, suppose that $v$ is in $\spn\{Z_a\}^\perp$, the subspace of polynomials orthogonal to the span of all $Z_a$. Then
\[
v(a) = \ip{Z_a}{v} = 0,
\]
from which it follows that $v = 0$ and $\spn\{Z_a\} = \Harm(k,l)$. Since $Z_a$ is defined by the inner product, it is also unitarily invariant:
\[
Z_a(z) = Z_{Ua}(Uz).
\]
Furthermore, the unitary mappings preserve distance, and using unitary matrices any pair of points can be mapped to any other pair with the same distance between them. It follows that $Z_a(z)$ depends only on $\abs{a^*z}$.

\begin{corollary} \label{cor:zonalharm}
$Z_a$ in $\Harm(k,k)$ is a zonal polynomial.
\end{corollary}

\begin{proof}
Using unitary mappings, it suffices to show the result for a single point $a \in \Om$, say $a = e_d$, the $d$-th standard basis vector in $\cx^d$. Then on $\Om$, $Z_a(z)$ depends only on $\abs{a^*z}^2 = z_d\bz_d$. More generally (again using unitary rotations), if we consider $Z_a$ as a polynomial on $\cx^d$, then $Z_a$ depends only on $z_d\bz_d$ and $Z = \sum z_i\bz_i$. It follows that $Z_a$ may be written as a polynomial in $z_d\bz_d$ and $Z$. Setting $Z = 1$ on $\Om$, we have a polynomial in $z_d\bz_d = \abs{a^*z}^2$.  \qed
\end{proof}

This zonal polynomial is called the \defn{zonal harmonic} with pole $a$. 

\subsection{Orthogonal zonal polynomials}

By Corollary \ref{cor:zonalharm}, any zonal harmonic $Z_a \in \Harm(k,k)$ may be written
\[
Z_a(z) = g_k(\abs{a^*z}^2)
\]
for some univariate polynomial $g_k$ of degree $k$. At this point we change notation for the zonal harmonics, since they are defined by $a$ and $g$. Relabel $Z_a$ as
\[
g_{k,a}(z) := g_k(\abs{a^*z}^2).
\]
\nomenclature{$g_{k,a}$}{zonal polynomial in $\Harm(k,k)$}%
\nomenclature{$g_k$}{Jacobi polynomial for $\Harm(k,k)$}%
Since $g_{k,a}$ is in $\Harm(k,k)$, by Theorem \ref{thm:homdecomp} it is orthogonal to $\Hom(i,i)$ for $i < k$. In particular, it is orthogonal to $g_{i,b}$ for any $b \in \Om$. For this reason, $g_{k,a}$ is called a \defn{zonal orthogonal polynomial} with respect to $a$. 

The zonal orthogonal polynomials have many nice properties. For example, by definition,
\[
\ip{g_{k,b}}{g_{k,a}} = g_{k,a}(b) = g_k(\abs{a^*b}^2),
\]

Another example is the following. This result, known as the \defn{addition formula} for $\Harm(k,k)$, was first proved by Koornwinder (see \cite{koo2}, \cite{koo1}).

\begin{theorem} \label{thm:addform1}
Let $S_1, \ldots, S_N$ be an orthonormal basis for $\Harm(k,k)$, and let $a$ and $b$ be in $\Om$. Then
\[
\sum_{i=1}^N \conj{S_i(a)}S_i(b) = g_k(\abs{a^*b}^2).
\]
\end{theorem}

\begin{proof}
Using Gram-Schmidt orthonormalization,
\[
g_{k,a} = \sum_{i=1}^N \ip{S_i}{g_{k,a}} S_i.
\]
But recall that $\ip{g_{k,a}}{S_i} = S_i(a)$. Taking the conjugate,
\[
g_{k,a} = \sum_{i=1}^N \conj{S_i(a)} S_i,
\]
which implies
\[
\sum_{i=1}^N \conj{S_i(a)} S_i(b) = g_{k,a}(b) = g_k(\abs{a^*b}^2). \qed
\] 
\end{proof}

\begin{corollary} \label{cor:dimharm}
\[
g_k(1) = \dim(\Harm(k,k)).
\]
\end{corollary}

\begin{proof}
Setting $a = b$ in the addition formula,
\[
\sum_{i=1}^N \conj{S_i(a)} S_i(a) = g_k(\abs{a^*a}^2) = g_k(1).
\]
The previous line is independent of the choice of $a$. Integrating over all of $\Om$, 
\begin{align*}
g_k(1) \int_\Om d\om &= \sum_{i=1}^N \int_\Om \conj{S_i(a)}S_i(a) \; d\om(a) \\
& = \sum_{i=1}^N \ip{S_i}{S_i} \\
& = N. \tag*{\sqr53}
\end{align*} 
\end{proof}

Since $g_{k,a}$ is harmonic, we can find an explicit formula for $g_k$ using the projection from $\Hom(k,k)$ into $\Harm(k,k)$ given by Corollary \ref{cor:harmproj}. Let $a = e_d$, and consider
\[
f_k(z) := \abs{a^*z}^{2k} = (z_d\bz_d)^k.
\]
Clearly $f_k$ is a zonal polynomial in $\Hom(k,k)$; the corresponding univariate polynomial is $x^k$. Note that $\De f_k = k^2f_{k-1}$. Its projection onto $\Harm(k,k)$ is
\begin{align*}
P(f_k) & = \sum_{r=0}^k (-1)^r \frac{(d+2k-r-2)!}{(d+2k-2)!r!} Z^r \De^r f_k \\
& = \sum_{r=0}^k (-1)^r \frac{(d+2k-r-2)!k!^2}{(d+2k-2)!r!(k-r)!^2} Z^r f_{k-r}. 
\end{align*}
On $\Om$, we may take $Z = 1$. Then the univariate polynomial underlying $P(f_k)$ is (abusing notation slightly):
\[
P(x^k) = \sum_{r=0}^k (-1)^r \frac{(d+2k-r-2)!k!^2}{(d+2k-2)!r!(k-r)!^2} \, x^{k-r}.
\]
Normalizing this polynomial so that the value at $1$ is $\dim(\Harm(k,k))$, we get
\[
g_k(x) = \frac{d+2k-1}{(d-1)!} \sum_{r=0}^k (-1)^r \frac{(d+2k-r-2)!}{r!(k-r)!^2} \, x^{k-r}.
\]
Explicitly, the first few polynomials are 
\begin{align*}
g_0(x) & = 1, \\
g_1(x) & = (d+1)(dx - 1), \\
g_2(x) & = \frac{d(d+3)}{4}((d+1)(d+2)x^2 - 4(d+1)x + 2).
\end{align*}

We will refer to these polynomials as the \defn{Jacobi polynomials}: up to a constant they are equivalent to a class of the usual Jacobi polynomials. In addition to defining orthogonal zonal polynomials, they are orthogonal in their own right. Fix $a \in \Om$, and define the following inner product on univariate polynomials:
\[
\ip{f}{g} := \ip{f_a}{g_a}
\]
where $f_a$ is the zonal polynomial with respect to $a$ induced by $f$. Clearly, for the Jacobi polynomials with $i \neq k$,
\[
\ip{g_i}{g_k} = 0.
\]
A sequence of polynomials $g_0, g_1, \ldots$ with $g_i$ of degree $i$ is called an \defn{orthogonal polynomial sequence} if the polynomials are pairwise orthogonal with respect to an inner product satisfying
\begin{equation}
\ip{f}{g} = \ip{1}{fg}.
\label{eqn:ops}
\end{equation}
Our inner product is defined in terms of integration of zonal polynomials, and it satisfies condition \eqref{eqn:ops} whenever $f$ and $g$ are real valued. Hence the Jacobi polynomials are an orthogonal polynomial sequence. The next lemma is a standard result.
 
\begin{lemma}
An orthogonal polynomial sequence satisfies the following three-term recursion: there are constants $a_k$, $b_k$ and $c_k$ such that
\[
a_kg_{k+1} = (x-b_k)g_k - c_kg_{k-1}.
\]
\end{lemma}

\begin{proof}
Since $xg_k$ has degree $k+1$, it is a linear combination of $\{g_0,\ldots,g_{k+1}\}$. Using Gram-Schmidt orthonormalization, 
\[
xg_k = \sum_{i=0}^{k+1} \frac{\ip{g_i}{xg_k}}{\ip{g_i}{g_i}} g_i.
\]
But for $i < k-1$, $xg_i$ is a polynomial of degree less than $k$, and so
\[
\ip{g_i}{xg_k} = \ip{xg_i}{g_k} = 0.
\]
Therefore, there are constants $a_k$, $b_k$, and $c_k$ such that
\[
xg_k = a_kg_{k+1} + b_kg_k + c_kg_{k-1}. \qed
\] 
\end{proof}

For the Jacobi polynomials, this three term recurrence was computed explicitly by Delsarte, Goethals, and Seidel \cite{dgs}. Define 
\begin{equation}
\la_k = \frac{k}{d+2k-1}; \qquad \mu_k = \frac{k+1}{d+2k}.
\label{eqn:lamu}
\end{equation}
Then
\[
g_{k+1} = \frac{x + (\la_k-1)\mu_k + \la_k(\mu_{k-1}-1)}{\la_{k+1}\mu_{k}}g_k - \frac{(\la_{k-1}-1)(\mu_{k-1}-1)}{\la_{k+1}\mu_{k}}g_{k-1}.
\]

The Jacobi polynomials which are so useful in $\Harm(k,k)$ can also be adapted for $\Hom(k,k)$. Define the \defn{Jacobi sum polynomial} of degree $k$ to be 
\[
p_k(x) := \sum_{r=0}^k g_r(x).
\]
\nomenclature{$p_k$}{Jacobi sum polynomial for $\Hom(k,k)$}%
\nomenclature{$p_{k,a}$}{zonal polynomial in $\Hom(k,k)$}%
Since $g_r$ has real coefficients, so does $p_k$. Now consider $p_{k,a}$, the zonal polynomial with pole $a$ induced by $p_k$. As with $g_{k,a}$, we may pad with multiples of $Z$ and therefore assume $p_{k,a}$ is in $\Hom(k,k)$. 

The harmonic decomposition of $\Hom(k,k)$ in Theorem \ref{thm:homdecomp} and the fact that $g_r(1) = \dim(\Harm(r,r))$ imply that
\[
p_r(1) = \dim(\Hom(r,r)).
\]

\begin{lemma} \label{lem:jacobsumzonal}
The Jacobi sum $p_{k,a}$ is the unique polynomial in $\Hom(k,k)$ such that for every $f \in \Hom(k,k)$, 
\[
\ip{p_{k,a}}{f} = f(a).
\]
\end{lemma}

\begin{proof}
By Corollary \ref{cor:homdecomp}, and the fact that multiples of $Z$ do not change the inner product, it suffices to show the result for $f \in \Harm(i,i)$, where $i \leq k$. But $\Harm(i,i)$ and $\Harm(r,r)$ are orthogonal for $i \neq r$, so
\[
\ip{p_{k,a}}{f} = \ip{g_{r,a}}{f} = f(a). \qed
\] 
\end{proof}

\subsection{$\Harm(k+1,k)$ zonals}

All of the results in the previous section about zonal polynomials in $\Harm(k,k)$ can be extended to $\Harm(k+1,k)$. Define $h_{k,a}$\index{Jacobi polynomials} to be the unique polynomial in $\Harm(k+1,k)$ such that
\[
\ip{h_{k,a}}{p} := p(a).
\]
Then similarly to Corollary \ref{cor:zonalharm}, we find that 
\[
h_{k,a}(z) = (a^*z)h_k(\abs{a^*z}^2)
\]
\nomenclature{$h_{k,a}$}{zonal polynomial in $\Harm(k+1,k)$}%
\nomenclature{$h_k$}{Jacobi polynomial for $\Harm(k+1,k)$}%
for univariate polynomials $h_k$. The addition formula for $\Harm(k+1,k)$ is again due to Koornwinder. The proof is nearly identical to that of Theorem \ref{thm:addform1}.

\begin{comment}
As before, by unitary rotations, $h_{k,a}(z)$ is a function of $a^*z$. Fixing $a = e_d$ for convenience, $h_{k,a}(z)$ is a function of $z_d$ on $\Om$ and of $z_d$ and $Z$ on $\cx^d$. Since it is in $\Harm(k+1,k)$, it is a polynomial, and thererfore a polynomial on $z_d$ and $Z$, and therefore a polynomial in $z_d$ on $\Om$.
\end{comment}

\begin{theorem} \label{thm:addform2}
Let $S_1, \ldots, S_N$ be an orthonormal basis for $\Harm(k+1,k)$, and let $a$ and $b$ be in $\Om$. Then
\[
\sum_{i=1}^N \conj{S_i(a)}S_i(b) = (a^*b)h_k(\abs{a^*b}^2). \qed
\]
\end{theorem}

As in Corollary \ref{cor:dimharm},
\[
h_k(1) = \dim(\Harm(k+1,k)).
\]
Projecting the polynomial
\[
f_k(z) := (a^*z)\abs{a^*z}^{2k} = z_d^{k+1}\bz_d^k
\]
from $\Hom(k+1,k)$ to $\Harm(k+1,k)$ and normalizing, we get an explicit formula:
\[
h_k(x) = \frac{d+2k}{(d-1)!} \sum_{r=0}^k (-1)^r \frac{(d+2k-r-1)!}{r!(k-r+1)!(k-r)!}x^{k-r}.
\]
\begin{comment}
(The corresponding univariate polynomial is $x^k$. Note that $\De p_k = (k+1)kp_{k-1}$. 
\begin{align*}
P(p_k) & = \sum_{r=0}^k (-1)^r \frac{(d+2k-r-1)!}{(d+2k-1)!r!}Z^r \De^r p_k \\
& = \sum_{r=0}^k (-1)^r \frac{(d+2k-r-1)!(k+1)!k!}{(d+2k-1)!r!(k-r+1)!(k-r)!}Z^r p_{k-r}. 
\end{align*}
Then 
\[
P(x^k) = \sum_{r=0}^k (-1)^r \frac{(d+2k-r-1)!(k+1)!k!}{(d+2k-1)!r!(k-r+1)!(k-r)!}x^{k-r},
\]
and normalize so that $h_k(1) = \dim(\Harm(k+1,k))$.
\end{comment}

The first few polynomials are 
\begin{align*}
h_0(x) & = d, \\
h_1(x) & = \frac{d(d+2)}{2}((d+1)x - 2), \\
h_2(x) & = \frac{d(d+1)(d+4)}{12}((d+2)(d+3)x^2 - 6(d+2)x + 6).
\end{align*}

Finally, the three term recurrence for $h_k$ is
\[
h_{k+1} = \frac{x + \la_{k+1}(\mu_k-1) + (\la_k-1)\mu_k}{\la_{k+1}\mu_{k+1}}h_k - \frac{(\la_k-1)(\mu_{k-1}-1)}{\la_{k+1}\mu_{k+1}}h_{k-1},
\]
with $\la_k$ and $\mu_k$ defined as in \eqref{eqn:lamu}.

The Jacobi sum polynomial\index{Jacobi sum polynomial} for $\Hom(k+1,k)$ is 
\[
q_k(x) := \sum_{r=0}^k h_r(x).
\]
\nomenclature{$q_k$}{Jacobi sum polynomial for $\Hom(k+1,k)$}%
\nomenclature{$q_{k,a}$}{zonal polynomial in $\Hom(k+1,k)$}%
From the harmonic decomposition of $\Hom(k+1,k)$, we get
\[
q_k(1) = \dim(\Hom(k+1,k)),
\]
and for every $f \in \Hom(k+1,k))$,
\[
\ip{q_{k,a}}{f} = f(a).
\]

\section{$s$-distance sets}\label{sec:sdist}

We are interested in sets of complex lines with restrictions on the angles between them. Let $X$ be a subset of $\Om$, the unit sphere in $\cx^d$. The \defn{degree set} of $X$ is the set
\[
A = \{\al: \abs{x^*y}^2 = \al, x \neq y \in X\}.
\]
\nomenclature[\a]{$\al$}{angle}%
Then $X$ an \defn{$s$-distance set} if $|A| = s$. We always assume a projective line occurs at most once in $X$, so that $1 \notin A$. If $X$ is the set of lines from mutually unbiased bases, then $A = \{0,\al\}$. If $X$ is a set of equiangular lines, then $|A| = 1$.

The following result, due to Delsarte, Goethals and Seidel \cite{dgs}, is called the \defn{absolute bound}. The proof is adapted from Godsil \cite[Theorem 14.4.1]{blue}.

\begin{theorem} \label{thm:sdistbnd}
Let $X$ be an $s$-distance set. Then
\[
|X| \leq \dim(\Hom(s,s)).
\]
If $0$ is in the degree set of $X$, then
\[
|X| \leq \dim(\Hom(s,s-1)).
\]
\end{theorem}

\begin{proof}
Let $A$ be the degree set of $X$, with $1 \notin A$. The \defn{annihilator} of $A$ is
\[
f(x) = \prod_{\al \in A}(x-\al).
\]
Now for each $v \in X$, consider the zonal polynomial with pole $v$ induced by $f$:
\[
f_v(z) := f(\abs{v^*z}^2).
\]
In general, $f_v$ is not homogeneous; however, by padding terms with $Z = \sum z_i\bz_i$, we may take $f_v$ in $\Hom(s,s)$. For any $u \neq v$ in $X$, note that $f_u(v) = 0$, while $f_v(v) \neq 0$. This implies that $\{f_v : v \in X\}$ is a linearly independent set. Since the polynomials are independent, they must number fewer than the dimension of the space in which they reside. 

When one of the angles is $0$, consider the polynomials $f_v(z) = (v^*z)f(\abs{v^*z}^2)$, where $f$ is the annihilator of $A-\{0\}$. These polynomials reside in $\Hom(s,s-1)$, and the proof is similar. \qed
\end{proof}

If equality holds, then the set $\{f_v : v \in X\}$ spans $\Hom(s,s)$ (or $\Hom(s,s-1)$ in the case of $0 \in A$.)

\begin{corollary} \label{cor:ealabsbnd}
Let $X$ be a set of equiangular lines in $\cx^d$. Then
\[
|X| \leq d^2. \qed
\]
\end{corollary}

\section{$t$-designs}\label{sec:tdes}

Let $X$ be a finite subset of $\Om$, and let $f \in \Hom(k,l)$ and $g \in \Hom(m,n)$ be polynomials. Then $\ip{f}{g}_X$ denotes the average of $\conj{f}g$ over $X$:
\[
\ip{f}{g}_X := \frac{1}{|X|}\sum_{z \in X}\conj{f(z)}{g(z)}.
\]
\nomenclature[\]{$\ip{f}{g}_X$}{average value of $\conj{f}g$ over $X$}%
We call $X$ a \defn{$t$-design} if for every $f$ in $\Hom(t,t)$,
\[
\ip{1}{f}_X = \ip{1}{f}.
\]
That is, the average of $f$ over $X$ is the same as the average of $f$ over all of $\Om$. 

\begin{lemma}
A subset $X$ is a $t$-design if and only if for every $f \in \Harm(k,k)$ with $1 \leq k \leq t$,
\[
\sum_{z \in X} f(z) = 0.
\]
\end{lemma}

\begin{proof}
Recall that if $f$ is in $\Harm(k,k)$, then $\ip{1}{f} = 0$. Thus if $X$ is a $t$-design, $\ip{1}{f}_X = 0$. Conversely, if $\ip{1}{f}_X = 0 = \ip{1}{f}$ for every harmonic $f$, the polynomial decomposition in Corollary \ref{cor:homdecomp} implies that $\ip{1}{f}_X = \ip{1}{f}$ for every $f$ in $\Hom(t,t)$. \qed
\end{proof}

Note that if $f$ is in $\Hom(t-1,t-1)$, then $Zf$ is in $\Hom(t,t)$ and takes the same values as $f$ on $\Om$. It follows that if $X$ is a $t$-design, it is also a $(t-1)$-design. A design has \defn{strength} $t$ if $t$ is the largest value such that it is a $t$-design.

\begin{theorem} \label{thm:tdesbnd}\index{absolute bound}
If $X$ is a $t$-design, then
\[
|X| \geq \dim(\Hom(\lc t/2 \rc,\lf t/2 \rf)).
\]
\end{theorem}

\begin{proof}
Let $S_1,\ldots,S_N$ be an orthogonal basis for $\Hom(\lc t/2 \rc,\lf t/2 \rf)$. Then $\conj{S_i}S_j$ is in $\Hom(t,t)$. Since $X$ is a $t$-design,
\[
\ip{S_i}{S_j} = \ip{1}{\conj{S_i}S_j} = \ip{1}{\conj{S_i}S_j}_X = \ip{S_i}{S_j}_X.
\]
So the polynomials $S_i: X \rightarrow \cx$ are orthogonal, and therefore independent, as functions on $X$. The space of functions on $X$ has dimension $|X|$. \qed
\end{proof}

If equality holds, then the basis for $\Hom(\lc t/2 \rc,\lf t/2 \rf)$ is also a basis for the functions on $X$. Bannai and Hoggar \cite{bh1} have shown that equality can only hold for $t < 6$.

Another lower bound on the size of a $t$-design is the following, known as the \defn{linear programming bound} (see \cite[Theorem 14.5.3]{blue}). 

\begin{lemma}
Let $X$ be a $t$-design, and let $p \in \Hom(t,t)$ be real and non-negative on $X$. Then for any $a \in X$,
\[
|X| \geq \frac{p(a)}{\ip{1}{p}}.
\]
\end{lemma}

\begin{proof}
Since $p$ is nonnegative on $X$,
\[
p(a) \leq \sum_{z \in X}p(z) = |X| \ip{1}{p}_X.
\]
But $X$ is a $t$-design, so $\ip{1}{p}_X = \ip{1}{p}$, and the result follows. \qed
\end{proof}

If equality holds, then $p(z) = 0$ for every $z$ in $X$ except $a$.

\section{Relative bounds}\label{sec:absrel}

In this section we establish tighter upper bounds for $s$-distance sets and $t$-designs. In the following, $g_r$ is the Jacobi polynomial of degree $r$.

\begin{lemma} \label{lem:jacobsum}
For any finite subset $X$ of $\Om$,
\[
\sum_{a,b \in X} g_r(\abs{a^*b}^2) \geq 0.
\]
\end{lemma}

\begin{proof}
\begin{align*}
\sum_{a,b \in X} g_r(\abs{a^*b}^2) & = \sum_{a,b \in X} \ip{g_{r,a}}{g_{r,b}} \\
& = \ip{\sum_{a \in X} g_{r,a}}{\sum_{a \in X} g_{r,a}} \\
& \geq 0. \tag*{\sqr53}
\end{align*} 
\end{proof}

Note that equality holds in Lemma \ref{lem:jacobsum} if and only if $\sum_{a \in X} g_{r,a} = 0$, which occurs if and only if $\ip{1}{g_{r,a}}_X = 0$ for every $a \in \Om$.

The following result is called the \defn{relative bound} for $s$-distance sets. It is due to Delsarte, Goethals and Seidel \cite{dgs}, although the proof is adapted from \cite[Theorem 16.4.2]{blue}. Note that the Jacobi polynomials span $\re[x]$, so any univariate polynomial may be written as a linear combination of them. 

\begin{theorem} \label{thm:sdistrel}
Let $X \sbs \Om$ have finite degree set $A$, and let $F(x) \in \re[x]$ be a polynomial such that
\begin{enumerate}[(a)]
\item $F(\al_i) \leq 0$ for each distance $\al_i \in A$, and
\item if $F(x) = \sum_r c_r g_r(x)$, then $c_r \geq 0$ for all $r$ and $c_0 > 0$.
\end{enumerate}
Then
\[
|X| \leq F(1)/c_0.
\]
\end{theorem}

\begin{proof}
Let $F_a$ denote the zonal polynomial induced by $F$ with pole $a$, so that 
$F_a(b) = F(\abs{a^*b}^2) \leq 0$ for $b \neq a$. Summing over all $b \in X$,
\[
|X|\ip{1}{F_a}_X \leq F_a(a) = F(1).
\]
Again averaging over all $a \in X$,
\begin{align*}
F(1) & \geq \sum_{a \in X} \ip{1}{F_a}_X \\
& = \sum_{a \in X} \sum_r c_r \ip{1}{g_{r,a}}_X \\
& = \sum_r c_r \sum_{a \in X} \ip{1}{g_{r,a}}_X.
\end{align*}
By Lemma \ref{lem:jacobsum}, the inner sum is non-negative for $r > 0$. If $r = 0$, then $g_{0,a}(b) = 1$ for all $b$, and hence,
\begin{align*}
F(1) & \geq c_0 \sum_{a \in X} \ip{1}{g_{0,a}}_X \\
& = c_0 |X|. \tag*{\sqr53}
\end{align*} 
\end{proof}

Equality holds in Theorem \ref{thm:sdistrel} if and only if $F(\al_i) = 0$ for every $\al_i \in A$, and for every $r > 0$, either $c_r = 0$ or $\ip{1}{g_{r,a}}_X = 0$ for every $a \in \Om$. Since $\{g_{r,a}: a \in \Om\}$ spans $\Harm(r,r)$, we have the following:

\begin{corollary} \label{cor:sdistrel}
If equality holds in Theorem \ref{thm:sdistrel}, and $c_r > 0$ for every $r$ less than $s$, the degree of $F$, then $X$ is an $s$-design. Conversely, if $X$ is an $s$-design and $F(\al_i) = 0$ for every $\al_i \in A$, then equality holds in Theorem \ref{thm:sdistrel}. \qed
\end{corollary}

An $s$-distance set of size $n$ in $\cx^d$ which is also a $t$-design is sometimes called a \textsl{$(d,n,s,t)$-configuration}.

The following bound is due to Wootters and Fields \cite{wf1}, while the equality condition is due to Klappenecker and R\"otteler \cite{kr2}.

\begin{corollary} \label{cor:mubbnd}
Let $X$ be the lines from a set of mutually unbiased bases in $\cx^d$. Then
\[
|X| \leq d(d+1).
\]
Equality holds if and only if $X$ is a $2$-design. \qed
\end{corollary}

\begin{proof}
It is not difficult to show that the angle between lines from different bases in $X$ must be $1/d$. Thus $A = \{0,1/d\}$. Let $F(x)$ be the annihilator of $A$:
\[
F(x) = x\left(x-\frac{1}{d}\right).
\]
Expanding in terms of the Jacobi polynomials, we find that each $c_i > 0$, and in particular,
\[
c_0 = \frac{d-1}{d^2(d+1)}.
\]
\begin{comment}
In fact,
\[
F(x) = \frac{4}{d(d+1)(d+2)(d+3)} g_2(x) + \frac{3d-2}{d^2(d+1)(d+2)} g_1(x) + \frac{d-1}{d^2(d+1)} g_0(x).
\]
\end{comment}
Now note that $F(1) = 1-1/d$ and apply Theorem \ref{thm:sdistrel}. For equality, apply Corollary \ref{cor:sdistrel}. \qed
\end{proof}

The relative bound\index{relative bound} for equiangular lines is the following.

\begin{corollary} \label{cor:ealrelbnd}
Let $X$ be a set of equiangular lines in $\cx^d$ with angle $\al < 1/d$. Then
\[
|X| \leq \frac{d(1-\al)}{1-d\al}.
\]
Equality holds if and only if $X$ is a $1$-design.
\end{corollary}

\begin{proof}
Let $F$ be the annihilator of $A = \{\al\}$:
\[
F(x) = x-\al = \frac{1}{d(d+1)}g_1(x) + \left(\frac{1}{d}-\al\right)g_0(x).
\]
Now apply Theorem \ref{thm:sdistrel}. For equality, apply Corollary \ref{cor:sdistrel}. \qed
\end{proof}

A particularly interesting case of equality in Theorem \ref{thm:sdistrel} is when $X$ is an $s$-distance set and $F$ is the Jacobi sum polynomial $p_s$. In this case, $F(1)/c_0$ is $\dim(\Hom(s,s))$, and equality is obtained in the bound in Theorem \ref{thm:sdistbnd}.

\begin{corollary}\label{cor:distbnddes}
If $X$ is an $s$-distance set and 
\[
|X| = \dim(\Hom(s,s)),
\]
then $X$ is a $2s$-design.
\end{corollary}

\begin{proof}
Recall from Theorem \ref{thm:sdistbnd} that if $|X| = \dim(\Hom(s,s))$, then the zonal polynomials $\{f_a : a \in X\}$ induced by the annihilator $f$ of $A$ are a basis for $\Hom(s,s)$. Now consider $p_{s,a}$ and $f_b$; by Lemma \ref{lem:jacobsumzonal},
\[
\ip{p_{s,a}}{f_b} = f_b(a).
\]
This inner product is $0$ when $a \neq b$ and $f(1) \neq 0$ otherwise. This implies that $\{p_{s,a} : a \in X\}$ is a second basis for $\Hom(s,s)$. It also implies that $p_{s,a}$ is a constant multiple of $f_a$; thus, $p_s(\al_i) = 0$ (and the hypotheses of Theorem \ref{thm:sdistrel} are satisfied).

On the other hand, since $f(\al_i) = p_s(\al_i) = 0$ for every $\al_i$, the average of $p_{s,a}(z)f_b(z)$ over $X$ is 
\[
\ip{p_{s,a}}{f_b}_X = \begin{cases}
f(1), & a = b; \\
0, & \text{otherwise.}
\end{cases}
\]
We conclude that $\ip{p_{s,a}}{f_b}_X = \ip{p_{s,a}}{f_b}$, for all $a$ and $b$. Since both of these sets are bases, we have $\ip{f}{g}_X = \ip{f}{g}$ for every $f$ and $g$ in $\Hom(s,s)$. But since $\ip{f}{g} = \ip{1}{\conj{f}g}$, this implies $\ip{1}{f}_X = \ip{1}{f}$ for every $f$ in $\Hom(2s,2s)$. \qed
\end{proof}

The argument in Theorem \ref{thm:sdistrel} yields another bound (again due to Delsarte, Goethals, and Seidel) which is most useful when one of the angles of $X$ is $0$.

\begin{theorem} \label{thm:sdistrel2}
Let $X \sbs \Om$ have finite degree set $A$, and let $F(x) \in \re[x]$ be a polynomial such that
\begin{enumerate}[(a)]
\item $\al_i F(\al_i) \leq 0$ for each distance $\al_i \in A$, and
\item if $F(x) = \sum_r c_r h_r(x)$, then $c_r \geq 0$ and $c_0 > 0$.
\end{enumerate}
Then
\[
|X| \leq F(1)/c_0.
\]
\end{theorem}

\begin{proof} Similarly to Lemma \ref{lem:jacobsum}, since 
\[
\abs{a^*b}^2 h_r(\abs{a^*b}^2) = (a^*b) \ip{h_{k,a}}{h_{k,b}} = \ip{a h_{k,a}}{b h_{k,b}},
\]
we get
\[
\sum_{a,b \in X} \abs{a^*b}^2 h_r(\abs{a^*b}^2) \geq 0,
\]
with equality if and only if $\ip{1}{a h_{r,a}}_X = 0$ for every $a \in \Om$. Since $\al_i F(\al_i) \leq 0$, 
\begin{align*}
|X|F(1) & \geq \sum_{a,b \in X} \abs{a^*b}^2 F(\abs{a^*b}^2) \\
& = \sum_r c_r \sum_{a,b \in X} \abs{a^*b}^2 h_r(\abs{a^*b}^2) \\ 
& \geq c_0 \sum_{a,b \in X} \abs{a^*b}^2 h_0(\abs{a^*b}^2). 
\end{align*}
Since $xh_0(x) = g_1(x)/(d+1) + g_0(x)$, this reduces to 
\begin{align*}
|X|F(1) & \geq  c_0 \sum_{a,b \in X} \frac{g_1(\abs{a^*b}^2)}{d+1} + g_0(\abs{a^*b}^2) \\
|X|F(1) & \geq  c_0 \sum_{a,b \in X} g_0(\abs{a^*b}^2) \\
& = c_0 |X|^2. \inqed
\end{align*}
\end{proof}

\begin{comment}
In the proof of the previous result,
\[
\abs{a^*b}^2 h_r(\abs{a^*b}^2) = (a^*b) \ip{h_{k,a}}{h_{k,b}} = \ip{a h_{k,a}}{b h_{k,b}},
\]
where $ah_{k,a}$ is in $\Harm(k+1,k)^d$, not $\Harm(k+1,k)$. Define the inner product on $\Harm(k+1,k)^d$ by
\[
\ip{f}{g} = \int_\Om f(\om)^*g(\om) \; d\om.
\]
\end{comment}

Equality holds if and only if $\al_i F(\al_i) = 0$ for every $\al_i$, and for every $r > 0$ and $a \in \Om$, either $c_r = 0$ or $\ip{1}{a h_{r,a}}_X = 0$. This latter condition implies $\ip{1}{f}_X = 0$ for every $f \in \Harm(r+1,r)$. Applying this result when $F$ is the annihilator of $A-\{0\} = \{1/d\}$ gives another proof of Corollary \ref{cor:mubbnd} for mutually unbiased bases. Equality implies $\ip{1}{f}_X = 0$ for every $f \in \Hom(2,1)$.

When $X$ is an $s$-distance set with $0 \in A$ and $F$ is the Jacobi sum polynomial $q_{s-1}$, equality in Theorem \ref{thm:sdistrel2} implies equality in the second half of Theorem \ref{thm:sdistbnd} as well.

\begin{corollary} \label{cor:distbnddes2}
If $X$ is an $s$-distance set with $0 \in A$ and
\[
|X| = \dim(\Hom(s,s-1)),
\]
then $X$ is a $(2s-1)$-design. \qed
\end{corollary}

There is also a \defn{relative bound} for $t$-designs. This result is due to Neumaier \cite{neu1}, although the analogous result over the reals was first given by Delsarte et al. \cite{dgs2}.

\begin{theorem} \label{thm:tdesrel}
Let $X$ be a $t$-design, and let $F(x) \in \re[x]$ be a polynomial such that 
\begin{enumerate}[(a)]
\item $F(\al) \geq 0$ for every $\al$ in the degree set of $X$, and $F(1) > 0$;
\item if $F(x) = \sum_r c_r g_r(x)$, then $c_0 > 0$ and $c_r \leq 0$ for $r > t$.
\end{enumerate}
Then
\[
|X| \geq F(1)/c_0.
\]
\end{theorem}

\begin{proof}
As in the proof of Theorem \ref{thm:sdistrel}, if $F_a$ is the zonal polynomial induced by $F$, then $F_a(b) \geq 0$, and
\[
|X|\ip{1}{F_a}_X \geq F(1).
\]
Averaging over all $a \in X$,
\begin{align*}
F(1) & \leq \sum_{a \in X} \ip{1}{F_a}_X \\
& = \sum_r c_r \sum_{a \in X} \ip{1}{g_{r,a}}_X.
\end{align*}
Again the inner sum is non-negative. Consider the three cases for $r$. When $r > t$, $c_r$ is non-positive by assumption. When $0 < r \leq t$, the inner sum is $0$ because $g_{r,a}$ is harmonic and $X$ is a $t$-design. When $r = 0$, $g_{0,a} = 1$. It follows that
\begin{align*}
F(1) & \leq c_0 \sum_{a \in X} \ip{1}{g_{0,a}}_X + \sum_{r>t} c_r \sum_{a \in X} \ip{1}{g_{r,a}}_X \\
& \leq c_0 |X|. \tag*{\sqr53}
\end{align*} 
\end{proof}

Equality holds in Theorem \ref{thm:tdesrel} if and only if $F(\al) = 0$ for every $\al$ in the degree set $A$ of $X$, and for every $r>t$, either $c_r = 0$ or $\sum_{a \in X} g_{r,a} = 0$. 

There is an analogous theorem working in $\Hom(k+1,k)$ that is most useful when $0$ is in the degree set. The proof of the following is similar to that of Theorem \ref{thm:tdesrel}, but this result is new.

\begin{theorem} \label{thm:tdesrel2}
Let $X$ be a $t$-design, and let $F(x) \in \re[x]$ be a polynomial such that 
\begin{enumerate}[(a)]
\item $\al F(\al) \geq 0$ for every $\al$ in the degree set of $X$, and $F(1) > 0$;
\item if $F(x) = \sum_r c_r h_r(x)$, then $c_0 > 0$ and $c_r \leq 0$ for $r \geq t$.
\end{enumerate}
Then
\[
|X| \geq F(1)/c_0. \qed
\]
\end{theorem}

Combining the theorems of this section gives useful information when $X$ is both an $s$-distance set and a $t$-design. Theorem \ref{thm:sdistrel} together with Theorem \ref{thm:tdesrel} give the following.

\begin{corollary}
Let $X$ be an $s$-distance set and a $t$-design with $t \geq s$. If the annihilator $F(x) = \sum_r c_rg_r(x)$ of the degree set of $X$ satisfies $c_r \geq 0$ for each $r$, then
\[
|X| = F(1)/c_0. \qed
\]
\end{corollary}

Similarly, from Theorems \ref{thm:sdistrel2} and \ref{thm:tdesrel2}:

\begin{corollary}
Let $X$ be an $s$-distance set and a $t$-design with $t \geq s$ and $0$ in $A$, the degree set of $X$. If the annihilator $F(x) = \sum c_rh_r(x)$ of $A-\{0\}$ satisfies $c_r \geq 0$ for each $r$, then
\[
|X| = F(1)/c_0. \qed
\]
\end{corollary}

As with Theorem \ref{thm:sdistrel}, the case of equality in Theorem \ref{thm:tdesrel} when $t = 2s$ and $F = p_s = \sum_{r=0}^s g_r$ is of particular interest. Here $F(1) = \dim(\Hom(t/2,t/2))$, the lower bound in Theorem \ref{thm:tdesbnd}. If equality holds, then any basis for $\Hom(s,s)$ spans the functions on $|X|$. Now suppose $X$ has a finite degree set $A$, and let $f$ be the annihilator of $A$. Then with $a \in X$, the zonal polynomial $f_a$ (restricted to $X$) is in $\Hom(s,s)$. We conclude that $f$ has degree at most $s$ and therefore $|A| \leq s$. 

\begin{corollary} \label{cor:desbnddist}
If $X$ is a $2s$-design with finite degree set and 
\[
|X| = \dim(\Hom(s,s)),
\]
then $X$ is an $s$-distance set. \qed
\end{corollary}

Combining Corollary \ref{cor:distbnddes} with Corollary \ref{cor:desbnddist}, we get:

\begin{corollary} Let $X$ be a set of lines in $\cx^d$. Then any two of the following conditions imply the third:
\begin{enumerate}[(a)]
\item $X$ is an $s$-distance set;
\item $X$ is a $2s$-design;
\item $|X| = \dim(\Hom(s,s))$. \qed
\end{enumerate}
\end{corollary}

One example is when $X$ is a maximal set of equiangular lines. This corollary was observed by Renes et al. \cite{ren1} and Zauner \cite{zau1}.

\begin{corollary} \label{cor:maxequides}
Let $X$ be a set of lines in $\cx^d$. Then any two of the following conditions imply the third:
\begin{enumerate}[(a)]
\item $X$ is a set of equiangular lines;
\item $X$ is a $2$-design;
\item $|X| = d^2$. \qed
\end{enumerate}
\end{corollary}

Finally, consider equality in the new relative bound, Theorem \ref{thm:tdesrel2}. Suppose $F = q_{s-1} = \sum_{r=0}^{s-1} h_r$, with $t = 2s-1$, and equality holds. Then $F(1) = \dim(\Hom(s,s-1))$, so equality holds in Theorem \ref{thm:tdesbnd} and any basis for $\Hom(s,s-1)$ spans the functions on $|X|$. Now assume $0$ is in the degree set $A$ of $X$, and let $f$ be the annihilator of $A-\{0\}$. Then for $a \in X$, 
\[
f_a(x) := (a^*x)f(\abs{a^*x}^2)
\]
is in $\Hom(s,s-1)$. Therefore $f$ has degree at most $s-1$, and so $|A| \leq s$.

\begin{corollary} \label{cor:desbnddist2}
If $X$ is a $(2s-1)$-design with finite degree set $A$ containing $0$ and 
\[
|X| = \dim(\Hom(s,s-1)),
\]
then $X$ is a $s$-distance set. \qed
\end{corollary}

The next result combines Corollary \ref{cor:desbnddist2} with Corollary \ref{cor:distbnddes2}.

\begin{corollary} \label{cor:tfae2}
Let $X$ be a set of lines in $\cx^d$ with $0$ in the degree set of $X$. Then any two of the following conditions imply the third:
\begin{enumerate}[(a)]
\item $X$ is an $s$-distance set;
\item $X$ is a $(2s-1)$-design;
\item $|X| = \dim(\Hom(s,s-1))$. \qed
\end{enumerate}
\end{corollary}

If $s = 1$ in Corollary \ref{cor:tfae2}, then $X$ is an orthonormal basis for $\cx^d$. Thus any orthonormal basis is a $1$-design. When $s = 2$, the degree set of $X$ is $\{0,2/(d+2)\}$ and we have a $3$-design. Three examples are known: a set of $6$ lines in $\cx^2$, which form three mutually unbiased bases; a set of $40$ lines in $\cx^4$, constructed from the Witting polytope (See Coxeter \cite[Section 12.5]{cox1}); and a set of $126$ lines in $\cx^6$ due to Mitchell \cite{mit1}.

\section{Algebras}\label{sec:gram}

Let $X \sbs \Om$ have degree set $A = \{\al_1,\ldots,\al_s\}$ and let $\al_0 = 1$. Then define a set of matrices $\cA = \{A_0,\ldots,A_s\}$ with rows and columns indexed by $X$ such that
\[
(A_i)_{a,b} := \begin{cases}
1, & \abs{a^*b}^2 = \al_i; \\
0, & \text{otherwise.}
\end{cases}
\]
Note that the matrices are Schur idempotents with $A_0 = I$ and $\sum_i A_i = J$. In this section, we consider the conditions under which $\cA$ is an association scheme. The results are due to Delsarte, Goethals, and Seidel, although some of the proofs are new.

Define a second set of matrices $E_0,E_1,\ldots$ also indexed by $X$ as follows:
\[
(E_r)_{a,b} := \frac{1}{|X|}g_r(\abs{a^*b}^2).
\]
Each $E_r$ is real, symmetric, and in the span of $\cA$:
\[
E_r = \frac{1}{|X|}\sum_{i=0}^s g_r(\al_i) A_i.
\]

\begin{lemma} \label{lem:desidem}
If $X$ is a $2e$-design, then $E_0, \ldots, E_e$ are orthogonal idempotents.
\end{lemma}

\begin{proof}
Consider the product of $E_i$ and $E_j$, for any $i,j \leq e$:
\begin{align*}
(E_iE_j)_{a,b} & = \frac{1}{|X|}\sum_{z \in X} g_i(\abs{a^*z}^2)g_j(\abs{b^*z}^2) \\
& = \ip{1}{g_{i,a}g_{j,b}}_X.
\end{align*}
Since $i+j \leq 2e$ and $X$ is a $2e$-design, this term equals $\ip{1}{g_{i,a}g_{j,b}}$. But $g_{i,a}$ and $g_{j,b}$ are orthogonal for $i \neq j$, and otherwise their inner product is $g_i(\abs{a^*b}^2)$. Thus,
\[
E_iE_j = \begin{cases}
E_i, & i = j; \\
0, & \text{otherwise.}
\end{cases} \qed
\] 
\end{proof}

The same argument shows that if $X$ is a $(2e+1)$-design, then $E_0,\ldots,E_{e+1}$ are linearly independent. Since these matrices are in $\spn(\cA)$, a space of dimension $s+1$, we have:

\begin{corollary}
If $X$ is an $s$-distance set and a $t$-design, then $t \leq 2s$. \qed
\end{corollary}

If $X$ is an $s$-distance set and a $2s$-design, then $E_0,\ldots,E_s$ are linearly independent and therefore spanning in $\spn(\cA)$. Since $E_0,\ldots,E_s$ are closed under matrix multiplication, it follows that $\spn(\cA)$ is also closed under matrix multiplication, and we have an association scheme. In fact, we can relax these conditions slightly. 

\begin{theorem} \label{thm:distdesscheme1}
If $X$ is an $s$-distance set and a $2(s-1)$-design, then $\cA$ is an association scheme.
\end{theorem}

\begin{proof} 
Since $X$ is a $2(s-1)$-design, $E_0,\ldots,E_{s-1}$ are linearly independent. We claim that $I$ is independent of $E_0,\ldots,E_{s-1}$. For, suppose that $I$ is a linear combination of $E_r$, say
\[
I = \sum_{r=0}^{s-1}c_rE_r.
\]
Examining the $(a,b)$-entry of $I$, we see that for every $\al_i$, $i > 0$,
\[
\sum_{r=0}^{s-1} c_r g_r(\al_i) = 0.
\]
But this implies that $\sum_{r=0}^{s-1}c_rg_r$, a polynomial of degree at most $s-1$, has $s$ distinct zeros. By contradiction, we conclude that $I,E_0,\ldots,E_{s-1}$ are linearly independent. These matrices therefore span $\cA$, and since they are closed under multiplication, $\cA$ an association scheme. \qed
\end{proof}

\begin{corollary}
If $X$ is any set of equiangular lines, or the lines from a set of $d+1$ mutually unbiased bases, then $\cA$ is an association scheme. \qed
\end{corollary}

If $X$ is a set of equiangular lines, then $\cA = \{I,J-I\}$ is the association scheme of a complete graph. If $X$ is a set of mutually unbiased bases, then $\cA$ is the association scheme of a complete multipartite graph. 

When $t = 2s$, we claim that $E_0,\ldots,E_s$ are the idempotents of the scheme $\cA$. To see that the idempotents sum to $I$, note that when $\abs{a^*b}^2 = \al_i$, the $(a,b)$-entry of the sum is
\[
\left(\sum_{r=0}^sE_r\right)_{a,b} = \frac{1}{|X|}\sum_{r=0}^s g_r(\al_i) = \frac{p_r(\al_i)}{|X|}.
\]
Since $|X| = \dim(\Hom(s,s)) = p_r(1)$, the diagonal entries are $1$. That the off-diagonal entries are $0$ follows from the fact that $p(x)$ is a multiple of the annihilator of the degree set of $X$, as in the proof of Corollary \ref{cor:distbnddes}.

In the following, let $p_{ii}(0)$ denote the valency of $A_i$ (the sum of the entries of any row of $A_i$). 

\begin{lemma} \label{lem:idemproj}
If $X$ is an $(s+r)$-design, then
\[
A_iE_r = \frac{p_{ii}(0)g_r(\al_i)}{g_r(1)} E_r.
\]
\end{lemma}

\begin{comment}
This theorem should also be true when $X$ is a $2r$-design. 
\end{comment}

\begin{proof}
There is a unique polynomial of degree $s$ with a given $s+1$ fixed values. Let $f_i$ be the polynomial of degree $s$ such that $f_i(\al_j) := \de_{ij}$ (the Kronecker delta function on the degree set of $X$), and let $f_{i,a}$ be its zonal polynomial at pole $a$. Then
\nomenclature[\d]{$\de_{ij}$}{Kronecker delta function}%
\begin{align*}
(A_iE_r)_{a,b} & = \frac{1}{|X|}\sum_{\abs{a^*c}^2 = \al_i} g_r(\abs{b^*c}^2) \\
& = \ip{f_{i,a}}{g_{r,b}}_X.
\end{align*}
When $X$ is an $(s+r)$-design, this equals $\ip{f_{i,a}}{g_{r,b}}$. Now express $f_i$ in terms of the Jacobi polynomials, say
\[
f_i(x) = \sum_{l=0}^s c_l g_l(x).
\]
Recalling that $g_{r,a}$ is orthogonal to all zonal polynomials of lower degree, we then have
\begin{align*}
\ip{f_{i,a}}{g_{r,b}} & = \sum_{l=0}^s c_l \ip{g_{l,a}}{g_{r,b}} \\
& = c_r g_r(\abs{a^*b}^2).
\end{align*}
To find $c_r$, consider the diagonal entries: 
\[
(A_iE_r)_{a,a} = \frac{1}{|X|}p_{ii}(0) g_r(\al_i) = c_r g_r(1).
\]
Thus $c_r = p_{ii}(0) g_r(\al_i)/|X|g_r(1)$, and the result follows. \qed
\end{proof}

Recall from Chapter \ref{chap:configs} that an association scheme is $Q$-polynomial if each idempotent $E_r$ is a Schur polynomial of degree $r$ in $E_1$. Since $g_r$ is a polynomial of degree $r$ in $g_1$, and the entries of $E_r$ are defined in terms of $g_r$ and $\al_i \in A$, it follows that the association scheme in Theorem \ref{thm:distdesscheme1} is $Q$-polynomial.

\subsection{Gram-matrix algebras}

There is a second weighted adjacency algebra associated with certain $s$-distance sets. When $X$ has degree set $A = \{\al_1,\ldots,\al_s\}$ (and $\al_0 = 1$), define matrices $\cA' = \{A'_0,\ldots,A'_s\}$ indexed by $X$ such that
\[
(A'_i)_{a,b} := \begin{cases}
a^*b, & \abs{a^*b}^2 = \al_i; \\
0, & \text{ otherwise.}
\end{cases}
\]
Note that if $G$ is the Gram-matrix of $|X|$, then
\[
A'_i = G \circ A_i.
\]
If $0$ is in the degree set of $X$, say $\al_s = 0$, then $A'_s = 0$. Thus $\spn(\cA')$ is dimension $s$ when $0 \in A$ and dimension $s+1$ otherwise. Note that $A'_i$ is now Hermitian instead of symmetric. If $A'_iA'_j$ is also Hermitian, then $A'_i$ and $A'_j$ commute; thus when $\spn(\cA)$ is an algebra, it is commutative. Define a second set of idempotents:
\[
(E'_r)_{a,b} =  \frac{1}{|X|}(a^*b)h_r(\abs{a^*b}^2).
\]
Again $E'_r$ is now Hermitian. Each $E'_r$ is still in the span of $\cA'$. The proof of the following is almost identical to Lemma \ref{lem:desidem}.

\begin{lemma} \label{lem:desidem2}
If $X$ is a $(2e+1)$-design, then $E'_0,\ldots,E'_e$ are orthogonal idempotents.
\end{lemma}

Similarly, if $X$ is a $2e$-design, then $E'_0,\ldots,E'_e$ are linearly independent. 

If $X$ is both an $s$-distance set and a $(2s+1)$-design, then the fact that $E'_0,\ldots,E'_s$ are spanning and closed under multiplication shows that $\spn(\cA')$ is an algebra. As with the association schemes, these hypotheses can be generalized.

\begin{theorem} \label{thm:distdesscheme2}
Let $X$ be an $s$-distance set. If $X$ is also a $(2s-1)$-design, then $\spn(\cA')$ is an algebra. Alternatively, if $X$ is a $(2s-3)$-design and $0$ is in the degree set of $X$, then $\spn(\cA')$ is an algebra. 
\end{theorem}

\begin{proof}
First consider the case when $X$ is a $(2s-1)$-design and $0$ is not in the degree set of $X$. The same argument as in the proof of Theorem \ref{thm:distdesscheme2} shows that $I,E_0,\ldots,E'_{s-1}$ are linearly independent, and since these matrices are spanning in $\spn(\cA')$ and closed under multiplication, we have an algebra. Similarly, when $X$ is a $(2s-3)$-design and $0$ is in the degree set, then $I,E'_0,\ldots,E'_{s-2}$ are linearly independent, spanning, and multiplicatively closed. \qed
\end{proof}

\begin{comment}
More precisely: suppose not, and $I$ can be written as a linear combination $I = \sum_r c_rE'_r$. Then the $a,b$ entry of I is
\[
(a^*b) \sum_{r=0}^{s-1}c_rh_r(\al_i) = 0
\]
where $\al_i = \abs{a^*b}^2$. Since $a^*b \neq 0$ for all $a,b$, we have a polynomial $\sum_{r=0}^{s-1}c_rh_r$ of degree $s-1$ with $s$ zeros. Thus $I$ is not a linear combination of $E'_0,\ldots,E'_{s-1}$. If $X$ is a $(2s-1)$-design, then these matrices are orthogonal idempotents. Therefore $\spn(\cA')$ is spanned by $I,E'_0,\ldots,E'_{s-1}$, which are closed under matrix multiplication, and $\cA'$ is an algebra.
\end{comment}

\begin{corollary} \label{cor:distdesscheme2}
Let $X$ be the lines from any set of mutually unbiased bases, or a set of 
\[
|X| = \frac{d(1-\al)}{1-d\al}
\]
equiangular lines with angle $\al$ in $\cx^d$. Also let $A'_1$ be the Gram matrix of $X$. Then the span of $\cA' = \{I,A'_1\}$ is an algebra.
\end{corollary}

\begin{proof}
Any orthonormal basis is a $1$-design, and the disjoint union of $1$-designs is also a $1$-design. Therefore the lines from any set of mutually unbiased bases form a $1$-design. If $X$ is the set of equiangular lines, then equality holds in Corollary \ref{cor:ealrelbnd} and again we have a $1$-design. In either case, Theorem \ref{thm:distdesscheme2} applies. \qed
\end{proof}

As with the previous association schemes, the orthogonal idempotents $E'_r$ are the projections onto the eigenspaces of $A'_i$. 

\begin{lemma}
Let $X$ be an $s$-distance set. If $X$ is a $(s+r+1)$-design, or if $X$ is an $(s+r)$-design and $0$ is in the degree set of $X$, then
\[
A'_iE'_r = \frac{p_{ii}(0) \al_i h_r(\al_i)}{h_r(1)} E'_r.
\]
\end{lemma}

\begin{proof}
Let $f_i$ be the unique polynomial of degree $s$ such that $f_i(\al_j) := \de_{ij}$. If $0$ is in the degree set of $X$, then we may ignore the value of $f_i$ at $0$ and assume $f_i$ has degree $s-1$. Then let $f_{i,a}$ be the zonal polynomial of $f_i$ at pole $a$ in $\Harm(s+1,s)$. That is,
\[
f_{i,a}(z) = \begin{cases}
(a^*z), & \abs{a^*z}^2 = \al_i; \\
0, & \text{otherwise.}
\end{cases}
\]
Then
\[
(A'_iE'_r)_{a,b} = \ip{h_{r,b}}{f_{i,a}}_X,
\]
and the proof now follows that of Lemma \ref{lem:idemproj}. \qed
\end{proof}

\begin{comment}
The details: 
\begin{align*}
(A'_iE'_r)_{a,b} & = \frac{1}{|X|}\sum_{\abs{a^*c}^2 = \al_i} (a^*c)(c^*b)h_r(\abs{b^*c}^2) \\
& = \ip{h_{r,b}}{f_{i,a}}_X.
\end{align*}
When $X$ is an $(s+r+1)$-design, this equals $\ip{h_{r,b}}{f_{i,a}}$. Now express $f_i$ in terms of the Jacobi polynomials, say
\[
f_i(x) = \sum_{l=0}^s c_l h_l(x).
\]
Since $h_{r,a}$ is orthogonal to all zonal polynomials of lower degree, we have
\begin{align*}
\ip{h_{r,b}}{f_{i,a}} & = \sum_{l=0}^s c_l \ip{h_{r,b}}{h_{l,a}} \\
& = c_r (a^*b) h_r(\abs{a^*b}^2).
\end{align*}
To find $c_r$, consider the diagonal entries: 
\[
(A'_iE'_r)_{a,a} = \frac{1}{|X|}v_i \al_i h_r(\al_i) = c_r h_r(1).
\]
Thus $c_r = v_i \al_i h_r(\al_i)/|X|h_r(1)$, and the result follows. \qed
\end{comment}

\section{Real bounds}
Essentially all of the results in this chapter have analogues for real projective lines. Over the reals, we work in $\Hom(k)$, the set of polynomials from $\re^d$ to $\re$ which are homogenous of degree $k$. The Laplacian is the real restriction of the complex Laplacian, namely
\[
\De := \frac{\dde^2}{(\dde x_1)^2} + \ldots + \frac{\dde^2}{(\dde x_d)^2}.
\]
The group of unitary transformations is replaced with the group of orthogonal transformations. The harmonic polynomials $\Harm(k)$ are the elements $f$ of $\Hom(k)$ satisfying $\De f = 0$. 

Compared to the complex situation, harmonic functions over the reals are well-studied; see for example Axler, Bourdon, and Ramey \cite{abr} for the standard results. In particular, letting
\[
|x|^2 = x_1^2 + \ldots + x_n^2,
\]
we have the following analogue of Theorem \ref{thm:homdecomp}.
\begin{theorem}
\[
\Hom(k) = \Harm(k) \oplus |x|^2 \Hom(k-2). \qed
\]
\end{theorem}

Since the dimension of $\Hom(k)$ is ${d+k-1 \choose d-1}$, we conclude that 
\[
\dim(\Harm(k)) = {d+k-1 \choose d-1} - {d+k-3 \choose d-1}.
\]
Let $\Om$ denote the unit sphere in $\re^d$, and let $\om$ denote the unique measure which is invariant with respect to orthogonal transformations. The inner product on real functions is 
\[
\ip{f}{g} := \int_\Om f(x)g(x) d\om(x),
\]
and if $f$ is in $\Harm(k)$ and $g$ is in $\Hom(k-2)$, then $f$ and $g$ are orthogonal. Using zonal polynomials, we can get an explicit formula for the projection from $\Hom(k)$ to $\Harm(k)$. The Jacobi polynomials for the reals are (see \cite[Theorem 5.38]{abr})
\[
g_k(x) := (d+2k-2) \sum_{r=0}^{\lf r/2 \rf} (-1)^r \frac{(d+2k-2r-4)!!}{r!(k-2r)!(d-2)!!}x^{k-2r}.
\]
These are the unique polynomials such that the induced zonal polynomials $g_{k,a}$ satisfy
\[
\ip{g_{k,a}}{p} = p(a)
\]
for every $a \in \Om$ and $p(x) \in \Harm(k)$.

With these fundamentals in place, we can establish bounds for real projective $s$-distance sets and $t$-designs. Define an $s$-distance set as a set such that the degree set
\[
A := \{(x^Ty)^2 : x,y \in X, x \neq y \}
\]
has size $s$, and a $t$-design as a set $X$ such that for every $f$ in $\Hom(2t)$, 
\[
\ip{1}{f}_X = \ip{1}{f}.
\]

This type of real $t$-design is in some way a generalization of a $t$-$(v,k,\la)$ block design: if $X$ is the set of characteristic vectors of the blocks of a $t$-$(v,k,\la)$ design, and $\Om$ is the set of all $\{0,1\}$ vectors in $\re^v$ with $k$ ones, then for all $f \in \Hom(t)$ the average value of $f$ over $X$ is the same as the average value of $f$ over $\Om$.

\begin{theorem} \label{thm:realabs}
If $X$ is a real $s$-distance set, then
\[
|X| \leq \dim(\Hom(2s)),
\]
with equality if and only if $X$ is a $2s$-design. 
If $X$ is a $2t$-design, then
\[
X| \geq \dim(\Hom(2t)),
\]
with equality if and only if $X$ is an $t$-distance set. 
\end{theorem}

\begin{comment}
It is more standard to define the degree set of $X$ to be $\{x^Ty: x \neq y \in X\}$, and define $X$ to be a $t$-design if $\ip{1}{f}_X = \ip{1}{f}$ for all $f$ in $\Hom(k)$, $k \leq t$. Then the results are as follows: an $s$-distance set satisfies
\[
|X| \leq \dim(Hom(s) \oplus Hom(s-1)),
\]
while a $t$-design satisfies
\[
|X| \geq \dim(Hom(\lc t-1/2 \rc) \oplus \Hom(\lf t-1/2 \rf).
\]
There is no special role for $0$ in the degree set in this situation.
\end{comment} 

Both the absolute bounds above and the relative bounds below can be proved by restricting the complex case to the reals. 

\begin{theorem} \label{thm:realrel}
Let $X \sbs \Om$ be an $s$-distance set, and let $F(x)  = \sum c_r g_r(x)$ be a real polynomial with $c_0 > 0$. If $F(\al_i) \leq 0$ for each $\al_i \in A$ and $c_r \geq 0$ for each $r$, then
\[
|X| \leq F(1)/c_0.
\]
If $X$ is a $t$-design, $F(\al) \geq 0$ for every $\al \in A$, and $c_r \leq 0$ for $r > t$, then 
\[
|X| \geq F(1)/c_0.
\]
\end{theorem}
Just as in the complex case, by defining a set of Schur-idempotent matrices in terms of the distances in $X$, we get an algebra whose eigenvalues are the values of the Jacobi polynomials. Over the complex numbers, by considering the zonal polynomials in $\Harm(k+1,k)$ instead of $\Harm(k,k)$, we established additional results particular to when $0$ is in the degree set. Over the reals, we consider $\Harm(2k+1)$ instead of $\Harm(2k)$, and the results are similar. 

Delsarte, Goethals, and Seidel also considered the case of non-projective vectors on the real unit sphere. If we define the degree set of $X$ to be 
\[
A := \{x^Ty: x,y \in X, x \neq y \},
\]
then an $s$-distance set in $\re^d$ satisfies
\[
|X| \leq \dim(\Hom(s) \oplus \Hom(s-1)).
\]
If we define a $t$-design to be a set $X$ such that $\ip{1}{f}_X = \ip{1}{f}$ for all $f$ in $\Hom(k)$, $k \leq t$, then a $t$-design satisfies
\[
|X| \geq \dim(\Hom(\lc t-1/2 \rc) \oplus \Hom(\lf t-1/2 \rf).
\]
In this situation, the results are no stronger when $0$ is in the degree set.

\chapter{General Constructions}\label{chap:constructions}

In this chapter, we present several general constructions for $s$-distance sets, which we will eventually specialize to mutually unbiased bases and equiangular lines. We begin with difference sets in abelian groups, and then relate those difference sets to Cayley graphs. Finally, we describe two ways to obtain $s$-distance sets from linear codes: one using coset graphs, and the other mapping codewords directly to complex lines. 

The results in this chapter are new unless otherwise noted, although specific instances of these constructions have been applied in several situations. Calderbank, Cameron, Kantor and Seidel \cite{ccks} used Kerdock codes to construct mutually unbiased bases, while Delsarte and Goethals \cite{dg1} used BCH codes with three non-trivial weights to construct generalized Hadamard matrices. 

\begin{comment}
\section{Equivalence}\label{sec:equiv}

Before presenting some constructions, we consider what it means for two sets of lines to be equivalent.

First, two lines $x$ and $y$ on the unit sphere in $\cx^d$ are considered the same if they are equal projectively: they span the same one-dimensional subspace. Often, we will move back and forth between projective and affine space without comment. Note that the angle $\abs{x^*y}^2$ does not depend on the choice of unit vector $x$ within the subspace $\lb x \rb$.

Two sets of complex lines $v_1,\ldots,v_n$ and $w_1,\ldots,w_n$ are \textsl{equivalent} if there is a unitary matrix $U$ which maps $V$ to $W$:
\[
\{w_1,\ldots,w_n\} = \{Uv_1,\ldots,Uv_n\}.
\]
Note that the matrix $U$ preserves angles: for any $x$ and $y$,
\[
\abs{\cip{x}{y}} = \abs{\cip{(Ux)}{(Uy)}}.
\]
For mutually unbiased bases, two set of bases $B_1,\ldots,B_k$ and $C_1,\ldots,C_k$ are equivalent if some unitary $U$ maps each $B_i$ to some $C_j$.
\end{comment}

\section{Difference sets}\label{sec:diffsets}

Let $G$ be an abelian group which is written multiplicatively. We work in the group algebra of $G$; denote the identity of $G$ by $1_G$ and identify a subset $D$ of $G$ with its sum in the algebra:
\[
D = \sum_{g \in D} g.
\]
Also let $D^{-1}$ denote the sum of the inverses of $D$:
\[
D^{-1} := \sum_{g \in D} g^{-1}.
\]
Then $DD^{-1}$ is called the set of \defn{differences} of $D$. Informally, $D$ is a difference set if $DD^{-1}$ has some sort of regular structure. For example, $DD^{-1}$ is a \textsl{$(v,k,\la)$-difference set}\index{difference set} if $G$ has size $v$, $D$ has size $k$, and 
\[
DD^{-1} = k1_g + \la(G\backslash\{1_g\}).
\]
If $\chi$ is a character of $G$, then extending linearly $\chi$ can be evaluated at any element of the group algebra. That is, if $x = \sum_{g \in G}c_g g$ for some constants $c_g$, then
\nomenclature[\x]{$\chi$}{character of a group}%
\[
\chi(x) := \sum_{g \in G}c_g \chi(g).
\]

\begin{lemma} \label{lem:diffdist}
Let $G$ be an abelian group of size $v$ and let $D$ be a subset of $G$ of size $d$ such that $D$ generates $G$, and $DD^{-1}$ takes exactly $s$ distinct values on the nontrivial characters of $G$. Then there is an $s$-distance set of size $v$ in $\cx^d$.
\end{lemma}

\begin{proof}
Let $\chi_a$ be a character of $G$ and $v_a$ the restriction of $\chi_a$ to $D$, written as a vector in $\cx^{|D|}$. Then
\begin{align*}
\cip{v_a}{v_b} & = \sum_{g \in D} \conj{\chi_a(g)}\chi_b(g) \\
& = \sum_{g \in D} \chi_{ba^{-1}}(g),
\end{align*}
where $\chi := \chi_{ba^{-1}}$ is another character of $G$. In terms of the group algebra, this sum is simply $\chi(D)$. Since 
\[
\conj{\chi(g)} = \chi(g^{-1}),
\]
the absolute value of this sum is
\[
\chi(D)\conj{\chi(D)} = \chi(DD^{-1}).
\]
Thus if $\chi(DD^{-1})$ takes only $s$ different values, then the lines $v_a$ form an $s$-distance set. \qed
\end{proof}

\begin{corollary} \label{cor:diffequi}
If $D$ is a $(v,k,\la)$-difference set, then there is a set of $(k^2-k+\la)/\la$ equiangular lines in $\cx^k$.
\end{corollary}

\begin{proof}
If $D$ is a $(v,k,\la)$-difference set, then
\[
DD^{-1} = (k-\la)1_G + \la G,
\]
and consequently $v = (k^2 -k +\la)/\la$. Now consider the value of $DD^{-1}$ evaluated at a character $\chi$. If $\chi$ is the trivial character, then $\chi(DD^{-1}) = k^2$. Note that $\chi_{ba^{-1}}$ is trivial only when $a = b$. Otherwise, $\chi(G) = 0$ and $\chi(DD^{-1}) = k-\la$. Thus for every $a \neq b$, the absolute value of the angle between $v_a$ and $v_b$ is a constant. Normalizing so that these vectors become unit vectors, we have a set of equiangular lines. \qed
\end{proof}

\section{Graphs}\label{sec:graphs}
The results of the previous section can also be described in the language of graph theory. 

When $G$ is a group and $D$ is a subset of $G$, let $\Cay(G,D)$ denote the \defn{Cayley digraph} of $G$ with connection set $D$: the graph with vertex set $G$ and arc set 
\nomenclature[a$X]{$\Cay(G,D)$}{Cayley graph}% 
\[
E := \{(x,x+d): x \in G, d \in D\}.
\]
If $D$ is inverse-closed and $0 \notin D$, then $\Cay(G,D)$ is a graph. It is a standard result (see Godsil \cite[Section 12.9]{blue}, for example) that the eigenvalues of $\Cay(G,D)$ have an explicit formula in terms of characters.

\begin{lemma} 
If $\chi$ is a character of $G$, then $\chi$ is an eigenvector of $\Cay(G,D)$ with eigenvalue $\chi(D)$.
\end{lemma}

\begin{proof}
Let $A$ be the adjacency matrix of $\Cay(G,D)$ and let $E$ be the arc set. Then
\[
(A\chi)_x = \sum_{xy \in E} \chi(y) = \sum_{d \in D} \chi(x+d) = \chi(x) \sum_{d \in D} \chi(d).
\]
Thus $A\chi = \chi(D)\chi$. \qed
\end{proof}

Note that the absolute value of $\chi(D)$ is $\chi(DD^{-1})$. Combining this result with Lemma \ref{lem:diffdist}, we get an $s$-distance set from any Cayley graph. An eigenvalue of $\Cay(G,D)$ is \textsl{nontrivial} if it is not the valency $|D|$.

\begin{theorem} \label{thm:caydist}
If $\Cay(G,D)$ is connected and has exactly $s$ nontrivial eigenvalues which are distinct in absolute value, then there is an $s$-distance set of size $|G|$ in $\cx^{|D|}$. \qed
\end{theorem}

For any digraph $\Cay(G,D)$, there is a simple graph on twice as many vertices with essentially the same eigenvalues. Let $A$ be the adjacency matrix of $\Cay(G,D)$, and consider the simple graph with adjacency matrix
\[
B := \twomat{0}{A}{A^T}{0}.
\]
Suppose $\chi$ is an eigenvalue of $A$ with eigenvector $\chi(D)$. Since $A$ is the adjacency matrix of $\Cay(G,D)$, it follows that $A^T$ is the adjacency matrix of $\Cay(G,D^{-1})$. Thus $\chi$ is also an eigenvector of $A^T$, with eigenvalue $\chi(D^{-1}) = \conj{\chi(D)}$. 

\begin{lemma} \label{lem:simpdieig}
If $A$ is the adjacency matrix of $\Cay(G,D)$ and $B$ is the adjacency matrix of the corresponding simple graph, then for each eigenvalue $\la$ of $A$, both $\abs{\la}$ and $-\abs{\la}$ are eigenvalues of $B$. 
\end{lemma}

\begin{proof}
Let $v$ be an eigenvector of $A$, so $Av = \la v$. Then $A^Tv = \conj{\la}v$. (To see this, note that it is true if $v$ is a character of $G$, and the characters of $G$ are a spanning set of eigenvectors.) It follows that the space spanned by $(0,v)^T$ and $(v,0)^T$ is a two-dimensional invariant subspace of $B$. Moreover,
\[
\twomat{0}{A}{A^T}{0}\twovec{\la v}{\abs{\la}v} = \abs{\la}\twovec{\la v}{\abs{\la} v},
\]
and
\[
\twomat{0}{A}{A^T}{0}\twovec{\la v}{-\abs{\la} v} = -\abs{\la}\twovec{\la v}{-\abs{\la} v}.
\]
Thus $\abs{\la}$ and $-\abs{\la}$ are eigenvalues of $B$. \qed
\end{proof}

Note that the graph for $B$ is bipartite and has an abelian group (namely $G$) acting regularly on the shores of the bipartition. If fact, this construction is reversible.

\begin{lemma} \label{lem:disimp}
Let $G$ be an automorphism group which acts regularly on each shore of a bipartite graph $\Ga$. Then there is a subset $D$ such that the absolute value of the eigenvalues of $\Cay(G,D)$ are the same as those of $\Ga$.
\end{lemma}

\begin{proof}
Suppose the two shores of $\Ga$ have vertex sets $Y$ and $Z$, so $G$ acts on $Y$ and $Z$. Choose $y_1 \in Y$ arbitarily, and for each $g \in G$, define $y_g$ to be $y_1^g$, the action of $g$ on $y_1$. Similarly define $z_1$ and $z_g$. Finally, let
\[
D := \{g \in G: z_g \sim y_1 \}.
\]
We claim that $\Cay(G,D)$ is the desired graph. To see this, let $B = \twomat{0}{A}{A^T}{0}$ be the adjacency matrix of $\Ga$, where the first block is indexed by $Y$ and the second by $Z$. Since $y_v$ is adjacent to $z_{v+g}$ for every $g \in D$, in fact $A$ is the adjacency matrix of $\Cay(G,D)$. The result now follows from Lemma \ref{lem:simpdieig}. \qed
\end{proof}

Lemma \ref{lem:disimp} and Theorem \ref{thm:caydist} together imply the next result.

\begin{theorem}
Let $\Ga$ be a connected, bipartite, $d$-regular graph with an abelian group $G$ acting regularly on each shore of the bipartition. If $\Ga$ has $s$ nontrivial eigenvalues which are distinct in absolute value, then there is an $s$-distance set of size $|G|$ in $\cx^d$. \qed
\end{theorem}

\section{Codes}\label{sec:codes}

In addition to designs and graphs, codes may also be used to construct complex lines. 

Let \named{$V(n,q)$}{vector space of dimension $n$ over $GF(q)$} denote the $n$-dimensional vector space over $GF(q)$ with standard basis $e_1,\ldots,e_n$. Also let $d(x,y)$ denote the {\it Hamming distance} between vectors $x$ and $y$: the number of coordinates in which $x$ and $y$ differ. Now suppose $C$ is an $(n,k)$-linear code over $GF(q)$. (That is, $C$ is a $k$-dimensional subspace of $V(n,q)$.) The \defn{coset graph} of $C$, denoted $\Ga(C)$, is the graph with the cosets of $C$ as vertices and $x+C$ and $y+C$ adjacent if some $w \in y+C$ is at Hamming distance one from $x$. If the minimum distance between any two codewords in $C$ is at least $2$, then $\Ga(C)$ is simple. The coset graph of $C$ is a Cayley graph with connection set 
\nomenclature[\c]{$\Ga(C)$}{coset graph}%
\nomenclature{$GF(q)$}{Galois field of order $q$}%
\nomenclature{$e_i$}{standard basis vector}%
\[
S := \{\al e_i + C : \al \in GF(q)^*, 1 \leq i \leq n \}.
\]
We usually assume the minimum distance in $C$ is at least $3$, so that each $\al e_i + C$ is a distinct coset.

For every $(n,k)$-linear code $C$, there is a corresponding $(n,n-k)$-linear code called the \defn{dual code}:
\[
C^\perp := \{x \in V(n,q) : x^Tc = 0 \text{ for all } c \in C\}.
\]
\nomenclature{$C^\perp$}{dual code}%
Note that $(C^\perp)^\perp = C$. As the following lemma from \cite[Section 12.9]{blue} shows, $C^\perp$ is closely related to $\Ga(C)$.

\begin{lemma} \label{lem:dualeigval}
If $c \in C^\perp$ has Hamming weight $a$, then 
\[
\la = (q-1)n-qa
\]
is an eigenvector of $\Ga(C)$.
\end{lemma}

\begin{proof}
Let $M$ be a generator matrix of $C^\perp$, so that $C^\perp$ is the row space of $M$ and $C$ is the kernel. Then let $\psi$ be the mapping from the column space of $M$ to $V(n,q)/C$ defined as follows: for any $x \in V(n,q)$, 
\[
\psi(Mx) := x+C.
\]
This mapping is well-defined, because if $x$ and $y$ are in the same coset of $C$, then $y = x+c$ for some $c \in C$, and
\[
My = M(x+c) = Mx + Mc = Mx.
\]
It follows that $V(n,q)/C$ is isomorphic to the column space of $M$. With this identification, if $\chi$ is a nontrivial character of $GF(q)$ and $a$ is in the column space of $M$, 
\[
\chi_a(x+C) := \chi(a^TMx)
\]
is a character of $V(n,q)/C$. Therefore the eigenvalues of $\Ga(C)$ are 
\[
\chi_a(S) = \sum_{i=1}^n \sum_{\al \in GF(q)^*} \chi(a^TM(\al e_i)).
\]
However, $a^TM$ is a codeword of $C^\perp$: call it $z$. In the $i$-th coordinate of $z$, we have
\[
\sum_{\al \in GF(q)^*} \chi(\al z^Te_i) = \begin{cases}
q-1, & z^Te_i = 0; \\
-1, & \text{otherwise.} 
\end{cases}
\]
Hence the eigenvalue $\chi_a(S)$ is a function of the weight of $z$. Furthermore, all eigenvalues of $\Ga(C)$ can be found this way. \qed
\end{proof}

Combining the previous lemma with Theorem \ref{thm:caydist} gives the following.

\begin{theorem}
Let $C$ be an $(n,k)$-linear code over $GF(q)$ with exactly $s$ nonzero weights, where $q > 2$ and $C^{\perp}$ has minimum distance at least $3$. Then there is an $s$-distance set of size $|C|$ in $\cx^{n(q-1)}$. \qed
\end{theorem}

In the case of a code over $GF(2)$, the characters of $V(n,q)/C$ in the proof of Lemma \ref{lem:dualeigval} take only $\pm 1$ values, so in fact the vectors are real. Additionally, the characters, when restricted to the set $S$, are not necessarily distinct vectors projectively. Let \named{$\one$}{all-ones vector} denote the all-ones vector. If $a^TM = \one$ is in $C^{\perp}$, then
\[
\chi_a(e_i+C) = \chi(\one^Te_i) = -1
\]
for each $e_i + C$. Therefore $\chi_a$, when restricted to the set $S$, is $-\one$, the same vector as $\chi_0 = \one$ up to a scalar. 

\begin{corollary} \label{cor:realbincode}
Let $C$ be an $(n,k)$-linear code over $GF(2)$ with exactly $s$ nonzero weights, where $C^{\perp}$ has minimum distance at least $3$. If $\one \notin C^{\perp}$, then there is an $s$-distance set of size $|C|$ in $\re^n$. If $\one \in C^{\perp}$, then there is an $\lf s/2 \rf$-distance set of size $|C|/2$ in $\re^n$. \qed
\end{corollary}

There is a second, more direct construction of $s$-distance sets from codes. Let $q = p^r$, let $\om$ be a primitive $p$-th root of unity, and let $\tr$ denote the trace function from $GF(q)$ to $GF(p)$. Then
\nomenclature[\z]{$\om$}{complex root of unity}
\[
\chi(a) := \om^{\tr(a)}
\]
is a character of $GF(q)$. Now define $\phi$ to be the homomorphism which takes a codeword $c = (c_1,\ldots,c_n) \in C$ to $\cx^n$:
\[
\phi(c) := (\chi(c_1),\ldots,\chi(c_n)).
\]
If $C$ is a linear code, then $\phi(C)$ is closed under Schur multiplication (that is, coordinate-wise multiplication). We will call $c$ \textsl{balanced}\index{balanced codeword} if every nonzero element of $GF(q)$ occurs the same number of times in the coordinates of $c$.

\begin{theorem} \label{thm:balancedcode}
Let $C$ be an $(n,k)$-linear code over $GF(q)$, $q > 2$, with exactly $s$ nonzero weights such that every codeword is balanced. Then $\phi(C)$ is an $s$-distance set of size $|C|$ in $\cx^n$.
\end{theorem}

\begin{proof}
Suppose $x$ and $y$ are in $C$. Since $C$ is linear, $y-x$ and \named{$\zero$}{all-zeros vector} (the all-zeros vector) are also in $C$, and 
\[
\phi(x)^*\phi(y) = \phi(\zero)^*\phi(y-x) = \one^*\phi(y-x).
\]
Therefore it suffices to consider the sum of the coordinates of $\phi(c)$, for each $c \in C$. Suppose each nonzero element of $GF(q)$ occurs $a$ times in $\phi(c)$, so the weight of $c$ is $a(q-1)$. Since the trace function is onto, each $\om^i$ with $i \neq 0$ occurs $aq/p$ times. But the $p$-th roots of unity sum to zero, and the remaining coordinates of $\phi(c)$ are $1$, so we have
\[
\one^*\phi(c) = n-qa.
\]
Thus if only $s$ distinct nonzero weights occur in $C$, then only $s$ distinct values occur in the angles of $\phi(C)$. \qed
\end{proof}

When $C$ is a code over $GF(2)$, every codeword is balanced and each $\phi(c)$ is real. Thus if $C$ has $s$ nonzero weights, then $\phi(C)$ is an $s$-distance set in $\re^n$. In this case $\phi(C)$ coincides with the construction in Corollary \ref{cor:realbincode}.

It is common for linear codes to contain $\one$, which is not a balanced codeword. We will say $c \in C$ is \textsl{near-balanced}\index{near-balanced codeword} if there is some $\al \in GF(q)$ such that every element of $GF(q)$ except $\al$ occurs the same number of times in $c$. Note that $\one$ is near-balanced, and that if $c$ is balanced then $c+\one$ is near-balanced. If every element except $0$ occurs $a$ times in $c$, then $c$ has weight $n-(q-1)a$ while $c + \one$ has weight $n-a$.

\begin{theorem} \label{thm:nearcode}
Let $C$ be an $(n,k)$-linear code over $GF(q)$ with every codeword near-balanced, $\one \in C$, and weight set
\[
\{0,n,n-a_1,n-(q-1)a_1,\ldots,n-a_s,n-(q-1)a_s \}.
\]
Then $\phi(C)$ is an $s$-distance set of size $|C|/q$ in $\cx^n$.
\end{theorem}

\begin{proof}
The proof is the same as in Theorem \ref{thm:balancedcode}, noting that since $\phi(c)$ and $\phi(c+\one)$ span the same $1$-dimensional vector space, each element of the $s$-distance set occurs $q$ times. \qed
\end{proof}

Using a near-balanced property in a class of tri-weight extended-BCH codes of length $n = p^{2m}$ and dimension $k = 3m+1$, Delsarte and Goethals \cite{dg1} produced a set of $p^m-1$ generalized Hadamard matrices of order $p^{2m}$. The connection between generalized Hadamard matrices and mutually unbiased bases will be discussed in Chapter \ref{chap:mubs}. 

\begin{corollary} \label{cor:nearcode}
Let $C$ be an $(n,k)$-linear code over $GF(2)$ with $\one \in C$ and $s$ distinct nonzero weights. Then $\phi(C)$ is an $\lf s/2 \rf$-distance set of size $|C|/2$ in $\re^n$. \qed
\end{corollary}

\subsection{Codes over $\zz{4}$}

Codes over $\zz{4}$ may also be used to construct lines with restricted angles, using tools that are the same as with finite fields: coset graphs, and a direct mapping from codewords into complex space. 

If we take $\zz{4}$ to be the set $\{-1,0,1,2\}$, then the \defn{Lee weight} of $x \in \zz{4}$ is
\[
wt(x) := \abs{x}.
\]
\nomenclature{$wt(x)$}{Lee weight}%
The \defn{Lee distance} between $x$ and $y$ is then the Lee weight of $x-y$. The Lee distance between two ``vectors" in $\zz{4}^n$ is the sum of the Lee distances of the coordinates. A \textsl{code}\index{$\zz{4}$-code} over $\zz{4}$ is a subset of $\zz{4}^n$, and a code is \textsl{linear} if it is a submodule. 

We begin with coset graphs. Assume $C$ is a $\zz{4}$-linear code. Then $\Ga(C)$ is the \textsl{coset graph} of $C$ if its vertices are the cosets of $C$, with $x+C$ and $y+C$ adjacent when they contain vectors at Lee distance $1$. If $e_1,\ldots,e_n$ denotes the standard basis for the free module $\zz{4}^n$, then as a Cayley graph $\Ga(C)$ has connection set
\[
S = \{\pm e_i: 1 \leq i \leq n\}.
\]

A linear code $C$ over $\zz{4}$ has a dual code
\[
C^\perp := \{x \in \zz{4}^n: x^Tc = 0 \text{ for all } c \in C \}.
\]
We still have $(C^\perp)^\perp = C$. A \defn{generator matrix} for $C$ is a matrix over $\zz{4}$ such that $C$ is the row space. It follows that $C^\perp$ is the $\zz{4}$-kernel of the generator of $C$. Without loss of generality, we may assume that the matrix has the form 
\[
M = \left(\begin{matrix}
I & A & B \\
0 & 2I & 2C \\
\end{matrix} \right).
\]
In this case, if $A$ has $k_1$ rows and $C$ has $k_2$ rows, then $|C| = 4^{k_1}2^{k_2}$. For more details on $\zz{4}$-linear codes, see Hammons et al. \cite{hkcss}.

Now suppose $M$ is a generator for $C^\perp$, so that $C$ is the kernel. Using the same isomorphism between the column space of $M$ and $\zz{4}^n/C$ as in Lemma \ref{lem:dualeigval}, we get the following analogous result.

\begin{lemma}
If $c \in C^\perp$ has Lee weight $a$, then
\[
\la = 2(n-a)
\]
is an eigenvalue of $\Ga(C)$. \qed
\end{lemma} 
 
\begin{comment}
\begin{proof}
As with codes over $GF(q)$, if $a$ is in the column space of $M$, then 
\[
\chi_a(x+C) := \om^{a^TMx}
\]
is a character of $\zz{4}^n/C$. Let $z$ denote $a^TM \in C^\perp$. Then the eigenvalues of $\Ga(C)$ are
\[
\chi_a(S) = \sum_{i=1}^n \om^{z^Te_i} + \om^{-z^Te_i} = 2(n_0-n_2),
\]
where $n_i$ is the number of occurences of $i$ in $z$. But $n_0 - n_2$ is $n$ minus the Lee weight of $z$. Thus the eigenvalues of $\Ga(C)$ depend only on the weights in $C^\perp$. \qed
\end{proof}
\end{comment}

Since the eigenvalues of $\Ga(C)$ depend only on the Lee weights of $C^\perp$, Theorem \ref{thm:caydist} gives the following result.

\begin{corollary}
Let $C$ be a linear code in $\zz{4}^n$ with exactly $s$ nonzero Lee weights. Then there is an $s$-distance set of size $|C|$ in $\cx^{2n}$. \qed
\end{corollary}

We now proceed to the direct mapping from $\zz{4}$-codes to complex vectors. Let $i = \sqrt{-1}$, and let
\[
\chi(x) := i^x.
\]
Then the character $\chi$ can be extended to a homomorphism from $\zz{4}^n$ to $\cx^n$: for $c = (c_1,\ldots,c_n)$ in $C$,
\[
\phi(c) := (\chi(c_1),\ldots,\chi(c_n)).
\]
We examine the angles in $\phi(C)$. As with codes over finite fields, if $C$ is linear, then $\phi(C)$ is closed under Schur multiplication. 

Suppose there are $n_j$ occurrences of $j$ in codeword $c$, for $j \in \{0,1,2,3\}$. Then the Lee weight of $c$ is
\[
wt(c) = n_1 + n_3 + 2n_2 = n - (n_0 - n_2).
\]
Similarly, the weight of $wt(c+2\one)$ is $n + (n_0 - n_2)$. Therefore
\[
wt(c+2\one) = 2n - wt(c),
\]
and if $\one \in C$, then the Lee weights of $C$ are symmetric about $n$. The weights of $c+\one$ and $c-\one$ are $n + (n_1 - n_3)$ and $n - (n_1 - n_3)$ respectively.

\begin{theorem}
Let $C$ be a linear code in $\zz{4}^n$ with $\one \in C$ and Lee weights
\[
\{0,n,2n,a_1,2n-a_1,\ldots,a_t,2n-a_t\}.
\]
Then $\phi(C)$ is an $s$-distance set of size $|C|/4$ in $\cx^n$, where $s \leq {t+2 \choose 2}$.
\end{theorem}

\begin{proof}
As with codes over finite fields, it suffices to consider the absolute value of $\one^T\phi(c)$, for each $c \in C$. Since $n_i$ is the number of occurrences of $\om^i$ in $\phi(c)$,
\begin{align*}
\abs{\one^T\phi(c)}^2 & = \abs{(n_0-n_2) + i(n_1-n_3)}^2 \\
& = (n_0-n_2)^2 + (n_1-n_3)^2 \\
& = (n-wt(c))^2 + (n-wt(c-\one))^2.
\end{align*}
Therefore, the angle depends only on the weights of $c$ and $c-\one$ (or equivalently, $c+2\one$ and $c+\one$). When $c$ is not a multiple of $\one$, each of $(n-wt(c))^2$ and $(n-wt(c-\one))^2$ can take one of $t+1$ possible values, namely $(n-a_j)^2$ for $1 \leq j \leq t$, or $0$. This leads to at most ${t+2 \choose 2}$ possible values for $(n-wt(c))^2 + (n-wt(c-\one))^2$, and so $\phi(C)$ is at most a ${t+2 \choose 2}$-distance set. Because $\phi(c)$ and $\phi(c+\one)$ are the same vector projectively, each vector occurs $4$ times in $\phi(C)$. \qed
\end{proof}

\begin{comment}
${t+2 \choose 2} = {t+1 \choose 2} + t+1$.
\end{comment}

Calderbank, Cameron, Kantor, and Seidel \cite{ccks} use this direct mapping to construct maximal sets of mutually unbiased bases from $\zz{4}$-Kerdock codes. They also use the binary version of $\phi$ to construct real mutually unbiased bases from classical Kerdock codes.

\chapter{Mutually Unbiased Bases}\label{chap:mubs}

Mutually unbiased bases have received considerable attention in the last few years, most likely because of their surprising number of connections to combinatorics. For example, Calderbank, Cameron, Kantor, and Seidel \cite{ccks} described maximal sets of bases using symplectic spreads and Kerdock codes, while Boykin, Sitharam, Tiep, and Wocjan \cite{bstw1} described them using orthogonal decompositions of $\slnc$. Several authors have made analogies between maximal sets and affine planes. Pairs of mutually unbiased bases are equivalent to complex Hadamard matrices, while triples of bases can be constructed from type-II matrices. 

Regardless of the combinatorial connections, however, little is known about how many bases actually exist in most dimensions. Recall from Corollary \ref{cor:mubbnd} that at most $d+1$ mutually unbiased bases exist in $\cx^d$. When $d$ is a prime power, equality holds. On the other hand, in dimension $6$ for example, three mutually unbiased bases have been constructed in several ways, but no-one has proved or disproved the existence of four or more. For most values of $d$, the best known construction yields far less than $d+1$ sets. 

Maximal sets of bases were first constructed for prime dimensions in 1980 by Alltop \cite{all1}, who was working in the context of communication sequences. In 1981, Ivanovic \cite{iva1} rediscovered the sets in prime dimensions and put them in the quantum setting; his construction was extended to prime powers by Wootters and Fields \cite{wf1} in 1989. Since then, several other constructions have appeared: Klappenecker and R\"otteler \cite{kr1} gave a shorter proof of the unbiasedness of Wootters and Fields' bases, and Bandyopadhyay, Boykin, Roychowdhury, and Vatan \cite{bbrv} gave yet another description of the same bases. Calderbank et al. gave their construction in 1996 in a context unrelated to quantum information. 

In this chapter, we offer a new construction of maximal sets of mutually unbiased bases in prime-power dimensions using relative difference sets and commutative semifields, or equivalently using antipodal covers of complete bipartite graphs. We then show that the resulting bases are equivalent to those of Calderbank et al., and that all other known maximal sets are encompassed by this construction. We also consider dimensions which are not prime-powers, focusing on $d = 6$.

\subsection*{Applications}

Mutually unbiased bases were introduced by Ivanovic as a tool for recovering a quantum state from a series of measurements. Let $\rho$ be a density matrix of order $d$. Since $\rho$ is a $d \times d$ Hermitian matrix with trace $1$, it is specified by $d^2-1$ real parameters. Now suppose $\rho$ is measured with respect to an orthonormal basis. The resulting state of the measurement is one of the $d$ basis elements, each of which occurs with a certain probability. These probabilities sum to $1$, so they have $d-1$ degrees of freedom. Thus at least $d+1$ different measurements are required to determine $\rho$ completely from measurement statistics.

Ivanovic showed that $d+1$ mutually unbiased bases are sufficient to reproduce $\rho$, and Wootters and Fields showed that unbiased bases are the optimal measurements in terms of statistical error. More precisely, in theory $\rho$ can be reconstructed from any $d+1$ complete measurements. In practice, this is done by preparing $\rho$ and then measuring it in each basis a finite number of times. Since we are finding probabilities by measuring a finite number of events, the results will be only approximate. This error is minimized when the bases are unbiased. 

\begin{comment} 
As with equiangular lines, $\rho = \sum_{ij}x_{ij}E_{ij}$, where $E_{ij}$ is the projection for $v_{ij} \in B_i$. Then $p_{ij} = x_{ij} + \sum_{i' \neq i} x_{i'j'}/d$, from which we recover $p_{ij}$. We use orthonormal bases here because that's how people normally measure.
\end{comment}  

More recently, mutually unbiased bases have been used in quantum cryptography. The classic BB84 protocol \cite{BB84} uses a pair of mutually unbiased bases in $\cx^2$ to distribute a cryptographically secure bit-string. This protocol can easily be generalized to use $q$ bases in $\cx^q$ to distribute a string on $q$ symbols. However, other protocols using mutually unbiased bases have also been developed (see Nikolopoulos and Alber \cite{na1} for a review). As well, there are applications to quantum fingerprinting (see Scott, Walgate, and Sanders \cite{sws}) and quantum tomography (Pittenger and Rubin \cite{pr1} and Gibbons, Hoffman, and Wootters \cite{ghw}).

\begin{comment}
In the protocol, Alice repeated transmits a random vector from $B_1 \cup B_2$ to Bob, and Bob randomly measures in either $B_1$ or $B_2$. If Bob choose the correct basis, then the resulting vector is the one Alice transmitted; otherwise, he ends up with a useless random vector. After transmission, Alice and Bob use a classical channel do decide which transmissions were measured in the correct basis. By testing a few of those vectors, Alice and Bob can be assured that the vectors were not tampered with or observed, and the remaining vectors are used to construct a key.

Quantum fingerprinting: like with equiangular lines, except assume you want to measure with respect to an orthonormal basis because that's how it's always done.

Quantum tomography: reconstructing $\rho$ by "slices" in the discrete setting. The slices are parallel lines in an affine plane, assigned to states. The states from parallel lines are orthonormal, with all states together forming MUBs. The Wigner matrix at a point $A(p)$ is the sum of $vv^*$, $v$ a line on $p$, so that different points have orthogonal matrices. To find a point $A(p)$, measure it with respect to each of the bases, and then add.
\end{comment}

\subsection*{Preliminaries}
Let $\scr{B} = \{B_0,\ldots,B_m\}$ be a collection of bases of $\cx^d$. Recall that $\scr{B}$ is mutually unbiased\index{mutually unbiased bases} if each $B_i$ is orthonormal and there is some constant $\al$ such that for $u$ and $v$ in different bases,
\[
\abs{\ip{u}{v}}^2 = \al.
\]
It is convenient to write the elements of a basis as the columns of a matrix. Then $B_i$ is orthonormal if and only if the matrix $B_i$ is unitary, and $B_i$ and $B_j$ are mutually unbiased if and only if the matrix $B_i^*B_j$ is \defn{flat}: all entries of $B_i^*B_j$ have the same absolute value.

By applying unitary transformations, we may assume without loss of generality that $B_0 = I$. If $B_i$ is unbiased with $B_0$, then all entries of $B_i$ have absolute value $\sqrt{\al}$. But $B_i$ is unitary, so $B_i^*B_i = I$, which implies that 
\[
\al = \frac{1}{d}.
\]
A flat matrix \named{$H$}{Hadamard matrix} satisfying $H^*H = dI$ is sometimes called a \defn{complex Hadamard matrix}.
We may also assume without loss of generality that the first row of each flat $B_i$ ($i \neq 0$) is the all-ones vector.

In Corollary \ref{cor:distdesscheme2}, we noted that the lines from any set of mutually unbiased bases form a $1$-design, and if $G$ is the Gram matrix, then $\{I,G\}$ is a coherently-weighted configuration. One further condition comes from Corollary \ref{cor:mubbnd}: the lines from $d+1$ mutually unbiased bases in $\cx^d$ form a $2$-design.

\begin{comment}
\begin{lemma}\label{lem:mubgram}
Let $G$ be the Gram matrix of any set of mutually unbiased bases. Then $\{I,G\}$ is a configuration.
\end{lemma}

\begin{proof}
Let $B_0,\ldots,B_m$ be mutually unbiased bases. Then $B_i^*B_i = B_iB_i^* = I$, and $B_i^*B_j$ is complex Hadamard for $i \neq j$. The Gram matrix is
\[
G = \left(\begin{matrix}
I & B_0^*B_1 & \ldots & B_0^*B_m \\
B_1^*B_0 & I & \ldots & B_1^*B_m \\
& \vdots & & \\
B_m^*B_0 & B_m^*B_1 & \ldots & I \\
\end{matrix} \right),
\]
and it is then easy to verify that
\[
G^2 = (m+1)G.
\]
Thus the span of $\{I,G\}$ is an algebra. \qed
\end{proof}
\end{comment}

\section{A construction for prime powers}\label{sec:prime}

In Theorem \ref{thm:caydist}, we showed if $X(G,D)$ is a Cayley digraph $\Cay(G,D)$ with $s$ distinct absolute values of nontrivial eigenvalues, then the characters of $G$ restricted to $D$ form an $s$-distance set in $\cx^{|D|}$. In the case of mutually unbiased bases, the relevant graph is a certain type of distance-regular graph called an antipodal cover of $K_{n,n}$. Equivalently, mutually unbiased bases can be constructed from semi-regular relative different sets, which can be found using commutative semifields. 

These observations are new, although the resulting maximal sets of bases were previously constructed by Calderbank, Cameron, Kantor, and Seidel \cite{ccks} using symplectic spreads. 

\subsection{Relative difference sets}
Recall that $D \sbs G$ is a $(v,k,\la)$-difference set if 
\[
DD^{-1} = k1_G + \la(G\backslash\{1_G\}),
\]
and that these sets produce equiangular lines, as in Corollary \ref{cor:diffequi}. For mutually unbiased bases, we need a different type of set. A \defn{relative difference set} in $G$ is a subset $D$ such that for some normal subgroup $N$,
\[
DD^{-1} = |D|1_G + \la(G\backslash N).
\]
If $|N| = n$, $|G| = mn$ and $|D| = k$, then $D$ is called a \textsl{$(m,n,k,\la)$-relative difference set}, and $N$ is called the \defn{excluded subgroup}. If $m = k$, then $D$ is \textsl{semi-regular}\index{relative difference set!semi-regular}. In this case, $k = \la n$.

\begin{lemma}\label{lem:rdsmub}
Let $D$ be a semi-regular $(k,n,k,\la)$-relative difference set in an abelian group $G$. Then there are $n+1$ mutually unbiased bases in $\cx^k$.
\end{lemma}

\begin{proof}
By Lemma \ref{lem:diffdist}, it suffices to show that the characters of $G$ evaluated at $DD^{-1}$ have absolute value $0$, $\sqrt{k}$, and $k$. This is a standard result (see Beth, Jungnickel and Lenz \cite[Lemma 10.9]{bjl}) which we include for completeness.

Let $\chi$ be character of $G/N$. Then $\chi$ induces a character of $G$ which is constant on the cosets of $N$: for $a \in G$ and $n \in N$, define
\[
\chi(a+n) := \chi(a+N).
\]
These characters form a subgroup of the characters of $G$, which we denote $H$. Now evaluating the characters of $G$ at $DD^{-1}$, we get
\[
\chi(DD^{-1}) = k\chi(1_G) + \la\chi(G\backslash N).
\]
If $\chi$ is the trivial character $\chi_1$, then this sum is $k^2$. If $\chi$ is in $H$, then $\chi$ is trivial on $N$, and so $\chi(G\backslash N) = -n$. If $\chi$ is not in $H$, then $\chi(G) = \chi(N) = 0$. Therefore,
\[
\chi(DD^{-1}) = \begin{cases}
k^2, & \chi = \chi_1; \\
0, & \chi \neq \chi_1, \chi \in H; \\
k, & \chi \notin H.
\end{cases}
\]
Consequently the $k$ characters in each coset of $H$ are orthogonal when restricted to $D$. Those $n$ cosets, in conjunction with the standard basis, form a set of $n+1$ mutually unbiased bases. \qed
\end{proof}

For a survey of semi-regular relative difference sets in abelian groups, see Davis and Jedwab \cite{dj1}.

\begin{comment}
Doesn't actually suffice to show the characters take only $3$ values: orthogonality is also needed. To show that the characters of an abelian group are also a group, use the decomposition of abelian groups into a product of cyclics.

Most of the known semi-regular relative difference sets in abelian groups have prime-power $k$; the others have $n=2$. 
\end{comment}

\subsection{Commutative semifields}

Informally, a semifield is a field in which multiplication need not be associative. Formally, a finite set $E$ with operations $+$ and $\circ$ is a \defn{semifield} if 
\begin{enumerate}[(a)]
\item $(E,+)$ is an abelian group;
\item $(E,\circ)$ has an identity $1$;
\item if $x \circ y = 0$, then either $x = 0$ or $y = 0$; and
\item $\circ$ is left and right distributive over $+$.
\end{enumerate}
\nomenclature[\]{$x \circ y$}{semifield multiplication}%
Since there are no zero divisors in $E$, the additive subgroup of $E$ generated by $1$ is a finite field of prime order, say $GF(p)$. Since $E$ is an additive group and multiplication by $GF(p)$ distributes over addition, $E$ is a vector space over $GF(p)$. Thus the order of $E$ is a prime power. For a survey of finite semifields, see Cordero and Wene \cite{cw1}.

\begin{comment}
Knuth constructed a semifield of every order $2^{mn}$, $n > 1$ odd. Dickson's semifields have order $p^{2n}$, $n > 1$ odd. Some of Albert's twisted fields are commutative semifields of order $p^n$, $n > 2$. Other infinite classes exist. Semifields are known to exist for every $p^n$ with $n > 2$ and $p^n > 8$. For order $16$, there are 23 proper semifields and two isotopy classes. Only five proper semifields (one isotopy class) exist for order $27$. 

A Dickson semifield is constructed as follows. Let $q = p^n$, $p$ odd, $n > 1$, and let $E$ be a vector space of dimension $2$ over $GF(q)$ with basis $\{1,\la\}$. Let $x$ be a non-square in $GF(q)$ and let $\sg$ be a nontrivial automorphism: $a^\sg = a^{p^r}$. Then multiplication is defined by
\[
(a + \la b) \circ (c + \la d) := (ac + x(bd)^\sg) + \la(ad + bc).
\]
Clearly we have distributivity and identity $1+\la 0$. To show it is a division ring, assume the product is $0$ and plug $a = xb^\sg d^\sg c^{-1}$ into $ad + bc = 0$. This implies $xb^{\sg-1}d^{\sg+1} + c^2 = 0$, a contradiction because $x$ is non-square.
\end{comment} 

Let $E$ be a finite semifield. We construct an incidence structure with points $(x,y) \in E^2$ and lines defined by 
\[
[m,z] := \{(x,m \circ x + z): x \in E \}
\]
\nomenclature[\]{$[m,z]$}{line with slope $m$ and $y$-intersect $z$}%
for $m$ and $z$ in $E$. This is the affine plane coordinatized by $E$, with one parallel class of lines (the one with infinite slope) removed. We construct a difference set on an automorphism group of this structure.

Let $T_{a,b}$ be the map on points defined by 
\[
T_{a,b}(x,y) := (x+a,y+b).
\]
Then $T_{a,b}$ is an automorphism: for a point on the line $[m,z]$,
\begin{align*}
T_{a,b}(x,m \circ x + z) & = (x+a,m \circ x + z+ b) \\
& = (x+a,m \circ (x+a) - m \circ a + z + b).
\end{align*}
Thus $T_{a,b}$ maps $[m,z]$ to $[m,z+b-m \circ a]$. Similarly, if we define $S_{a,b}$ on lines such that
\[
S_{a,b}([m,z]) := [m+a,z+b],
\]
the $S_{a,b}$ is an automorphism mapping $(x,y)$ to $(x,y+b+a \circ x)$. In general the group of automorphisms generated by $T_{a,b}$ and $S_{a,b}$ is not abelian, but it has a large abelian subgroup. Define 
\[
G_{a,b} := T_{a,b}S_{a,0}.
\]
Then let $G$ be the set of all $G_{a,b}$ and let $D$ be the subset of $G$ with $b = 0$. The following result is due to Hughes \cite{hug1} in 1956.

\begin{theorem}
If $E$ is a commutative semifield, then $G$ is an abelian group and $D$ is an $(|E|,|E|,|E|,1)$-relative difference set.
\end{theorem}

\begin{proof}
From the definition of $G_{a,b}$, we have
\[
G_{a,b}(x,y) = (x+a,y+b+a \circ x).
\]
We first show $G$ is a group.
\begin{align*}
G_{a,b}G_{c,d}(x,y) & = G_{a,b}(x+c,y+d+c\circ x) \\
& = (x+c+a,y+d+c \circ x + b+ a \circ(x+c)) \\
& = G_{a+c,b+d+a \circ c}(x,y).
\end{align*}
Therefore $G_{a,b}G_{c,d} = G_{a+c,b+d+a \circ c}$, so $G$ is an abelian group when multiplication is commutative in $E$.

Next we show $D$ is a relative difference set. Note that the inverse of $G_{a,0}$ in $G$ is $G_{-a,a \circ a}$. Then an arbitrary element of $DD^{-1}$ is of the form
\[
G_{a,0}G_{-b,b \circ b} = G_{a-b,-(a-b)\circ b}.
\]
If we let
\[
N := \{G_{0,b} : b \in E\},
\]
then no element of $N$ occurs in $DD^{-1}$ except $G_{0,0}$. Furthermore, every element of $G\backslash N$ occurs exactly once. For, suppose $G_{a-b,-(a-b)\circ b} = G_{c-d,-(c-d) \circ d}$. Then $a-b = c-d$, and substituting into the second index, 
\[
-(a-b) \circ b = -(a-b) \circ d.
\]
We conclude that either $G_{a-b,-(a-b)\circ b} = G_{0,0}$, or $(a,b) = (c,d)$. \qed
\end{proof}

\begin{corollary} \label{cor:semimub}
If $E$ is a finite commutative semifield of order $q$, then there is a set of $q+1$ mutually unbiased bases in $\cx^q$. \qed
\end{corollary}

By finding the characters of the group $G$, we can give explicit formulas for the mutually unbiased bases arising from Corollary \ref{cor:semimub}.

We consider the case when $E$ is a semifield of odd order first. Then $E$ is a vector space over $GF(p)$; let $a^Tx$ denote the standard scalar product from $E \times E$ to $GF(p)$, and let $\om$ be a $p$-th root of unity.

\begin{lemma} \label{lem:oddsemichar}
When $|E|$ is odd, the characters of $G$ have the form
\[
\chi_{y,z}(G_{a,b}) = \om^{\scp{z}{(a^{\circ 2}-2b)}+2\scp{y}{a}}
\]
for $a,b,y,z \in E$.
\end{lemma}

\begin{proof}
\begin{align*}
\chi_{y,z}(G_{a,b})\chi_{y,z}(G_{c,d}) & = \om^{\scp{z}{(a^{\circ 2}-2b)}+2\scp{y}{a}}\om^{\scp{z}{(c^{\circ 2}-2d)}+2\scp{y}{c}} \\
& = \om^{\scp{z}{((a+c)^{\circ 2}-2(b+d+a \circ c))}+2\scp{y}{(a+c)}} \\
& = \chi_{y,z}(G_{a+c,b+d+a \circ c}). \tag*{\sqr53}
\end{align*}
\end{proof}

In fact, the scalar product $a^Tx$ can be replaced with any nondegenerate bilinear form. When $E$ is a field, let $\tr$ denote the trace function from $E$ to $GF(p)$.

\begin{comment}
Bilinear form: bilinear, from $V \times V$ to its underlying field. Nondegenerate: if $B(x,y) = 0$ for all $y$, then $x = 0$ and similarly $B(y,x)$. If the form is degenerate, then the characters will not be distinct.
\end{comment}

\begin{corollary} \label{cor:oddfieldchar}
When $E$ is a finite field of odd order, the characters of $G$ have the form 
\[
\chi_{y,z}(G_{a,b}) = \om^{\tr(z(a^2-2b)+2ya)}. \qed
\]
\end{corollary}

Next suppose $E$ is a semifield of even order. Here $E$ is a vector space over $GF(2)$, but we ``lift" to a ring over $\zz{4}$. Let $e_1,\ldots,e_m$ be the standard basis for $E$ over $GF(2)$, and let $\widehat{e}_1,\ldots,\widehat{e}_m$ be the standard basis for a free module \named{$R$}{ring} over $\zz{4}$. Then we can embed every element of $E$ into $R$ as follows: if $x = \sum_j x_je_j$ with $x_j \in \zz{2}$, then
\[
x \mapsto \widehat{x} := \sum_{j=1}^m x_j \widehat{e}_j.
\]
Any element of $R$ can be written uniquely in the form $\widehat{x} + 2\widehat{y}$, for some $x$ and $y$ in $E$. Note that the embedding map does not preserve addition; however, 
\[
2(\widehat{x+y}) = 2(\widehat{x} + \widehat{y}).
\]
Multiplication in $R$ is defined by 
\[
\widehat{e}_j\widehat{e}_k := \widehat{e_j \circ e_k}
\]
for basis elements and then extended linearly to all of $R$. It follows that multiplication distributes over addition. Again multiplication is not preserved by the embedding map, but
\[
2(\widehat{x \circ y}) = 2\widehat{x}\widehat{y}.
\]
Finally, since $\widehat{x} + \widehat{y} = \widehat{x+y} + 2\widehat{z}$ for some $z \in E$, we find that 
\[
(\widehat{x}+ \widehat{y})^2 = (\widehat{x+y} + 2\widehat{z})^2 = (\widehat{x+y})^2.
\]
We can now write down the characters of $G$.  If $\widehat{x} = \sum_j x_j \widehat{e}_j$ and $\widehat{y} = \sum_j y_j \widehat{e}_j$, then 
\[
\scp{\widehat{x}}{\;\widehat{y}} := \sum_{j=1}^m x_jy_j
\]
is a bilinear map from $R \times R$ to $\zz{4}$. Let $i$ be a primitive $4$-th root of unity.

\begin{comment}
By a ``ring over $\zz{4}$" we mean ring $R$ that is a module over $\zz{4}$: $R$ is an abelian group, and multiplication by the ring $\zz{4}$ distributes over addition.
\end{comment}

\begin{lemma} \label{lem:evensemichar}
When $|E|$ is even, the characters of $G$ have the form
\[
\chi_{y,z}(G_{a,b}) = i^{\scp{\widehat{z}}{(\widehat{a}^2-2\widehat{b})}+2\scp{\widehat{y}}{\widehat{a}}}
\]
for $a,b,y,z \in E$.
\end{lemma}

\begin{proof}
With the observations above, the proof is the same as in Lemma \ref{lem:oddsemichar}. \qed
\end{proof}

Suppose $E$ is the finite field $GF(2^m)$. Let $\tha$ be a primitive element of $E$, and assume $\tha$ has minimal polynomial $f(x) \in \zz{2}[x]$. It follows from Hensel's Lemma (see \cite[Theorem 13.4]{mcd1}) that there is a unique ``lifting polynomial" $h(x) \in \zz{4}[x]$ such that $f(x) = h(x) \mod 2$ and $h(x)$ divides $x^{2^m-1}-1 \mod 4$. Then $R$ is the \defn{Galois ring} defined as follows:
\[
GR(4^m) := \zz{4}[x]/\lb h(x) \rb.
\]
\nomenclature{$GR(4^m)$}{Galois ring of order $4^m$}%
If $\xi$ is a root of $h(x)$ in $R$, then the set
\[
T := \{0,1,\xi,\ldots,\xi^{2^m-2}\}
\]
\nomenclature{$T$}{Teichm\"uller set}%
is call the \defn{Teichm\"uller set} of $R$. Every element of $R$ is congruent mod $2$ to exactly one element of $T$. Therefore, we may take $T$ to be the embedding of $E$ in $R$. That is, for each $x \in GF(2^m)$, define $\widehat{x}$ to be the unique element in the Teichm\"uller set of $GR(4^m)$ such that 
\[
\widehat{x} = x \mod 2.
\]
If $z = \widehat{x} + 2\widehat{y}$ for $\widehat{x}$ and $\widehat{y}$ in $T$, then the \defn{Galois ring trace} is defined by
\[
\tr(z) = \widehat{x} + 2\widehat{y} + \widehat{x}^2 + 2\widehat{y}^2 + \ldots + \widehat{x}^{2^{m-1}} + 2\widehat{y}^{2^{m-1}}.
\]
This is a linear map from $R$ to $\zz{4}$. For more details on Galois rings, see Hammons, Kumar, Calderbank, Sloane, and Sol\'e \cite{hkcss} or McDonald \cite{mcd1}.

\begin{corollary}\label{cor:evenfieldchar}
When $E$ is a field of even order, the characters of $G$ have the form 
\[
\chi_{y,z}(G_{a,b}) = i^{\tr(\widehat{z}(\widehat{a}^2-2\widehat{b})+2\widehat{ya})}. \qed
\]
\end{corollary}

In each of the above cases, since the relative difference set is $\nobreak{\{G_{a,0}: a \in E\}}$, the mutually unbiased bases are the characters $\chi_{y,z}$ restricted to the set with $b = 0$ and suitably normalized. For example, when $E$ is an odd semifield, the bases have matrix form
\[
(W_z)_{a,y} :=  \frac{1}{\sqrt{q}}\om^{z^T(a \circ a)+2y^T a},
\]
for $a$, $y$ and $z$ in $E$.

\subsection{Symplectic spreads}

There is a direct correspondence between commutative semifields and symplectic spreads, due to Kantor \cite{kan1}. As a result, our mutually unbiased bases can also be constructed from spreads: this was done by Calderbank, Cameron, Kantor, and Seidel \cite{ccks} in 1996.

Let $V$ be a vector space of dimension $n$ over $GF(q)$. A \defn{spread} in $V^2$ is a collection of $n$-dimensional subspaces $U_0,\ldots,U_{q^n}$ such that 
\[
\{U_i-\{0\} : 0 \leq i \leq q^n \}
\]
is a partition of $V^2-\{0\}$. Let $\col(M)$ denote the column space of a matrix $M$. Then using unitary transformations, we can assume without loss of generality that
\nomenclature[a$c]{$\col$}{column space}%
\[
U_0 = \col\twovec{0}{I},
\]
and for $i \geq 1$,
\[
U_i = \col\twovec{I}{M_i}.
\]
Also without loss of generality, $M_1 = 0$. Then $U_i$ and $U_j$ have a trivial intersection if and only if 
\[
\twomat{I}{I}{M_i}{M_j} 
\]
is invertible, which occurs if and only if $M_i-M_j$ is invertible. Thus, we are looking for a collection of $q^n$ matrices of order $n$ with invertible differences. For this reason a spread refers to either a collection of subspaces or the corresponding collection of matrices. 

Spreads can be used to construct affine planes in the same manner as semifields. The line $y = m \circ x + b$ with elements from a semifield is replaced with $y = Mx + b$, with $M$ from a spread and $y$, $x$, and $b$ from $V$. 

\begin{lemma}\label{lem:semispread}
Let $E$ be a semifield of order $q^n$, a vector space of dimension $n$ over $GF(q)$. For each $a$ in $E$, let $M_a$ be the $GF(q)$-linear transformation corresponding to multiplication by $a$. Then $\{M_a: a \in E\}$ forms a spread.
\end{lemma}

\begin{proof}
The linear transformation $M_a$ is the unique matrix such that 
\[
M_ax = a \circ x
\]
for every $x \in E$. Since
\[
M_ax - M_bx = ax - bx  = (a-b) \circ x = M_{a-b}x,
\]
we see that $M_a - M_b = M_{a-b}$. Thus $M_a - M_b$ is invertible for every $a \neq b$. \qed
\end{proof}

Conversely, given an additively closed spread on $V$ containing the identity, associate each element $a \in V$ with a spread matrix $M_a$. Choose $M_0 = 0$ and $M_1 = I$. Then, with multiplication defined by
\[
a \circ x := M_a x,
\]
$V$ becomes a semifield. 

\begin{comment}
The bijection between commutative semifields and symplectic semifield spreads due to Kantor is slightly more complicated. Let $\{v_1,\ldots,v_n\}$ be a basis for $E$, and assume multiplication on the left by $v_i$ is given by the matrix $M_i$:
\[
v_i \circ v_j = \sum_k (M_i)_{kj}v_k.
\]
Thus the set of entries $\{(M_i)_{kj} \}$ determines a semifield. However these entries also determine a spread, namely the span of the matrices $M_i$. In general a set of entries $\{(M_i)_{jk}\}$ determines a spread if and only if the columns $(M_i)_k$ are all linearly independent. Therefore if $\{(M_i)_{jk}\}$ determines a spread, so does $\{(M_k)_{ji}\}$. Call this the ``dual spread" of $M$. 

Now the semifield of $(M_i)_{jk}$ is commutative if and only if $(M_i)_{jk} = (M_j)_{ik}$. This occurs if and only if $(M_k)_{ji} = (M_k)_{ij}$ in the dual spread; that is, if and only if the dual spread is symmetric. This establishes the bijection of Kantor.

Alternatively, we can see the bijection between symplectic semifield spreads and commutative semifields from the later section on equivalence of mutually unbiased bases.

Semifields are isotopic (which is defined in terms of linear transformations from on to the other) if and only if their projective planes are isomorphic.
\end{comment}

Now let $B$ be a bilinear form on a vector space. Then $B$ can be represented as a matrix: $B(x,y) = x^TBy$. We call $B$ \textsl{alternating}\index{alternating form} if $B(x,y) = -B(y,x)$; equivalently, the matrix of $B$ is skew-symmetric. There is a non-degenerate alternating bilinear form on a vector space $W$ if and only if it has even dimension, say $W = V^2$. A subspace $U_i \sbs V^2$ is \defn{totally isotropic} if $B$ vanishes on $U_i$. That is, $B(x,y) = 0$ for every $x$ and $y$ in $U_i$.

\begin{comment}
If $V^2$ is a vector space over $GF(2)$, then $B$ can be constructed from a quadratic form. A function $Q: V^2 \rightarrow GF(q)$ is a \defn{quadratic form} if
\begin{enumerate}[(a)]
\item $Q(\la x) = \la^2 Q(x)$ for all $\la \in GF(q), x \in V$, and \\
\item $B(x,y) := Q(x+y)-Q(x)-Q(y)$ is a bilinear form.
\end{enumerate}
Clearly $B$ is symmetric, and therefore over $GF(2)$ it is alternating. If $Q$ vanishes on $V^2$, then so does $B$.
\end{comment}

A spread in $V^2$ is \textsl{symplectic}\index{spread!symplectic} is there is a nondegenerate alternating bilinear form for which every subspace in the spread is totally isotropic. As an example, consider the spread from a commutative semifield $E$ in Lemma \ref{lem:semispread}. Let $a^T b$ denote the standard vector space scalar product for $E$, and define a bilinear form on $E^2$ by
\[
[(a_1,b_1),(a_2,b_2)] := a_2^Tb_1 - a_1^T b_2.
\]
For each matrix $M_i$ representing multiplication by $a$ in $E$, the corresponding subspace $U_i$ in the spread has elements of the form $(x,M_ix)$. Then $U_i$ is isotropic if and only if, for all $x$ and $y$ in $E$,
\[
[(x,M_ix),(y,M_iy)] = y^T(M_i - M_i^T)x = 0.
\]
Thus the spread from Lemma \ref{lem:semispread} is symplectic if and only if each $M_i$ is symmetric.

We now describe the construction of Calderbank, Cameron, Kantor, and Seidel. Let $V = GF(p)^m$, and consider a vector space of dimension $|V|$ over $\cx$ with standard basis $\{e_v: v \in V\}$. Then for each $a \in V$, define the $|V| \times |V|$ \defn{generalized Pauli matrices} by the following linear maps:
\begin{align*}
X(a): e_v \mapsto e_{v+a}, \\
Y(a): e_v \mapsto \om^{v^T a} e_v,
\end{align*}
\nomenclature{$X(a),Y(a)$}{generalized Pauli matrices}%
where $\om$ is a $p$-th primitive root of unity. We work with the group 
\[
G = \lb X(a),Y(a): a \in V \rb / \lb\om I \rb,
\]
which has size $p^{2m}$. This group is abelian, and every element of $G$ can be written uniquely in the form 
\[
X(a)Y(b)\lb\om I \rb
\]
for some $a,b \in V$. Then $V^2$ is isomorphic to $G$ as a vector space via the following map:
\[
\phi:(a,b) \mapsto X(a)Y(b)\lb\om I \rb.
\]
For,
\begin{align*}
\phi(a_1+a_2,b_1+b_2) & = X(a_1+a_2)Y(b_1+b_2)\lb\om I \rb  \\
& = (X(a_1)Y(b_1)\lb\om I \rb)(X(a_2)Y(b_2)\lb\om I \rb) \\
& = \phi(a_1,b_1)\phi(a_2,b_2).
\end{align*}
Define a bilinear form on $G$ as follows:
\[
[ \phi(a_1,b_1), \phi(a_2,b_2)] := a_2^T b_1 - a_1^T b_2.
\]
This form is nondegenerate and alternating.

\begin{lemma} \label{lem:2groupcom}
The matrices $X(a_1)Y(b_1)$ and $X(a_2)Y(b_2)$ commute if and only if $a_2^T b_1 - a_1^T b_2 = 0$.
\end{lemma}

\begin{proof}
Consider the action of the matrices on $e_v$:
\begin{align*}
X(a_1)Y(b_1)X(a_2)Y(b_2)e_v & = \om^{v^T b_2} X(a_1)Y(b_1)e_{v+a_2} \\
& = \om^{v^T b_2 + (v+a_2)^T b_1} e_{v+a_1+a_2} \\
& = \om^{a_2^T b_1} X(a_1+a_2)Y(b_1+b_2)e_v.
\end{align*}
Similarly,
\[
X(a_2)Y(b_2)X(a_1)Y(b_1) = \om^{a_1^T b_2} X(a_1+a_2)Y(b_1+b_2).
\]
The matrices coincide if and only if $\om^{a_2^T b_1} = \om^{a_1^T b_2}$. \qed
\end{proof}

Thus a set of matrices $\{X(a_i)Y(b_i)\}$ commute whenever the bilinear form vanishes on the set $\{(a_i,b_i)\}$. Since every $X(a)Y(b)$ is normal, a set of commuting matrices of that form are simultaneously diagonalizable. 

For example, the set $Y(V) := \{Y(a): a \in V\}$ is commuting, and the standard basis $\{e_a: a \in V\}$ is a complete set of orthonormal eigenvalues for $Y(V)$. Similarly, if 
\[
e_a^* := \frac{1}{\sqrt{p^m}} \sum_{v \in V}\om^{a^T v}e_v,
\]
then $\{e_a^*: a \in V\}$ is a complete set of eigenvalues for $X(V) := \{X(a): a \in V\}$. Moreover, the bases $\{e_a\}$ and $\{e_a^*\}$ are mutually unbiased. 

\begin{theorem} \label{thm:spreadmub}
Let $U_0,\ldots,U_{p^m}$ be a symplectic spread in $G$, and let $B_i$ be an orthonormal basis of eigenvalues for the matrices of $U_i$. Then $\{B_0,\ldots,B_{p^m}\}$ is a set of $p^m+1$ mutually unbiased bases in $\cx^{p^m}$. \qed
\end{theorem}

Calderbank, Cameron, Kantor, and Seidel proved this theorem by finding a large number of isomorphisms of $G$, so that every pair of $p^m$ commuting matrices are equivalent to $X(V)$ and $Y(V)$ under some isomorphism.

In fact, we can show explicitly that the mutually unbiased bases from semifields in Corollary \ref{cor:semimub} are equivalent to those from symplectic spreads in Theorem \ref{thm:spreadmub}. 

Let $E$ be a commutative semifield of order $q = p^m$, and for $a \in E$ let $M_a$ be the matrix representing multiplication by $a$.

\begin{lemma}\label{lem:compauli}
Fix $z \in E$, and for each $a \in E$, let $b_a = 2M_a^Tz$. Then the set
\[
U_z = \{X(a)Y(b_a): a \in E\}
\]
is commuting.
\end{lemma}

\begin{proof}
Recall that $X(a_1)Y(b_1)$ and $X(a_2)Y(b_2)$ commute if and only if $a_2^T b_1 = a_1^T b_2$. Here,
\[
a_2^Tb_1 = 2 a_2^T M_{a_1}^Tz = 2(M_{a_1}a_2)^Tz = 2(a_1 \circ a_2)^Tz,
\]
and similarly
\[
a_1^T b_2 = 2(a_2 \circ a_1)^Tz.
\]
Since the semifield is commutative, the expressions are equal and the matrices commute. \qed
\end{proof}

Thus the space over $GF(p)$ spanned by $U_z$ is isotropic, and the set of all such $U_z$ is a symplectic spread.

For odd order, the mutually unbiased bases in Corollary \ref{cor:semimub} have matrix form
\[
(W_z)_{x,y} := \frac{1}{\sqrt{q}}\om^{\scp{z}{(x \circ x)} + 2\scp{y}{x}},
\]
for $x,y,z \in E$. 

\begin{lemma} \label{lem:oddsemiequiv}
Let $b = 2M_a^Tz$ in $E$, where $|E|$ is odd. Then the columns of $W_z$ form a spanning set of eigenvectors for $X(a)Y(b)$.
\end{lemma}

\begin{proof}
Let $(W_z)_y$ denote the $y$-th column on $W_z$.
\begin{align*}
X(a)Y(b)(W_z)_y & = X(a)Y(b)\sum_{x \in G} \om^{\scp{z}{(x \circ x)} + 2\scp{y}{x}} e_x \\
& = \sum_{x \in G} \om^{\scp{z}{(x \circ x)} + 2\scp{(y+b)}{x}} e_{x+a} \\
& = \sum_{x \in G} \om^{\scp{z}{(x+a)^{\circ 2}} + 2\scp{y}{(x+a)} + \scp{b}{x} - 2\scp{z}{(a \circ x)} - \scp{z}{(a \circ a)} - 2\scp{y}{a}} e_{x+a}.
\end{align*}
Now since $b = 2M_a^Tz$, we have
\[
b^Tx = 2z^TM_ax = 2\scp{z}{(a \circ x)}.
\]
Therefore the previous expression simplifies to 
\begin{align*}
X(a)Y(b)(W_z)_y & = \om^{-\scp{z}{(a \circ a)} - 2\scp{y}{a}} \sum_{x \in G} \om^{\scp{z}{(x+a)^{\circ 2}} + 2\scp{y}{(x+a)}} e_{x+a} \\
& = \om^{-\scp{z}{(a \circ a)} - 2\scp{y}{a}} (W_z)_y. \inqed
\end{align*}
\end{proof} 

The even case is similar. Let $E$ have even-order, and for $x \in E$ let $\widehat{x}$ be the embedding of $E$ into the free module $R$ over $\zz{4}$, as in Lemma \ref{lem:evensemichar}. Then the bases in Corollary \ref{cor:semimub} have matrix form
\[
(W_z)_{x,y} := \frac{1}{\sqrt{q}}i^{\lb \widehat{z}, \widehat{x}^2 \rb + \lb 2\widehat{y}, \widehat{x}\rb}.
\]
\begin{lemma}\label{lem:evensemiequiv}
Let $b = M_a^Tz$ in $E$, where $|E|$ is even. Then the columns of $W_z$ form a complete set of  eigenvectors for $X(a)Y(b)$.
\end{lemma}

\begin{proof}
As in Lemma \ref{lem:oddsemiequiv}. \qed
\end{proof}

The situation in which $E$ is a field is special. In the odd case, let $x$, $y$, and $z$ be in $GF(q)$, and let
\[
(W_z)_{x,y} := \frac{1}{\sqrt{q}}i^{\tr(zx^2 + 2yx)}.
\]
Then when $b = 2az$ in $GF(q)$, the columns of $W_z$ form a complete set of  eigenvectors for $X(a)Y(b)$. In the even case, let $T$ be the Teichm\"uller set of the Galois ring $R$, and for $x$, $y$, and $z$ in $T$, let
\[
(W_z)_{x,y} := \frac{1}{\sqrt{q}}i^{\tr(zx^2+2yx)}.
\]
Then the columns of $W_z$ form a complete set of eigenvectors for $X(a)Y(b)$ when $b = az$.
 
\subsection{Covering graphs}

Any difference set determines a bipartite graph with a group of automorphisms acting on the colour classes. In the case of $(n,k,n,\la)$-relative difference sets, the graph has an interesting structure.

\begin{theorem} \label{thm:rdscover}
There exists an $(n,k,n,\la)$-relative difference set if and only if there exists an $n$-fold distance-regular cover of \named{$K_{k,k}$}{complete bipartite graph} whose automorphism group has a subgroup acting regularly on each colour class.
\end{theorem}

\begin{proof}
Let $D$ be an $(n,k,n,\la)$-relative difference set in $G$, with excluded subgroup $N$. Define a graph $\Ga$ with vertices $\zz{2} \times G$ and edges as follows:
\[
(0,x) \sim (1,y) \quad \Longleftrightarrow \quad y-x \in D.
\]
Clearly, $\Ga$ is bipartite, and the automorphism group
\[
\{ \phi_a: (i,x) \mapsto (i,x+a) \mid a \in G\}
\]
acts regularly on each colour class. 

We show that $\Ga$ is a distance-regular antipodal cover. Without loss of generality, consider the neighbourhoods of $(0,0)$. The first neighbourhood of $(0,0)$ is $\{(1,y): y \in D\}$, while $(0,x)$ is at distance $2$ if and only if there is some $y \in D$ such that $y-x$ is also in $D$. But this occurs if and only if $x$ is a difference in $D$. We have (for $x \neq 0$):
\[
d((0,0),(0,x)) = \begin{cases}
2, & x \notin N; \\
4, & x \in N.
\end{cases}
\]
Since $N$ is a subgroup of $G$, vertices at distance $4$ from $G$ are also at distance $4$ from each other: $N$ is antipodal. If $x \notin N$, then $x$ occurs $\la$ times as a difference in $D$, and therefore $(0,x)$ has $\la$ common neighbours with $(0,0)$. From this information, it follows that $\Ga$ is distance-regular with intersection array
\[
\{k,k-1,k-\la,1;1,\la,k-1,k\}.
\]
This is the intersection array of an antipodal cover of $K_{k,k}$.

The converse is similar. Let $X$ and $Y$ be the two colour classes of an antipodal distance-regular $n$-fold cover of $K_{k,k}$, and assume the group $G$ acts regularly on both $X$ and $Y$. Identify $G$ with $X$ as follows: fix some $x \in X$, and for each $u \in X$ let $g_u$ be the unique element of $G$ such that $g_u(x) = u$. Now for some fixed $y \in Y$, define
\[
D := \{g_u \in G: u \sim y\}.
\]
Using similar counting arguments, it follows that $D$ is an $(n,k,n,\la)$-difference set. If we identify $Y$ with $G$ by uniquely letting $h_v \in G$ satisfy $h_v(y) = v$ for each $v \in Y$, then the excluded subgroup is
\[
N := \{h_v \in G: d(v,y) \in \{0,4\}\}. \qed
\]
\end{proof}

Note that the graph $\Ga$ in the proof of Theorem \ref{thm:rdscover} has
\[
\phi: (i,x) \mapsto (1-i,-x)
\]
as an automorphism in addition to the automorphism acting on the colour classes. Therefore the graph corresponding to a relative difference set is vertex transitive (but not necessarily Cayley with respect to an abelian group).

\begin{comment}
Note that the neighbourhood of $v \in Y$ is $\{h_v(u): u \sim y\}$, since $h_v$ is an automorphism mapping $y$ to $v$. Now $v$ and $y$ are at distance $2$ if and only if $v \sim u \sim y$ for some $u$. This says that $h_v(w) = u$ for some $w \sim y$. But this says that $h_vg_w(x) = g_u(x)$, or $h_v = g_ug_w^{-1}$, for $g_u$ and $g_w$ in $D$.
\end{comment}

An $(n,k,n,\la)$-relative difference set is also equivalent to a symmetric transversal design $ST(n,\la)$ admitting a Singer group. See Jungnickel \cite{jun1} for details.

\begin{comment}
A $TD_\la(k,n)$ has $k$ partition classes of size $n$, blocks of size $k$, and $p, q$ contained in exactly $0$ or $\la$ common blocks depending on whether or not $p$ and $q$ are in the same cell. Symmetric means the dual structure is also a $TD$. A Singer group acts on the vertices. The bipartite distance-regular antipodal cover is the incidence graph.

A relative difference set $D$ of size $q$ in a group $G$ of size $q^2$ is equivalent to a certain type of affine plane. Take $G$ to be the points of the plane, take the develop of $D$ plus the cosets of $N$ to be the lines.
\end{comment}

\section{Equivalence of known constructions}\label{sec:mubequiv}

In this section, we consider vectors projectively: two vectors are considered the same if they span the same $1$-dimensional space. Note that the angle $\abs{x^*y}^2$ does not depend on the choice of unit vector $x$ within the subspace $\lb x \rb$. Two sets of complex lines are \textsl{equivalent}\index{equivalence of lines} if there is a unitary matrix $U$ which maps one to the other. Recall that $U$ preserves the angles between vectors:
\[
\abs{(Ux)^*(Uy)} = \abs{x^*U^*Uy} = \abs{x^*y}.
\]
For two equivalent sets of mutually unbiased bases, say $\scr{B} = \{B_0,B_1,\ldots,B_n\}$ and $\scr{B'} =\{B_0',\ldots,B_n'\}$, this means that two vectors from the same basis $B_i$ will be mapped to the same basis $B_j'$ for some $j$. Here, we show that all of the known maximal sets of mutually unbiased bases are equivalent to the ones in Section \ref{sec:prime}. The results in this section are new.

We work with bases in matrix form. Let $q = p^m$ be odd with $p > 3$, and let $\om$ be a primitive $p$-th root of unity. For $x$, $y$, and $z$ in $GF(q)$, define a $q \times q$ matrix with entries
\[
(A_z)_{x,y} := \frac{1}{\sqrt{q}}\om^{\tr((x+z)^3 + y(x+z))}.
\]
We call this an \textsl{Alltop} matrix, as it was Alltop \cite{all1} who showed that for $p > 3$, the set $\{A_z: z \in GF(q)\}$ together with $I$ is a maximal set of mutually unbiased bases. Also consider the bases constructed in Corollary \ref{cor:semimub} in the case where the semifield $E$ is just $GF(q)$:
\[
(W_z)_{x,y} := \frac{1}{\sqrt{q}}\om^{\tr(zx^2 + 2yx)}.
\]
We call $W_z$ a \textsl{Wootters \& Fields} matrix (see \cite{wf1} for their construction).

\begin{theorem}
For $p > 3$, the Alltop matrices are equivalent as mutually unbiased bases to the Wootters \& Fields matrices. 
\end{theorem}

\begin{proof}
For convenience, let 
\[
\chi(x) := \om^{\tr{x}}.
\]
Multiply each $A_a$ on the left by the unitary matrix $A_0^*$. Since $A_0^* = A_0^{-1}$, this map takes $A_0$ to $I$ and $I$ to $A_0^*$ (which, after dividing column $x$ by $\om^{\tr{x^3}}$, is $W_0$).  In the remaining cases:
\begin{align*}
\left(A_0^*A_a\right)_{x,y} & = \sum_{z \in GF(q)} \left(A_0^*\right)_{x,z}\left(A_a\right)_{z,y} \\
& = \frac{1}{q}\sum_{z \in GF(q)} \charac{-z^3-xz}\charac{(z+a)^3+y(z+a)} \\
& = \frac{1}{q}\sum_{z \in GF(q)} \charac{3a z^2 + (3a^2+y-x)z + (a^3+ya)}.
\end{align*}
This expression is known as a Weil sum and can be evaluated with the following formula from Lidl and Niederreiter \cite[Theorem 5.33]{lid1}:
\[
\sum_{z \in GF(q)} \charac{a_2z^2 + a_1z + a_0} = \charac{a_0 - \frac{a_1^2}{4a_2}} \eta(a_2)G(\eta, \chi).
\]
Here $\eta(a_2)$ is the quadratic residue of $a_2$ and $G(\eta, \chi)$ is a Gaussian sum which is independent of $a_0$, $a_1$ and $a_2$. Thus,
\[
\left(A_0^*A_a\right)_{x,y} = \frac{1}{q}\charac{\frac{12a^4 + 12ya^2 - (3a^2+y-x)^2}{12a}} \eta\left(3a\right)G(\eta, \chi).
\]
Now divide each column by its entry in the row $x = 0$, namely $(A_0^*A_a)_{0,y}$. (This does not affect the absolute value of the angle between the columns.) Most of the terms cancel. The result is
\begin{align*}
\frac{\left(A_0^*A_a\right)_{x,y}}{\left(A_0^*A_a\right)_{0,y}} & = \charac{\frac{-x^2 +2x(3a^2+y)}{12a}} \\
& = \charac{-\frac{1}{12a}x^2 + \frac{3a^2+y}{6}x} \\
& = \left(W_{-\frac{1}{12a}}\right)_{x,\frac{3a^2+y}{6}}.
\end{align*}
We conclude that pre-multiplying by $A_0^*$ maps $A_a$ to $W_{-1/12a}$, up to the column permutation $y \mapsto (3a^2+y)/6$. Thus the mutually unbiased bases are equivalent. \qed
\end{proof}

A third construction of maximal sets of bases in odd prime-power dimensions is due to Bandyopadhyay, Boykin, Roychowdhury, and Vatan \cite{bbrv}. They partition the Generalized Pauli matrices into maximal commuting sets and show that the common eigenvectors of these sets are mutually unbiased. Their partition is that of Lemma \ref{lem:oddsemiequiv}: for each $z \in GF(q)$,
\[
\{X(a)Y(2az): a \in GF(q)\}
\]
is commuting set. This implies that their construction is a special case of the bases in Corollary \ref{cor:semimub} when $E$ is a field. 

\begin{corollary}
The mutually unbiased bases of Bandyopadhyay, Boykin, Roychowdhury, and Vatan are equivalent to the Wootters \& Fields matrices. \qed
\end{corollary}

There are fewer constructions for even dimensions $q = 2^m$. Let $R = GR(4^m)$, let $T$ be the Teichm\"uller set and let $\tr: R \rightarrow \zz{4}$ be the Galois ring trace. The \textsl{Wootters \& Fields} matrices are, for $x$, $y$, and $z$ in $T$ and $i = \sqrt{-1}$,
\[
(W_z)_{x,y} := \frac{1}{\sqrt{q}}i^{\tr(zx^2+2yx)}.
\]
Since $\tr(x^2) = \tr(x)$ in $R$, the exponent $\tr(zx^2+2yx)$ can be simplified to $\tr((z+2y)x)$ (for some $z$). This description was given by Klappenecker and R\"otteler \cite{kr1}. These matrices are again equivalent to those in Corollary \ref{cor:semimub} when $E$ is a field. Lemma \ref{lem:evensemiequiv} shows that the columns of these matrices are eigenvectors for the generalized Pauli matrices, but it is instructive to see this explicitly.

Recall that for $x \in GF(q)$, there is a unique $\widehat{x} \in T$ such that $\widehat{x} = x \mod 2$. Using this bijection, the generalized Pauli matrices act on $T$. If $x$ and $y$ are in $T$, then in general $x+y$ is not. However, $T$ is closed under multiplication, so $(x+y)^2 = x^2 + y^2 + 2xy$
is in $T$, and so is its square root $x+y+2\sqrt{xy}$. Moreover, $x+y+2\sqrt{xy}$ is the unique element of $T$ congruent to $x+y$ mod $2$. Therefore, the Pauli matrices act on $T$ as follows:
\begin{align*}
X(a): e_v \mapsto e_{v+a+2\sqrt{av}}; \\
Y(a): e_v \mapsto i^{\tr(2av)}e_v.
\end{align*}

\begin{comment}
For $T$ in $GR(4^m)$, since $x^{2^m} = x$, define $\sqrt{x} := x^{2^{m-1}}$.
\end{comment}

\begin{lemma}
Let $b = az$ in $T$. Then the columns of $W_z$ form a complete set of  eigenvectors for $X(a)Y(b)$.
\end{lemma}

\begin{proof}
Let $(W_z)_y$ denote the $y$-th column on $W_z$.
\begin{align*}
X(a)Y(b)(W_z)_y & = X(a)Y(b)\sum_{x \in G} i^{\tr(zx^2 + 2yx)} e_x \\
& = \sum_{x \in G} i^{\tr(zx^2 + 2yx + 2bx)} e_{x+a+2\sqrt{xa}} \\
& = \sum_{x \in G} i^{\tr(z(x+a)^2 + 2y(x+a) + 2bx - 2zax - za^2 - 2ya)} e_{x+a+2\sqrt{xa}} \\
& = i^{-\tr(za^2 + 2ya)} \sum_{x \in G} i^{\tr(z(x+a)^2 + 2y(x+a))} e_{x+a+2\sqrt{xa}} \\
& = i^{-\tr(za^2 + 2ya)} (W_z)_y.
\end{align*}
In the second last line, $(x+a)^2 = (x+a+2\sqrt{xa})^2$ and $2(x+a) = 2(x+a+2\sqrt{xa})$. \qed
\end{proof} 

As with in the odd case, this description of the Wootters \& Fields matrices in dimension $2^m$ as the eigenvalues of generalized Pauli matrices was given by Bandyopadhay et al. \cite{bbrv}.

\section{Non-prime-power dimensions}\label{sec:nonprime}

We have seen that $n+1$ mutually unbiased bases in $\cx^k$ can be constructed from $(k,n,k,\la)$-relative difference sets, and that a $(k,k,k,1)$-relative difference set exists whenever $k$ is a prime power. For other dimensions, the largest known general construction is the following, due to R\"otteler and Klappenecker \cite{kr1}.

\begin{lemma}
Suppose there exist $n$ mutually unbiased bases in both $\cx^{k_1}$ and in $\cx^{k_2}$. Then there exist $n$ mutually unbiased bases in $\cx^{k_1k_2}$.
\end{lemma}

\begin{proof}
Let $A_1,\ldots,A_n$ be mutually unbiased bases (in matrix form) in $\cx^{k_1}$, and $B_1,\ldots,B_n$ mutually unbiased in $\cx^{k_2}$. Then for $i \neq j$, both $A_i^*A_j$ and $B_i^*B_j$ are flat matrices, and 
\[
(A_i \otimes B_i)^*(A_j \otimes B_j) = (A_i^*A_j) \otimes (B_i^*B_j) 
\]
is also flat. Thus
\[
A_1 \otimes B_1,\ldots,A_n \otimes B_n
\]
is a set of mutually unbiased bases in $\cx^{k_1k_2}$. \qed
\end{proof}

\begin{corollary} \label{cor:mubprod}
Let $d = p_1^{e_1}\ldots p_r^{e_r}$, where $p_1,\ldots,p_r$ are distinct primes. Then there exists a set of
\[
\min\{p_1^{e_1},\ldots,p_r^{e_r}\} + 1
\]
mutually unbiased bases in $\cx^d$. \qed
\end{corollary}

Wocjan and Beth \cite{wb1} have a construction which slightly improves the lower bound in certain square dimensions: they construct $n+2$ bases in dimension $d = k^2$ from a set of $n$ mutually orthogonal Latin squares of size $k \times k$. However, nothing better than Corollary \ref{cor:mubprod} is known for most $d$. The bound implies that there are at least three mutually unbiased bases in any dimension; in Section \ref{sec:typeii} we construct three in every dimension using spin models.

\begin{comment}
Wocjan and Beth use a $(k,s)$-net: a set of $ks$ blocks of size $s$ on $s^2$ points partitioned such that 
\[
|B \cap B'| = \begin{cases}
1, & B,B' \mbox{ in different cells;} \\
0, & B,B' \mbox{ in the same cell.} 
\end{cases}
\]
The nonzero entries of the characteristic vectors of each block are then replaced with columns from an $s \times s$ generalized Hadamard matrix, resulting in a total of $ks^2$ blocks and $k$ mutually unbiased bases.
\end{comment}

\subsection{Dimension $6$}

At least three and at most seven mutually unbiased bases exist in $\cx^6$: the exact number is not known. Here we consider the possibility that more than three exist.

If $B_0$ and $B_1$ are mutually unbiased, then without loss of generality $B_0 = I$ and $B_1$ is a complex Hadamard matrix. The list of known Hadamard matrices of order $6$ is short. Let $s$ and $t$ be complex numbers of absolute value $1$, let $i = \sqrt{-1}$, and let $\om$ be a primitive third root of unity. Then
\begin{equation} \label{eqn:hadsym}
\left( \begin{matrix}
1 & 1 & 1 & 1 & 1 & 1 \\
1 & -1 & i & -i & -i & i \\
1 & i & -1 & t & -t & -i \\
1 & -i & -\bar{t} & -1 & i & \bar{t} \\
1 & -i & \bar{t} & i & -1 & -\bar{t} \\
1 & i & -i & -t & t & -1 
\end{matrix} \right)
\end{equation}
and
\begin{equation}\label{eqn:hadchar}
\left( \begin{matrix}
1 & 1 & 1 & 1 & 1 & 1 \\
1 & 1 & \om & \om & \om^2 & \om^2 \\
1 & 1 & \om^2 & \om^2 & \om & \om \\
1 & -1 & s & -s & t & -t \\
1 & -1 & s\om & -s\om & t\om^2 & -t\om^2 \\
1 & -1 & s\om^2 & -s\om^2 & t\om & -t\om 
\end{matrix} \right)
\end{equation}
are complex Hadamard. Note that \eqref{eqn:hadsym} is symmetric, while \eqref{eqn:hadchar} is the character table of $\zz{6}$ when $s = t = 1$. Now let 
\[
d := \frac{1-\sqrt{3}}{2} + i\sqrt{\frac{\sqrt{3}}{2}}.
\]
Then
\begin{equation}\label{eqn:hadd}
\left( \begin{matrix}
1 & 1 & 1 & 1 & 1 & 1 \\
1 & -1 & -d & -d^2 & d^2 & d \\
1 & \bar{d} & 1 & d^2 & -d^3 & d^2 \\
1 & -\bar{d}^2 & \bar{d}^2 & -1 & d^2 & -d^2 \\
1 & \bar{d}^2 & -\bar{d}^3 & \bar{d}^2 & 1 & -d \\
1 & \bar{d} & d^2 & -\bar{d}^2 & -\bar{d} & -1 
\end{matrix} \right)
\end{equation}
is skew-symmetric Hadamard. Two Hadamard matrices $B_1$ and $B_2$ are \textsl{equivalent} if 
\[
B_1 = P_1D_1B_2D_2P_2,
\]
where $P_i$ is a permutation matrix and $D_i$ is diagonal (with all diagonal entries having the same absolute value). Up to equivalence, \eqref{eqn:hadsym}, \eqref{eqn:hadchar} and \eqref{eqn:hadd} is the complete list of known order-$6$ Hadamard matrices. To this list, we add another class, which is skew-symmetric and a generalization of \eqref{eqn:hadd}.

\begin{comment}
See Tadej and Zyczkowski (2005) for a survey of known small order complex Hadamard matrices.
\end{comment}

\begin{lemma}
Let $s$, $t$ and $u$ be complex numbers of absolute value $1$ satisfying
\[
stu+s+t+u+2 = 0.
\]
Then
\begin{equation}\label{eqn:hadskew}
\left(\begin{matrix}
1 & 1 & 1 & 1 & 1 & 1 \\
1 & -1 & -s & s & \conj{t} & -\conj{t} \\
1 & -\conj{s} & -1 & u & \conj{s} & -u \\
1 & \conj{s} & \conj{u} & 1 & \overline{stu} & \conj{t} \\
1 & t & s & stu & 1 & u \\
1 & -t & -\conj{u} & t & \conj{u} & -1 \\
\end{matrix} \right) 
\end{equation}
is a complex Hadamard matrix. \qed
\end{lemma}

To see that \eqref{eqn:hadskew} is a generalization of \eqref{eqn:hadd}, take $s = t = d^2$ and $u = -\conj{d}$. Theorem \ref{thm:rdscover} and Lemma \ref{lem:rdsmub} show that any distance regular antipodal $n$-fold cover of $K_{k,k}$ with an appropriate automorphism group produces a set of $n+1$ mutually unbiased bases in $\cx^k$. In fact, a $3$-fold cover of $K_{6,6}$ exists. It was found by Farad\v{z}ev, Ivanov, and Ivanov \cite{fii}, although this description is due to Aldred \cite{ald2}, who works with a so-called \defn{tank-trap} (see \cite{ald1} for more details). 

Let $T_0$ be a $3 \times 5$ matrix where each entry is a subset of $\{\infty,0,\ldots,4\}$. The $i$-th row of $T_0$ (for $i \in \zz{5}$) is 
\[
(\{\infty,i\}, \{1+i,4+i\}, \{2+i,3+i\}).
\]
Each row is a $1$-factor of $K_6$, and the entire array is a $1$-factorization. Define two more arrays $T_1$ and $T_2$ by shifting the columns of $T_0$:
\[
T_j(i,h) = T_0(i,h-j).
\]
Here columns are indexed mod $3$. Clearly $T_1$ and $T_2$ are also $1$-factorizations. Now define the cover of $K_{6,6}$: let $B_i$ and $W_i$ be the fibres of the two colour classes, where $B_i(j)$ is the $j$-th vertex in fibre $i$ (taking $0 \leq j \leq 2$ and $0 \leq i \leq 4$). Then
\[
B_i(j) \sim W_k(h) \text{ if and only if } k \in T_j(i,h).
\]
Additionally, set $B_\infty(j) \sim W_k(j)$ and $W_\infty(j) \sim B_k(j)$ for all $j$ and $k$. 

Clearly, there is a matching between $B_i$ and $W_k$, since $T_j(i,h) = T_{j+l}(i,h+l)$ and therefore $B_i(j)$ and $W_k(h)$ if and only if $B_i(j+l)$ and $W_k(h+l)$ are. To see that $B_i(j)$ and $B_{i'}(j')$ have two common neighbours for $i \neq i'$, note that $W_k(h)$ is a common neighbour if and only if $k$ is in both $T_0(i,h-j)$ and $T_0(i',h-j')$. As $h$ runs over the columns of $T_0$, by inspection $T_0(i,h-j)$ and $T_0(i',h-j')$ have nontrivial intersection exactly twice. A similar argument applies for the vertices of $B_\infty$, and it follows that the graph is a distance-regular antipodal cover. 

If this graph had an automorphism group which acted regularly on each colour class, then four mutually unbiased bases would exist in $\cx^6$. Unfortunately, no such automorphism group exists. An exhaustive computer search of the two abelian groups of order $18$ shows that there is no relative difference set of size $6$. 

\begin{comment}
Note that there is an automorphism group which permutes the vertices in each fibre: namely, for $l \leq 2$,
\[
\begin{cases}
B_i(j) \mapsto B_i(j+l), \\
W_k(h) \mapsto W_k(h+l).
\end{cases}
\]
Similarly, there is a group which permutes fibres $\{B_0,\ldots,B_4\}$ and $\{W_0,\ldots,W_4\}$:
\[
\begin{cases}
B_i(j) \mapsto B_{i+l}(j), \\
W_k(h) \mapsto W_{k+l}(h).
\end{cases}
\]
\end{comment}

\section{Type-II matrices}\label{sec:typeii}

In this section, we consider the connection between mutually unbiased bases and a class of type-II matrices called spin models. The term ``spin model" refers to a model of statistical mechanics that the matrices represent, while ``type II'' refers to the second Reidemeister move, an operation under which any link invariant must remain constant. The connection between the two was found by Jones \cite{jon1}. For more of an introduction to knot theory, link invariants, and the connections to Lie algebras, see Kauffman \cite{kau1}. The results in this section, unless otherwise noted, are due to Godsil.

\begin{comment}
Type-II matrices were not named until 1998 \cite{jmn}.
\end{comment}

Let $W$ be an $n \times n$ matrix with no zero entries. Then the \defn{Schur inverse} of $W$ is the matrix $W^{(-)}$ such that
\[
W \circ W^{(-)} = J.
\]
An invertible, Schur-invertible matrix \named{$W$}{type II matrix} is \defn{type II} if
\[
W W^{(-)T} = nI.
\]

Recall that if $W$ is mutually unbiased with $I$ in $\cx^n$, then $W$ is unitary and flat with entries of absolute value $1/\sqrt{n}$. Then
\[
W^{(-)} = n\overline{W},
\]
which implies that $W$ is type II. More generally, any two of the following imply the third:
\begin{enumerate}[(a)]
\item some real multiple of $W$ is unitary;
\item $W$ is flat;
\item $W$ is type II.
\end{enumerate}
Moreover, satisfying all three conditions is equivalent to some multiple of $W$ being unitary and mutually unbiased with $I$. There is another characterization of type-II matrices due to Godsil and Chan \cite{gc1}.

\begin{lemma}\label{lem:typeiichar}
An $n \times n$ matrix $W$ is type II if and only if 
\[
n\tr(AW^{-1}BW) = \tr(A)\tr(B), 
\]
for every diagonal $A$ and $B$. \qed
\end{lemma}

Define the \defn{Schur ratio} of columns $i$ and $j$ of $W$ to be the $i$-th column of $W$ Schur-divided by the $j$-th column:
\[
W_{i/j} := We_i \circ We_j^{(-)}.
\]
\nomenclature{$W_{i/j}$}{Schur-ratio of columns}%
Then $W$ is a \defn{spin model} if $W$ is type II, and every $W_{i/j}$ is an eigenvector for $W$. 

One example of a spin model is the following: let $\tha$ be a root of unity such that $\tha^2$ is a primitive $n$-th root, and define an $n \times n$ matrix
\[
W_{ij} := \tha^{(i-j)^2}.
\]
Clearly, $W$ is flat, and 
\begin{align*}
(W^*W)_{ij} & = \sum_k\tha^{-(i-k)^2+(j-k)^2} \\
& = \tha^{j^2-i^2} \sum_k \tha^{2k(i-j)} \\
& = \begin{cases}
n, & i=j; \\
0, & i \neq j.
\end{cases}
\end{align*}
Thus $W/\sqrt{n}$ is unitary and $W$ is type II. Moreover,
\[
(W_{j/i})_k = \tha^{-(i-k)^2+(j-k)^2} = \tha^{j^2-i^2}\om^{2(i-j)k},
\]
which is an eigenvector for the circulant $W$. Thus $W$ is a spin model.

\begin{lemma}
Let $W$ be an $n \times n$ spin model and let $D_j$ be diagonal with 
\[
(D_j)_{i,i} := \sqrt{n}(W^{(-)})_{i,j}.
\]
Then 
\[
D_jWD_j^{-1} = W^{-1}D_jW.
\]
\end{lemma}

\begin{proof}
Let $W_{i/j}$ have eigenvalue $\la_{i/j}$. Since the $i$-th column of $D_jW$ is $\sqrt{n}W_{i/j}$, an eigenvector, it follows that $WD_jW = D_jW\La_j$, where $\La_j$ is diagonal with $(\La_j)_{ii} = \la_{i/j}$. Now using Lemma \ref{lem:typeiichar} with $A = E_{ii}$ (the matrix with $ii$ entry $1$ and zeros elsewhere) and $B = D_j^{-1}$, we get
\begin{align*}
tr(D_j) & = n\tr(E_{ii}W^{-1}D_j^{-1}W) \\
& = n\tr(E_{ii}\La_j(D_jW)^{-1}) \\
& = n(\La_j)_{ii}(D_jW)_{ii}^{-1},
\end{align*}
From which it follows that $\La_j = D_j$. \qed
\end{proof}

\begin{comment}
Note that
\[
(\La_j)_ii = \tr(D_j)(D_jW)_{ii}/n = \tr(D_j)W_{ii}/nW_{ij},
\]
while $(D_j)_ii = \sqrt{n}/W_{ij}$. So it suffices to show that $W_{ii}$ is constant.
\end{comment}

If $D$ is diagonal and all diagonal entries have absolute value $1$, then $D$ is unitary. Then $D^{-1}WD$ is unitary whenever $W$ is unitary, and moreover $D^{-1}WD$ is flat whenever $W$ is flat. Therefore, when $W$ is a unitary spin model and $D = D_j$, we have that $D_jWD_j^{-1} = W^{-1}D_jW$ is both flat and unitary.

\begin{corollary}
If $W$ is a unitary spin model, then $I$, $W$, and $D_jW$ are mutually unbiased.
\end{corollary}

Since there is a spin model of order $n$ for every $n$, spin models produce three mutually unbiased bases in every dimension. 

All of the known maximal sets of mutually unbiased bases are equivalent to sets of the form 
\[
\scr{B} = \{I,W,D_1W,\ldots,D_{n-1}W\},
\]
where $W$ is a flat type-II matrix and $D_i$ is diagonal with entries of absolute value $1$. In particular, $W$ is the character table of an abelian group, and $\{D_0 = I,D_1,\ldots,D_{n-1}\}$ is also a group of diagonal matrices. If the diagonals do form a group so that $D_i^*D_j = D_k$, then $\scr{B}$ is a set of mutually unbiased bases if and only if
\[
(D_iW)^*(D_j^*W) = W^*D_kW
\]
is flat for each $D_k$.

\subsection{Orthogonal decompositions of Lie algebras}\label{sec:lie}

Here we describe an important connection between mutually unbiased bases and subalgebras of $\slnc$ discovered by Boykin, Sitharam, Tiep, and Wocjan \cite{bstw}.
\nomenclature[a$s]{$\slnc$}{trace-zero $n \times n$ matrices}

A \defn{Lie algebra} is an algebra with a skew-symmetric bilinear bracket multiplication satisfying the \defn{Jacobi identity}:
\[
[X,[Y,Z]]+[Y,[Z,X]]+[Z,[X,Y]] = 0.
\]
Note that skew-symmetry implies that $[X,X] = 0$. Given an associative algebra $A$, the Lie product
\[
[X,Y] := XY-YX
\]
turns $A$ into a Lie algebra. The algebra we are interested in is $\slnc$, the set of $n \times n$ complex matrices with trace zero. This Lie algebra is \textsl{simple}\index{Lie algebra!simple}: the only proper ideal is the trivial ideal. As a vector space over $\cx$, the dimension of $\slnc$ is $n^2-1$.

A \defn{Cartan subalgebra} of a simple Lie algebra is a maximal abelian subalgebra. If $\cH$ is an abelian subalgebra of $\slnc$ in which every matrix is normal, then commutativity implies that $\cH$ is simultaneously diagonalizable. The traceless diagonal matrices form a vector space of dimension $n-1$, so $\dim(\cH) \leq n-1$ with equality if and only if $\cH$ is Cartan.

\begin{comment}
More generally, a Cartan subalgebra is a maximal self-normalizing subalgebra: if $[A,B] \in \cH$ for all $B \in \cH$, then $A$ is also in $\cH$.
\end{comment}

Since the bracket product is bilinear, the map $X \mapsto [A,X]$ is a linear operation, denoted $\ad A$. The \defn{Killing form} of a Lie algebra is a nondegenerate bilinear form defined by 
\[
K(X,Y) := \tr(\ad X \ad Y).
\]
In the case of $\slnc$, it can be shown that this reduces to
\[
K(X,Y) = 2n\tr(XY).
\]
Now suppose the Lie algebra $A$ can be decomposed (as a vector space) into a direct sum of Cartan subalgebras:
\[
A = \cH_0 \oplus \ldots \oplus \cH_h.
\]
An \textsl{orthogonal}\index{Lie algebra!orthogonal decomposition} decomposition refers to one in which every $\cH_i$ and $\cH_j$ are orthogonal with respect to the Killing form. In the case of $\slnc$, the decomposition has $n+1$ subalgebras. 

See Kostrikin and Tiep \cite{kt1} for a more detailed introduction to orthogonal decompositions or de Graaf \cite{deg1} for Lie algebras in general.

\begin{theorem}
There exists a set of $k$ mutually unbiased bases in $\cx^n$ if and only if there exists a set of $k$ pairwise orthogonal normal Cartan subalgebras of $\slnc$. In particular, there exists a maximal set of mutually unbiased bases if and only if there exists a normal orthogonal decomposition.
\end{theorem}

\begin{proof}
Let $\cH_0,\ldots,\cH_k$ be a set of normal Cartan subalgebras. Since $\cH_i$ is simultaneously diagonalizable, let $B_i$ be a complete set of orthonormal eigenvectors for $\cH_i$. We show that the set of $B_0,\ldots,B_k$ are mutually unbiased. 

Suppose $H_i$ is in $\cH_i$, so $H_iB_i = B_iD_i$ for some diagonal $D_i$. Then for different $i$ and $j$, since $K(H_i,H_j) = 0$ we have
\[
2n\tr(H_iH_j) = 2n\tr(B_iD_iB_i^*B_jD_jB_j^*) = \tr(D_i(B_j^*B_i)^{-1}D_j(B_j^*B_i)) = 0.
\]
Note also that $\tr(D_i) = \tr(H_i) = 0$. Now the set of diagonals $D_i$ for $\cH_i$, together with $I$, span the space of all diagonal matrices. Letting $W = B_j^*B_i$, we have for any diagonals $D$ and $D'$,
\[
\tr(DW^{-1}D'W) = \tr(D)\tr(D').
\]
By Lemma \ref{lem:typeiichar}, $W$ is type II. Since $W$ is also unitary, it is therefore flat, and hence $B_i$ and $B_j$ are mutually unbiased. The converse is similar. \qed
\end{proof}

\begin{comment}
Recall that normal is equivalent to unitarily diagonalizable. By normal Cartan subalgebras, we mean Cartan subalgebras in which every matrix is normal.
Let $B_i$ be unbiased, and $D_i$ a spanning set of trace-$0$ diagonals (plus $D_0 = I$). Then $W = B_j^*B_i$ is flat type-II, and working backwards, $H_{ij} = B_iD_jB_i^*$ gives orthogonal normal Cartan subalgebras.
\end{comment}

A Cartan subalgebra $\cH$ is \textsl{monomial}\index{Cartan subalgebra!monomial} if it has a basis of monomial matrices. Similarly, a set of mutually unbiased bases is \textsl{monomial}\index{mutually unbiased bases!monomial}  if it is equivalent to a set of bases which are the eigenvalues of monomial Cartan subalgebras. The following theorem is from Kostrikin and Tiep \cite{kt1}.

\begin{theorem}
In $sl_6(\cx)$ there are at most $3$ pairwise orthogonal monomial Cartan subalgebras. \qed
\end{theorem}

\begin{corollary}
In dimension $6$, there are at most $3$ monomial mutually unbiased bases. \qed
\end{corollary}

As an example, the subalgebras spanned by the matrices in Lemma \ref{lem:compauli} are monomial. This implies that all known maximal sets of mutually unbiased bases are monomial, since all known maximal sets are equivalent to a set constructed from symplectic spreads.

\section{Real MUBs}\label{sec:realmub}

In this section we review what is known about mutually unbiased bases in $\re^d$. While these bases are not as useful as the complex ones for quantum measurements, they do have connections to coding theory. Moreover, the questions of existence are probably much easier, because the search space for unbiased bases is finite for any given dimension.

In the complex case, at most $d+1$ mutually unbiased bases exist in $\cx^d$; here, at most $d/2+1$ bases exist in $\re^d$. However, more can be said depending the particular value of $d$.

If $B_0$ and $B_1$ are real and mutually unbiased, then by applying an orthogonal transformation we may assume $B_0 = I$ and $B_1$ is flat. Up to some scalar multiple, a real, flat, unitary matrix is a Hadamard matrix, which can exist only in dimensions which are multiples of $4$. 

\begin{lemma}
At most $2$ real mutually unbiased bases exist in $\re^{4s}$ if $s$ is not square. 
\end{lemma}

\begin{proof}
Let $B_i$ be mutually unbiased with $I$. The angle between lines from different bases is $\al = 1/4s$, and
\[
B_i = \frac{1}{\sqrt{4s}}H_i,
\]
where $H_i$ is a Hadamard matrix. Now suppose $B_1$ and $B_2$ are mutually unbiased with each other as well as $I$. Then 
\[
B_1^TB_2 = \frac{1}{4s}H_1^TH_2
\]
is flat and also has entries of absolute value $1/\sqrt{4s}$. Since the entries of $H_1^TH_2$ are integers, this implies $\sqrt{4s}$ is an integer and $s$ is square. \qed
\end{proof}

Boykin, Sitharam, Tarifi, and Wocjan \cite{bstw1} used another counting argument for Hadamard matrices to find a second bound. 

\begin{lemma}
At most $3$ real mutually unbiased bases exist in $\re^{4s}$ if $s$ is odd.
\end{lemma}

A result analogous to Corollary \ref{cor:distdesscheme2} applies to real bases: if $G$ is the Gram matrix of a set of real mutually unbiased bases, then $\{I,G\}$ is coherently-weighted configuration. 

\subsection{Constructions}

Sets of complex lines with small angles may be used to construct sets of real lines.

\begin{lemma}\label{lem:realcomplex}
Let $X$ be a set of vectors in $\cx^d$ such that 
\[
\abs{u^*v}^2 \leq \ep
\]
for all $u$ and $v$ in $X$. Then there is a set of $2|X|$ vectors in $\re^{2d}$ satisfying the same bound.
\end{lemma}

\begin{proof}
Let $v = (a_1+ib_1,\ldots,a_d+ib_d)$ be a vector in $X$, with $a_j$ and $b_j$ real. Then we construct two vectors in $\re^{2d}$:
\begin{align*}
v_1 & = (a_1,b_1,\ldots,a_d,b_d), \\
v_2 & = (b_1,-a_1,\ldots,b_d,-a_d).
\end{align*}
Note that $v_1$ and $v_2$ are orthogonal. Similarly, given $u = (c_1+id_1,\ldots,c_n+id_n)$ construct $u_1$ and $u_2$. Then
\begin{align*}
\abs{u^*v}^2 & = \Big|\sum_j(a_j - ib_j)(c_j + id_j)\Big|^2 \\
& = \Big|\sum_j a_jc_j + b_jd_j + i(a_jd_j - b_jc_j)\Big|^2 \\
& = \Big(\sum_j a_jc_j + b_jd_j \Big)^2 + \Big(\sum_j a_jd_j - b_jc_j\Big)^2 \\
& = \abs{v_1^Tu_1}^2 + \abs{v_1^Tu_2}^2.
\end{align*}
Since $\abs{u^*v}^2$ is at most $\ep$, so is each of $\abs{v_1^Tu_1}^2$ and $\abs{v_1^Tu_2}^2$. Similarly, 
\[
\abs{u^*v}^2 = \abs{v_2^Tu_1}^2 + \abs{v_2^Tu_2}^2.
\]
Therefore all of the angles between $v_1$, $v_2$, $u_1$, and $u_2$ are at most $\ep$. \qed
\end{proof}

Suppose $d$ is a power of $2$, and $B_0 = I,\ldots,B_d$ is one of the known maximal sets of mutually unbiased bases in $\cx^d$. Applying Lemma \ref{lem:realcomplex} produces a maximal set of $d+1$ mutually unbiased bases in $\re^{2d}$. These sets were originally constructed by Cameron and Seidel \cite{cs1}.

\begin{comment} Some the vectors must be multiplied by $e^{i\pi/4}$.
\end{comment}

Several of the results for complex lines in Chapter \ref{chap:constructions} can be specialized to real lines for fields of characteristic $2$. For example, Theorem \ref{thm:caydist} implies that if the graph $\Cay(\zz{2}^k,D)$ has exactly $s$ nontrivial eigenvalues which are distinct in absolute value, then there is an $s$-distance set of size $|G|$ in $\re^{|D|}$. Another example is Corollary \ref{cor:nearcode}, in which codes of length $n$ over $\zz{2}$ are mapped to $\re^n$ via $x \mapsto (-1)^x$. In fact, as noted by Cameron and Seidel \cite{cs1} and Calderbank et al. \cite{ccks}, applying Corollary \ref{cor:nearcode} to binary Kerdock codes produces maximal sets of mutually unbiased bases. (Kerdock codes are closely related to orthogonal spreads, which have the same role for real bases as symplectic spreads do for complex bases.)

The binary Kerdock code $K(m)$ is a code of length $2^m$ with the following weight distribution:
\[
\begin{array}{c|c}
Weight & Multiplicity \\ 
\hline
0 & 1 \\
2^{m/2}\pm 2^{m/2-1} & 2^m(2^{m-1}-1) \\
2^{m/2} & 2^{m+1}-2 \\
2^m & 1
\end{array}
\]
After discarding one of $\{x,\one+x\}$, without loss of generality the remaining code has these weights:
\[
\begin{array}{c|c}
Weight & Multiplicity \\ 
\hline
0 & 1 \\
2^{m/2}- 2^{m/2-1} & 2^m(2^{m-1}-1) \\
2^{m/2} & 2^m-1 \\
\end{array}
\]
Mapping $\{0,1\}$ to $\{1,-1\}$, the words of weight $2^{m/2}$ are orthogonal to $\zero$, and the words of weight $2^{m/2}-2^{m/2-1}$ all have angle $2^{m/2}$ with $\zero$. Since the Kerdock code is distance transitive, the same angles occur for any codewords, and these lines form a set of $2^{m-1}$ real mutually unbiased bases. Together with the standard basis, this is a maximal set.

\begin{comment}
Calderbank et al. also described these results using orthogonal spreads, in analogy with symplectic spreads. This is the standard construction for Kerdock codes.
\end{comment}

When the dimension is not a power of $2$, Boykin et al. \cite{bstw1} gave a construction using Latin squares. Given a $\sqrt{d} \times \sqrt{d}$ Hadamard matrix and $k$ mutually orthogonal Latin squares of order $\sqrt{d}$, there exist $k+2$ mutually unbiased bases in $\re^d$. This is the best known lower bound for $d \neq 2^m$.

\chapter{Equiangular Lines}\label{chap:eals}

Equiangular lines are perhaps the most interesting instance of complex lines with few angles. They have an even wider range of applications than mutually unbiased bases and have significant connections to combinatorics (for example, Corollary \ref{cor:maxequides} showed that maximal sets of equiangular lines are equivalent to minimal complex $2$-designs). Most intriguingly, there is significant evidence that maximal sets exist in every dimension, but only a small number of dimensions actually have proofs. In this chapter, we summarize the known maximal sets and try to extend the ideas to higher dimensions.

\subsection*{Applications}

Like mutually unbiased bases, one of the primary motivations for studying complex equiangular lines comes from quantum measurements. A measurement $\{M_1,\ldots,M_n\}$ is \textsl{informationally complete}\index{informationally complete} if it is possible to recover any density matrix $\rho$ from the measurement statistics $p_i = \tr(M_i\rho M_i^*)$. Since $\rho$ is a $d \times d$ Hermitian matrix with trace $1$, it has $d^2-1$ degrees of freedom. Therefore an informationally complete measurement must have at least $d^2$ matrices (since the probabilities $p_i$ sum to $1$, a measurement's degrees of freedom is one fewer than the number of matrices).

If $M_i$ has rank one, then $E_i := M_i^*M_i$ is proportional to a projection matrix for a pure quantum state, say $v_i$. A measurement is \textsl{symmetric}\index{measurement!symmetric} if $\tr(E_iE_j)$ is a constant for all $i \neq j$; such measurements make it particularly easy to reconstruct $\rho$. With this in mind, a \textsl{symmetric informationally complete POVM} or \defn{SIC-POVM} is a symmetric POVM consisting of $d^2$ rank-one matrices with constant trace. Since the matrices sum to $I$, each $E_i$ has trace $1/d$. Since 
\[
\tr(E_iE_j) = \frac{1}{d^2}\abs{v_i^*v_j}^2
\]
is a constant, a SIC-POVM is equivalent to a set of $d^2$ equiangular lines in $\cx^d$. 

\begin{comment}
To construct $\rho$: the $E_i$'s span the Hermitian matrices, so let $\rho = \sum_i x_i E_i$. Then 
\[
p_i = \tr(\rho E_i) = \sum_j\tr(x_iE_jE_i) = \frac{x_i}{d} + \sum_{j \neq  i}\frac{x_i}{d^2(d+1)}.
\]
This particular system of linear equations $p = (aJ+bI)x$ is straightforward to solve: it has the form $x = (a'J+b'I)p$. The maximal error in $x_i$ is a function of the max error in $p_i$. For any other set of $d^2$ informationally complete lines, the maximal error would be at least big owing the uniformity of $aJ+bI$.
\end{comment}

One example of a quantum application of equiangular lines is in fingerprinting. In classic fingerprinting, Alice first sends a message $x$ to Bob over an unsecured public channel. Then Alice sends Bob a single bit $a$ from $x$ over an authenticated public channel. (Alice chooses a position in $x$ at random, and transmits the bit $a$ along with its position so that Bob may compare $a$ to the appropriate bit in $x$.) If $a$ matches $x$, then Bob takes the message to be authentic; this will always be the case if Eve did not tamper with $x$. However, Eve might replace $x$ with $y$, which also matches $a$. This is called \defn{one-sided error}. The authentication process is repeated with different bits $a$ until Alice and Bob are satisfied.

To minimize error, Alice and Bob might choose an initial pool of valid messages $C$ such that any pair from $C$ has a small number of bits in common; this ensures that the probability of authenticating $y$ instead of $x$ is small. If $p(x,y)$ is the number of bits $x$ and $y$ have in common, then the goal is minimize the \defn{worst-case error} probability:
\[
P_{wce} = \max_{x \neq y \in C} p(x,y).
\]
This could be accomplished, for example, with a binary error-correcting code of large distance.

In quantum fingerprinting, assume the authentication ``bit" is some pure quantum state $\rho(x) = v_xv_x^*$. Authentication consists of measuring $\rho(x)$ using the POVM $\{E_1 = \rho(x),E_2 = I-\rho(x)\}$. The message is authenticated if the outcome is $1$ (that is, $\rho(x)$ is measured with respect to $E_1$). Again there is a one-sided error, where with some probability Bob could take a substitute message $\rho(y)$ as valid. If the valid message pool is a finite set of pure states $C$, then the worst-case error probability is
\[
P_{wce} = \max_{x \neq y \in C} \abs{v_x^*v_y}^2.
\] 
This error is minimized when $C$ is a set of equiangular lines. 

\begin{comment}
The fingerprint is what is replaced by a quantum state, sent over a secure quantum channel.
\end{comment}

For more details on quantum fingerprinting, see for example Scott, Walgate, and Sanders \cite{sws}. Equiangular lines have several other applications to quantum information: like mutually unbiased bases, they have been used in quantum cryptographic protocols (see Fuchs and Sasaki \cite{fs1}) and in quantum tomography (see Caves, Fuchs, and Schack \cite{cfs}). Minimizing the error $P_{wce}$ also has applications in classical communications. In that context, maximal sets of equiangular lines are sometimes called \defn{Grassmannian frames} (see Strohmer and Heath \cite{sh1}) or \textsl{$2$-uniform $(n,d)$-frames} (see Bodmann and Paulsen \cite{bp1}).

\begin{comment}
Fuchs and Sasaki replaced MUBs with SIC-POVMs in the key distribution. There is a trade-off between rate of key accumulation and security from Eve. Equiangular lines are optimal under a certain set of assumptions.
\end{comment}

\subsection*{Background}

The problem of equiangular lines in $\cx^d$ was studied as early as 1975 by Delsarte, Goethals, and Seidel \cite{dgs}, who, in addition to proving the upper bound of $d^2$ lines, found maximal sets in dimensions $2$ and $3$. Since then, others such as Hoggar \cite{hog1} and K\"onig \cite{kon1} have examined lines in $\cx^d$ with various applications in mind. Equiangular lines were introduced in the quantum setting by Zauner \cite{zau1} in 1999. The first major progress in finding maximal sets in $\cx^d$ was made by Renes, Blume-Kohout, Scott and Caves \cite{ren1} in 2003, when they found numerical solutions for $d \leq 45$. This strongly suggests that $d^2$ lines exist for every $d$, but a proof seems elusive. Analytic solutions have now been found for $d \leq 8$ and $d = 19$.

Corollary \ref{cor:maxequides} (the \textsl{relative bound}) stated that if $X$ is a set of equiangular lines in $\cx^d$ with angle $\al$, then
\[
|X| \leq \frac{d(1-\al)}{1-d\al},
\]
with equality if and only if $X$ a $1$-design. Solving for $\al$, we get one case of the \defn{Welch Lower Bound} (see \cite{wel1}).

\begin{corollary}
If $X$ is a set of equiangular lines in $\cx^d$, then 
\[
\al \geq \frac{|X|-d}{d(|X|-1)},
\]
with equality if and only if $X$ is a $1$-design. \qed
\end{corollary}

\begin{comment}
In fact, the Welch bound is a lower bound on the maximum angle across any set of $|X|$ lines in $\cx^d$, not just equiangular lines.
\end{comment}

Corollary \ref{cor:maxequides} (the \textsl{absolute bound}) stated that if $X$ is a set of $d^2$ equiangular lines in $\cx^d$, then $X$ is a $2$-design. In this case, the relative bound implies that 
\[
\al = \frac{1}{d+1}.
\]

\section{Fiducial vectors}\label{sec:fiducial}

In Section \ref{sec:prime} we introduced the generalized Pauli matrices $X(a)$ and $Y(a)$ for $a$ in some finite vector space $V$. We can also define these matrices over $\zz{d}$\index{generalized Pauli matrices}. Let $\{e_j: j \in \zz{d}\}$ be the standard basis for $\cx^d$, and let $\om$ be an $d$-th primitive root of unity in $\cx$. Then the Pauli matrices for $\zz{d}$ act as follows:
\begin{align*}
X(j): & \; e_k \mapsto e_{k+j}; \\
Y(j): & \; e_k \mapsto \om^{jk}e_k.
\end{align*}
When $d$ is prime, $X(j)$ and $Y(j)$ coincide with the Pauli matrices for a vector space. 

The quotient group
\[
G := \lb X(j)Y(k) : j,k \in \zz{d} \rb / \lb \om I \rb
\]
has order $d^2$, and every element can be written uniquely as $X(j)Y(k)\lb \om I \rb$. Moreover, $G$ is isomorphic to $\zz{d}^2$ as a free module over $\zz{d}$. As in Lemma \ref{lem:2groupcom}, $X(j)Y(k)$ commutes with $X(j')Y(k')$ if and only if $j'k = jk'$. The proof of the following is the same as that of Lemma \ref{lem:oddsemiequiv}. Note that if $\tha = -e^{\pi i/d}$, then $\tha^2 = \om$ is a primitive $d$-th root of unity, while $\tha^d$ is $1$ or $-1$ if $d$ is odd or even respectively.

\begin{lemma}
Define $\theta = -e^{\pi i/d}$, and let $k = jz$ in $\zz{d}$. Then
\[
\phi_{z,y} := \sum_{x \in \zz{d}} \tha^{zx^2 +  2yx} e_x
\]
is an eigenvector for $X(j)Y(k)$. Moreover, for each $z$, the set $\{\phi_{z,y}: y \in \zz{d} \}$ spans $\cx^d$. \qed
\end{lemma}

\begin{comment}
In the case when $d$ is even, $\tha$ has order $2d$. In fact it does not matter whether we consider $x$, $y$, and $z$ to be in $\zz{d}$ or in $\zz{2d}$. For, let $f(x,y,z) = zx^2 +  2yx$ in $\zz{2d}$. Then $f(x,y,z) = f(x+d,y,z) = f(x,y+d,z)$, so we may take $x$ and $y$ in $\zz{d}$. On the other hand, if we take $z \in \zz{2d}$, then it suffices to note that
\[
\tha^{(z+d)x^2 +  2yx} = \tha^{dx^2}\tha^{zx^2 +  2yx} = (-1)^{x^2}\tha^{zx^2 +  2yx} = (-1)^x\tha^{zx^2 +  2yx} = \tha^{zx^2 +  2(y+d/2)x},
\]
and therefore $\phi_{z+d,y} = \phi_{z,y+d/2}$.

\begin{proof}
\begin{align*}
X(a)Y(b)\phi_{z,y} & = X(j)Y(k)\sum_{x \in \zz{n}} \tha^{zx^2 +  2yx} e_x \\
& = \sum_{x \in \zz{n}} \tha^{zx^2 +  2yx + 2kx} e_{x+j} \\
& = \sum_{x \in \zz{n}} \tha^{z(x+j)^2 +  2y(x+j) + 2kx - 2jzx - zj^2 - 2yj} e_{x+j} \\
& = \tha^{- zj^2 - 2yj}\phi_{z,y}.
\end{align*}
\end{proof} 
\end{comment}

Almost all of the known constructions of maximal sets of equiangular lines have the form 
\[
\{X(j)Y(k)v : j,k \in \zz{d}\}
\]
for some $v \in \cx^d$. The vector $v$ is called the \defn{fiducial vector}. The lone exception is Hoggar's set of $64$ lines in $\cx^8$, which uses the group $\zz{2}^3$ instead of $\zz{8}$. Hoggar's construction is discussed in Section \ref{sec:hoggar}. 

Zauner \cite{zau1} finds a fiducial vector for every $d \leq 5$, and Renes et al.~\cite{ren1} find all possible vectors for $d \leq 4$. In dimension $2$, there are two possible vectors (up to orbits under the Pauli group):
\[
v := \frac{1}{\sqrt{6}}\left(\begin{matrix}
\pm\sqrt{3 \pm \sqrt{3}} 
\vspace{0.1cm} \\
e^{i\pi/4}\sqrt{3 \mp \sqrt{3}}
\end{matrix}\right).
\]
In dimension $3$, there is an infinite number of fiducial vectors. One example is
\[
v := \frac{1}{\sqrt{2}}\left(\begin{matrix}
1 \\
1 \\
0 \\
\end{matrix}\right).
\]
For dimensions $4$, $5$, $6$, and $7$, Renes et al.~found that there are $16$, $80$, $96$, and $336$ vectors respectively. However, the analytic descriptions become more complicated as the dimension increases. Grassl \cite{gra1} gives a fiducial vector in dimension $6$ which takes two-thirds of a page to describe. Appleby \cite{app1} found analytic solutions for $d = 7$ and $d = 19$.

\begin{comment}
\subsection{Clifford group}
\end{comment}

Denote by \named{$GP(d)$}{generalized Pauli group} the group of matrices of the form $\tha^k X(i)Y(j)$, with $i,j \in \zz{d}$ and $k \in \zz{2d}$. Then the \defn{Clifford group} \named{$C(d)$}{Clifford group} is the group of unitary matrices $U$ that normalize $GP(d)$:
\[
UGP(d)U^* = GP(d).
\]
Suppose $U$ is in the Clifford group, and $M$ is in $GP(d)$. Then there is some other $M'$ in $GP(d)$ such that
\[
(Uv)^*M(Uv) = v^*M'v.
\]
Thus if $v$ is a fiducial vector, then so is $Uv$. 

An operation $U: \cx^d \rightarrow \cx^d$ is \defn{anti-linear} if, for all $\al, \be \in \cx$ and $u,v \in \cx^d$,
\[
U(\al u + \be v) = \al^*U(u) + \be^*U(v).
\]
Any anti-linear operation is a linear operation composed with the complex conjugacy operation. Let $U^*: \cx^d \rightarrow \cx^d$ denote the unique operation such that 
\[
u^*U(v) = v^*U^*(u)
\]
for any $u$ and $v$. Then an \defn{anti-unitary} operation is an anti-linear operation such that $U \circ U^*$ is the identity. The \defn{extended Clifford group} \named{$EC(d)$}{extended Clifford group} is the group of unitary and anti-unitary operations that normalize $GP(d)$. 

Each of the fiducial vectors constructed by Renes et al.~\cite{ren1} is an eigenvector for some element $U$ in the extended Clifford group. Moreover, each such $U$ has order $3$. Zauner \cite{zau1} conjectures that this is the case in all dimensions. (See Appleby \cite{app1} for more details.)

\begin{comment}
Let $f_U$ and $g_U$ be the functions on $\zz{d}$ such that 
\[
UX(i)Y(j)U^*\lb \tha I \rb = X(f_U(i))Y(g_U(j))\lb \tha I \rb.
\]
Since $X(i)Y(j)X(i')Y(j')\lb \tha I \rb = X(i+i')Y(j+j')\lb \tha I \rb$, it follows that $f_U$ and $g_U$ are linear. Since $f_U$ and $g_U$ are only functions of $i$ and $j$, there is a $2 \times 2$ matrix $F_U$ such that
\[
\twovec{f_U(i)}{g_U(j)} = F_U \twovec{i}{j}.
\]
\end{comment}

\section{The difference set construction}\label{sec:ealdiff}

Corollary \ref{cor:diffequi} showed that if $D$ is a $(v,k,\la)$-difference set in an abelian group $G$, then the characters of $G$, restricted to $D$, are a set of $v$ equiangular lines in $\cx^k$. Since $v = (k^2-k+\la)/\la$, this produces the most lines for a given $k$ when $\la = 1$ and $v = k^2-k+1$. These lines are maximal with respect to the relative bound (Corollary \ref{cor:ealrelbnd}).

If $k=q+1$ for some prime power $q$, then the Singer subgroup for \named{$PG(2,q)$}{projective plane over $GF(q)$} is a cyclic group of size $v = q^q+q+1$ containing a $(v,k,1)$-difference set. (For details, see Beth, Jungnickel, and Lenz \cite[Theorem VI.1.9]{bjl}.) The lines from this set were constructed by K\"onig \cite{kon1} in 1999 and then rediscovered by Xia, Zhou, and Giannakis \cite{xzg}. (K\"onig observed the construction only when $q$ is prime, while Xia, Zhou, and Giannakis in fact found the more general difference set construction of Corollary \ref{cor:diffequi} for any cyclic group.)

\begin{comment}
The Singer subgroup is an automorphism group of $PG(2,q)$, cyclic of size $q^2+q+1$. Identifying $PG(2,q)$ with $GF(q^3)$, if $\tha$ is a primitive element, then $\tha^{q^2+q+1}$ generates the cycle on the points. The difference set is simply of the lines.
\end{comment}

Since the lines in this construction are characters (restricted to a particular subset), the vectors are flat. Although the lines are not maximal with respect to the absolute bound, they are maximal with respect to flatness. The following result is new.

\begin{lemma}
There are at most $k^2-k+1$ flat equiangular lines in $\cx^k$.
\end{lemma}

\begin{proof}
Let $\{v_1,\ldots,v_m\}$ be a set of $m$ flat equiangular lines in $\cx^k$, and let $e_1,\ldots,e_k$ be the standard basis. Then consider the Gram matrix $G$ of \[
X := \{v_1v_1^*,\ldots,v_mv_m^*,e_1e_1^*,\ldots,e_ke_k^*\}
\]
For vectors $u$ and $v$, the entry in the Gram matrix is
\[
G_{uu^*,vv^*} = \tr(uu^*vv^*) = \abs{u^*v}^2.
\]
Assume $\abs{v_i^*v_j}^2 = \al$. Then we have
\[
G = \left( \begin{matrix}
\al J + (1-\al)I & \frac{1}{k}J \\
\frac{1}{k}J & I \\
\end{matrix} \right),
\]
where each $J$ is an appropriately sized all-ones matrix. Using elementary row operations, $G$ is row-equivalent to 
\[
\left( \begin{matrix}
(\al - \frac{1}{k})J + (1-\al)I & 0 \\
\frac{1}{k} J & I \\
\end{matrix} \right).
\]
It is then easy to find the eigenvalues of $G$. When the relative bound from Corollary \ref{cor:ealrelbnd} holds, $G$ has rank $m+k-1$. Otherwise, it has full rank $m+k$. In either case, the rank, which is also the dimension of the span of $X$, is at least $m+k-1$. But $X$ is a subset of the Hermitian matrices, a space of dimension $k^2$. Thus
\[
m+k-1 \leq k^2,
\]
and so $m \leq k^2-k+1$. \qed
\end{proof}

\section{Hoggar's construction}\label{sec:hoggar}

In this section and the next, we discuss two particular constructions of equiangular lines due to Hoggar \cite{hog1} and Appleby \cite{app1}, and we show that certain natural generalizations do not work in higher dimensions. Lemmas \ref{lem:hoggen} and \ref{lem:almostflat} are new.

Hoggar found $64$ equiangular lines in $\cx^8$. He worked with quaternions, but there is a simple description of his construction using generalized Pauli matrices. Let $V = GF(2)^3$. Then the Pauli matrices $\{X(a),Y(a): a \in V\}$ generate a group of size $128$ in which every element may be written $\pm X(a)Y(b)$.

\begin{lemma}
Let 
\[
r := \sqrt{2}, \; s := \frac{1+i}{\sqrt{2}}, \; t := \frac{1-i}{\sqrt{2}},
\] 
and let 
\[
v := (0,0,s,t,s,-s,0,r).
\] 
Then
\[
\{X(a)Y(b)v : a,b \in V\}
\]
is a set of $64$ equiangular lines in $\cx^8$. \qed
\end{lemma}

It is natural to ask if there is a similar construction for other powers of two. Let $V(k,2) = GF(2)^k$, and consider the group $G_k$ generated by the Pauli matrices of $V(k,2)$. 

\begin{lemma}\label{lem:hoggen}
For any $v \in \cx^{2^k}$, the lines
\[
\{X(a)Y(b)v : a,b \in V(k,2)\}
\]
can only be equiangular for $k \in \{1,3\}$.
\end{lemma}

\begin{proof}
Let $X$ and $Y$ denote the $2 \times 2$ Pauli matrices. Then modulo $-I$, the matrices of $G_k$ have the form 
\[
\Om_k := \{I,X,Y,XY\}^{\otimes k},
\]
and the angles of interest have the form $v^*Mv$, for $M \in \Om_k$. Let $v = (v_1,\ldots,v_d)$, where $d = 2^k$. Then the angles will give a system of constraints on the values of $v_i$.

Let $\al_j = v_j^*v_j$. Then from $v^*Iv = 1$, we get
\[
\al_1+ \ldots+\al_d = 1.
\]
Next consider 
\[
\abs{v^*(I \otimes \ldots \otimes I \otimes Y)v} = \frac{1}{\sqrt{d+1}}.
\]
Since $v_j^*v_j$ is real, the value of $v^*(I \otimes \ldots \otimes I \otimes Y)v$ must be real, and we get
\[
\al_1-\al_2+\ldots+\al_{d-1}-\al_d = \pm \frac{1}{\sqrt{d+1}}.
\]
More generally, let $\al = (\al_1,\ldots,\al_d)$, and let $H$ be the following $d \times d$ Hadamard matrix:
\[
H = \left(\begin{matrix} 
1 & 1 \\ 
1 & -1 
\end{matrix} \right)^{\otimes k}.
\]
Then by considering $v^*Av$ for $A \in \{I,Y\}^{\otimes k}$, we get the following system of real equations:
\[
H \al = \frac{1}{\sqrt{d+1}}\left(\begin{matrix}
\sqrt{d+1} \\
\pm 1 \\
\vdots \\
\pm 1
\end{matrix} \right).
\]
Since $H^{-1} = \frac{1}{d}H$, this system is easily solved for $\al$:
\begin{equation}
\al_j = \frac{\sqrt{d+1} + l_j}{d \sqrt{d+1}},
\label{eqn:ali}
\end{equation}
for some odd integer $l_j$. 

Generalizing this, suppose $a \in V(k,2)$ and $a^Ta = 1$ (mod $2$). Let $\sg$ be the involution of $V(k,2)$ corresponding to $X(a)$, so that $\sg$ takes the coordinate for $x \in V(k,2)$ to the coordinate for $x+a$. Also let $f_j = v_{\sg(j)}^*v_j$. Then from $v^*X(a)v$, we get
\[
f_1 + f_1^* + \ldots + f_{d-1} + f_{d-1}^* = \pm \frac{1}{\sqrt{d+1}}.
\]
Again, since $f_j + f_j^*$ is real, the right-hand side is also real. From $v^*X(a)Y(a)v$, we get
\[
f_1 - f_1^* + \ldots + f_{d/2} - f_{d/2}^* = \pm \frac{i}{\sqrt{d+1}}.
\]
Here, both the left and right sides are purely imaginary. In fact for each $b \in V$, from $v^*X(a)Y(b)v$ we get a purely real or purely imaginary linear constraint involving $\pm f_j$ and $\pm f_j^*$. Letting $f = (f_1, f_1^*, \ldots, f_{d-1}, f_{d-1}^*)$, we have
\[
H P f = \frac{1}{\sqrt{d+1}}\left(\begin{matrix}
\pm 1 \\
\pm i \\
\vdots \\
\pm 1 \\
\pm i
\end{matrix} \right),
\]
for some permutation $P$. The solutions in $f$ are of the form
\[
f_j \in \frac{\pm \{0,2,4,\ldots\} \pm \{0,2,4,\ldots\}i}{d \sqrt{d+1}}.
\]
Thus,
\[
f_jf_j^* = \frac{m_j}{d^2(d+1)},
\]
for some integer $m_j$; that is, $f_jf_j^*$ is rational. However, 
\[
f_jf_j^* = v_{\sg(j)}^*v_jv_j^* v_{\sg(j)} = \al_j\al_{\sg(j)}.
\]
This is true for all $\sg$, so $\al_j\al_{j'}$ is rational for any $j \neq j'$. From\eqref{eqn:ali},
\[
\al_j\al_{j'} = \frac{d+1 + l_jl_{j'} + (l_j+l_{j'})\sqrt{d+1}}{d^2(d+1)},
\]
which is rational if and only if $\sqrt{d+1}$ is rational or $l_j = -l_{j'}$. If $l_j = -l_{j'}$ for all $j \neq j'$, then there are only two possible indices of $j$ and $j'$, so $d = 2$. On the other hand, $\sqrt{2^k+1}$ is rational only if $k = 3$. We conclude that the lines can only be equiangular for $d \in \{2,8\}$. \qed
\end{proof}

\section{Almost flat constructions} 

Appleby \cite{app1} constructed fiducial vectors in dimensions $7$ and $19$, which have a very similar description. Recall that if $d$ is prime, the \defn{Legendre symbol} on $\zz{d}$ is defined as follows:
\[
\leg{x}{d} := \begin{cases}
0, & x = 0; \\
1, & x \text{ is a quadratic residue mod } d; \\
-1 & x \text{ is not  a quadratic residue mod } d.
\end{cases}
\]
More generally, the \defn{Jacobi symbol} is defined for any odd $d$. If $d = p_1^{k_1}p_2^{k_2}\ldots p_r^{k_r}$, then the Jacobi symbol is 
\[
\leg{x}{d} := \leg{x}{p_1}^{k_1}\leg{x}{p_2}^{k_2}\ldots\leg{x}{p_r}^{k_r}.
\]
%\nomenclature{$\leg{x}{d}$}{Jacobi symbol}%
Now define two constants:
\[
a := \sqrt{\frac{1-1/\sqrt{d+1}}{d}}, \quad b := \sqrt{\frac{1+(d-1)/\sqrt{d+1}}{d}},
\]
and consider the following equation in $y$:
\begin{equation} \label{eqn:quartic}
(2by+(d-1)ay^2-a)^2 + 4(1-y^2)(b-ay)^2 = \frac{1}{a^2(d+1)}.
\end{equation}
This equation is quartic; call the solutions $y_1,y_2,y_3,y_4$. Then both of Appleby's fiducial vectors $v = (v_1,\ldots,v_d)$ have the form 
\[
v_x = \begin{cases}
ae^{i\cos^{-1}(y_j)(\frac{x}{d})}, & 1 \leq x \leq d-1; \\
b, & x = d.
\end{cases}
\]
The question, then, is whether or not there are fiducial vectors of this form for dimensions other than $7$ and $19$. For each odd $d$, we can solve equation \eqref{eqn:quartic} and test if $v$ is fiducial with each solution $y_j$. Using Maple, we find that there are no other fiducial vectors of this form for odd $d$ less than $400$. 

However, we can at least confirm that for fiducial vectors that are flat except for one entry, the absolute values of the entries must be exactly Appleby's choices of $a$ and $b$. The following result is new.

\begin{lemma}\label{lem:almostflat}
Let $v$ be a fiducial vector in $\cx^d$ such that one entry of $v$ has squared absolute value $b$, and all other entries have squared absolute value $a$. Then
\[
a = \frac{1 \pm 1/\sqrt{d+1}}{d}; \quad b = \frac{1 \mp (d-1)/\sqrt{d+1}}{d}.
\]
\end{lemma}

\begin{proof}
Assume without loss of generality that the first entry of $v$ has absolute value $b$. Let $v=v_1,v_2,\ldots,v_{d^2}$ be the action of the Pauli matrices on $v$, where for $j,k \leq d$, the $k$-th entry of $v_{dj+k}$ has absolute value $b$. Then consider the Gram matrix of 
\[
X := \{v_1v_1^*,\ldots,v_{d^2}v_{d^2}^*,e_1e_1^*,\ldots,e_de_d^*\}.
\]
Clearly, $\abs{e_l^*v_{dj+k}}^2$ is $b$ if $l = k$ and $a$ otherwise. For simplicity, let $\al = 1/(d+1)$. Then the Gram matrix (written with $(d+1)^2$ blocks of size $d$) is
\[
G = \left(\begin{matrix}
\al J + (1-\al)I & \al J & \ldots & \al J & aJ + (b-a)I \\
\al J & \al J + (1-\al)I & \ldots & \al J & aJ + (b-a)I\\
& & \vdots & & \\
\al J & \al J & \ldots & \al J + (1-\al)I & aJ + (b-a)I  \\
aJ + (b-a)I & aJ + (b-a)I & \ldots & aJ + (b-a)I & I 
\end{matrix} \right).
\]
By subtracting linear combinations of the last $d$ rows, we can find the eigenvalues of $G$. Let
\[
x := \al-2ab-(d-2)a^2; \quad y := (b-a)^2.
\]
Then $G$ is equivalent by row and column operations to
\begin{align*}
G' & = \left(\begin{matrix}
xJ + (1-\al-y)I & xJ - yI & \ldots & xJ - yI & 0 \\
xJ - yI & xJ + (1-\al-y)I & \ldots & xJ - yI & 0 \\
& & \vdots & & \\
xJ - yI & xJ - yI & \ldots & xJ + (1-\al-y)I & 0 \\
0 & 0 & \ldots & 0 & I 
\end{matrix} \right) \\
& = \left(\begin{matrix}
(1-\al)I_{d^2} + (xJ_d-yI_d) \otimes J_d & 0 \\
0 & I_d
\end{matrix} \right),
\end{align*}
which has the following eigenvalues:
\[
\{ 1^{(d)}, (1-\al)^{(d^2-d)}, (1-\al-dy)^{(d-1)}, (1-\al+d^2x - dy)^{(1)} \}.
\]
Since the elements of $X$ are Hermitian matrices, the rank of $G$ must be at most $d^2$; therefore, the last two eigenvalues must be $0$. Solving for $a$ and $b$, we get the stated values. \qed
\end{proof}

\begin{comment}
There is a common theme among the proofs of this chapter, namely the so-called \defn{rank-bound}. Roughly speaking, the rank bound is simply the fact that the size of a set of linearly independent objects is at most the dimension of the space in which they reside. One example is Fischer's inequality, in which the ``objects'' are the characteristic vectors of points of $2$-design $(V,B)$, which resdie in the space $\re^B$. In this case we treat the characteristic vectors as the rows of the incidence matrix $N$ and show they are linearly independent by showing $NN^T = (r-\la)I + \la J$ is nonsingular.

In the case of this chapter, the objects are a set $X$ of Hermitian $d \times d$ matrices, which reside in a real vector space of dimension $d^2$. Writing each $M$ as a vector of length $d^2$, we can construct a $d^2 \times |X|$ ``incidence matrix'' $N$. The Gram matrix $G$ of $X$ is $NN^T$, and so
\[
\rk(G) \leq \rk(N) = \dim(\spn(X)) \leq d^2. 
\]
\end{comment}

\chapter{Future Work}\label{chap:outro}

A number of issues relating to the construction of maximal sets of complex lines are unresolved. Most importantly:

\begin{itemize}
\item How many mutually unbiased bases exist in $\cx^d$, when $d$ is not a prime power? Are all maximal sets of mutually unbiased bases monomial?
\item Do $d^2$ equiangular lines exist in $\cx^d$ for every $d$? If so, are they always the orbits of fiducial vectors?
\end{itemize}

Another issue is raised by the weighted adjacency algebras of Chapter 2. At present, the Gram-matrix algebras of Delsarte, Goethals, and Seidel in Section \ref{sec:gram} are the only known examples that come from Hermitian (but not symmetric) homogeneous coherently-weighted configurations. It would be interesting to know if other examples exist, and if so, whether or not they also have combinatorial interpretations. Also, recall that the weighted adjacency matrices of Section \ref{sec:gram} have the form
\[
A'_i = A_i \circ G,
\]
where $G$ is a Gram matrix and $A_i$ is a Schur idempotent in an association scheme. Is that true of every Hermitian homogeneous coherently-weighted configuration?

Yet another issue comes from the constructions of Chapter \ref{chap:constructions}. While Chapter \ref{chap:bounds} gives bounds on the size of an $s$-distance set for any $s$, very little is known about constructions for $s \geq 3$. Historically there has been less motivation to study the problem in its full generality, and the difficulty almost certainly increases with $s$. Even when $s = 2$, apart from the results already mentioned, there are only a few known constructions: Delsarte et al. \cite{dgs} refer to $2$-distance sets of size $45$ in $\cx^5$ and size $4060$ in $\cx^{28}$, each of which satisfies the relative bound with equality. However, it seems likely that the results in Chapter \ref{chap:constructions} could be applied when $s$ is greater than $2$, and this would be worth investigating.

\section*{Mutually unbiased subspaces}

There is one extension of these problems that seems to be very well motivated: using subspaces instead of unit vectors. Recall that if $M_u$ and $M_v$ are the projection matrices onto the one-dimensional subspaces spanned by vectors $u$ and $v$ respectively, then the angle between $u$ and $v$ is 
\[
\abs{u^*v}^2 = \tr(M_uM_v).
\]
More generally, given two subspaces $U$ and $V$ with projection matrices $M_U$ and $M_V$, consider $\tr(M_UM_V)$, which we call the \textsl{overlap} between $U$ and $V$. If $U$ and $V$ are orthogonal, then $\tr(M_UM_V) = 0$. What, then, is the maximum size of a collection of subspaces of a fixed dimension with a given set of overlaps? This question has some important implications in quantum computing.

Recall from the introduction that a projective measurement in quantum mechanics is a set of projection matrices $M_1,\ldots,M_n$ which decompose $\cx^d$ into orthogonal subspaces. Call a collection of measurements \textsl{mutually unbiased} if each measurement is projective and the overlap between any two subspaces from distinct measurements is some fixed $\al$. Mutually unbiased bases are a special case. In terms of subspaces, the objective here is to find a maximal $2$-distance set with overlaps $\{0,\al\}$, where the subspaces can be partitioned into orthogonal decompositions of $\cx^n$.

At first glance it may appear that this problem may be reduced to that of finding mutually unbiased bases by observing that the inner product space of $d \times d$ matrices is isomorphic to $\cx^{d^2}$. However, because the matrices we are looking for must be projections, the two questions are not equivalent.

Such collections of measurements are useful in \textsl{quantum tomography}, where the objective is to reproduce a quantum state using only certain restricted types of measurements. Gibbons, Hoffman, and Wootters \cite{ghw} described how to perform quantum tomography using a \textsl{Wigner function}, which, given a state $\rho$ and collection of projective measurements $\{M_i\}$, is essentially a formula for reconstructing $\rho$ from the measurement statistics $\tr(M_i\rho)$. Gibbons et al.~only considered mutually unbiased bases, but their method applies to any mutually unbiased projections. Investigating the existence of these subspaces is therefore a natural direction for future research. 

\begin{comment}
In this case, suppose a $d \times d$ state $\rho$ can only be measured using projections onto subspaces of dimension $d/2$; this can happen, for example, if only one qubit in the system is easy to manipulate. Each measurement $\{M_i, I-M_i\}$ (where $M_i$ is a projective matrix) has one degree of freedom, so in principle $\rho$ can be reconstructed from a collection of $d^2-1$ projective measurements. As with mutually unbiased bases, the reconstruction is simplest using mutually unbiased measurements.
\end{comment}

%\bibliographystyle{agbib} %mwnbib
%\bibliography{Lcodes}

\begin{thebibliography}{10}

\bibitem{ald2}
{\sc R.~E.~L. Aldred}, {\em Distance-regular antipodal covering graphs}, PhD
  thesis, University of Melbourne, 1986.

\bibitem{ald1}
{\sc R.~E.~L. Aldred and C.~D. Godsil}, {\em Distance-regular antipodal
  covering graphs}, J. Combin. Theory Ser. B, 45 (1988), 127--134.

\bibitem{all1}
{\sc W.~O. Alltop}, {\em Complex sequences with low periodic correlations},
  IEEE Trans. Inform. Theory, 26 (1980), 350--354.

\bibitem{app1}
{\sc D.~M. Appleby}, {\em Symmetric informationally complete-positive operator
  valued measures and the extended {C}lifford group}, J. Math. Phys., 46
  (2005), 052107, 29.

\bibitem{pr1}
{\sc M.~H.~R. Arthur O.~Pittenger}, {\em Wigner functions and separability for
  finite systems}, J. Phys. A: Math. Gen., 38 (2005), 6005--6036.

\bibitem{abr}
{\sc S.~Axler, P.~Bourdon, and W.~Ramey}, {\em Harmonic function theory},
  vol.~137 of Graduate Texts in Mathematics, Springer-Verlag, New York,
  second~ed., 2001.

\bibitem{bbrv}
{\sc S.~Bandyopadhyay, P.~O. Boykin, V.~Roychowdhury, and F.~Vatan}, {\em A new
  proof for the existence of mutually unbiased bases}, Algorithmica, 34 (2002),
  512--528.

\bibitem{bh1}
{\sc E.~Bannai and S.~G. Hoggar}, {\em Tight {$t$}-designs and squarefree
  integers}, European J. Combin., 10 (1989), 113--135.

\bibitem{bar1}
{\sc R.~H. Barker}, {\em Group synchronizing of binary digital sequences},
  Communication Theory,  (1953), 273--287.

\bibitem{BB84}
{\sc C.~H. Bennett and G.~Brassard}, {\em Quantum cryptography: {P}ublic key
  distribution and coin tossing}, in Proceedings of IEEE international
  Conference on Computers, Systems and Signal Processing, Bangalore, India, New
  York, 1984, IEEE {P}ress, 175.

\bibitem{bjl}
{\sc T.~Beth, D.~Jungnickel, and H.~Lenz}, {\em Design {T}heory}, vol.~69 of
  Encyclopedia of Mathematics and its Applications, Cambridge University Press,
  Cambridge, second~ed., 1999.

\bibitem{bp1}
{\sc B.~G. Bodmann and V.~I. Paulsen}, {\em Frames, graphs and erasures},
  Linear Algebra Appl., 404 (2005), 118--146.

\bibitem{bstw1}
{\sc P.~O. Boykin, M.~Sitharam, M.~Tarifi, and P.~Wocjan}, {\em Real mutually
  unbiased bases}, www.arxiv.org/quant-ph/0502024,  (2005).

\bibitem{bstw}
{\sc P.~O. Boykin, M.~Sitharam, P.~H. Ti{\cfudot{e}}p, and P.~Wocjan}, {\em
  Mutually unbiased bases and orthogonal decompositions of lie algebras},
  www.arXiv.org/quant-ph/0506089,  (2005).

\bibitem{bcn}
{\sc A.~E. Brouwer, A.~M. Cohen, and A.~Neumaier}, {\em Distance-{R}egular
  {G}raphs}, Springer-Verlag, Berlin, 1989.

\bibitem{bun1}
{\sc L.~Bungart}, {\em Boundary kernel functions for domains on complex
  manifolds}, Pacific J. Math., 14 (1964), 1151--1164.

\bibitem{ccks}
{\sc A.~R. Calderbank, P.~J. Cameron, W.~M. Kantor, and J.~J. Seidel}, {\em
  {$Z\sb 4$}-{K}erdock codes, orthogonal spreads, and extremal {E}uclidean
  line-sets}, Proc. London Math. Soc. (3), 75 (1997), 436--480.

\bibitem{cam2}
{\sc P.~J. Cameron}, {\em Permutation groups}, vol.~45 of London Mathematical
  Society Student Texts, Cambridge University Press, Cambridge, 1999.

\bibitem{cs1}
{\sc P.~J. Cameron and J.~J. Seidel}, {\em Quadratic forms over {$GF(2)$}},
  Nederl. Akad. Wetensch. Proc. Ser. A {\bf 76}=Indag. Math., 35 (1973), 1--8.

\bibitem{cfs}
{\sc C.~M. Caves, C.~A. Fuchs, and R.~Schack}, {\em Unknown quantum states: the
  quantum de {F}inetti representation}, J. Math. Phys., 43 (2002), 4537--4559.

\bibitem{cw1}
{\sc M.~Cordero and G.~P. Wene}, {\em A survey of finite semifields}, Discrete
  Math., 208/209 (1999), 125--137.

\bibitem{cox1}
{\sc H.~S.~M. Coxeter}, {\em Regular complex polytopes}, Cambridge University
  Press, Cambridge, second~ed., 1991.

\bibitem{dj1}
{\sc J.~A. Davis and J.~Jedwab}, {\em A unifying construction for difference
  sets}, J. Combin. Theory Ser. A, 80 (1997), 13--78.

\bibitem{deg1}
{\sc W.~A. de~Graaf}, {\em Lie algebras: theory and algorithms}, vol.~56 of
  North-Holland Mathematical Library, North-Holland Publishing Co., Amsterdam,
  2000.

\bibitem{del1}
{\sc P.~Delsarte}, {\em An algebraic approach to the association schemes of
  coding theory}, Philips Res. Rep. Suppl.,  (1973), vi+97.

\bibitem{dg1}
{\sc P.~Delsarte and J.-M. Goethals}, {\em Tri-weight codes and generalized
  {H}adamard matrices}, Information and Control, 15 (1969), 196--206.

\bibitem{dgs}
{\sc P.~Delsarte, J.~M. Goethals, and J.~J. Seidel}, {\em Bounds for systems of
  lines, and {J}acobi polynomials}, Philips Res. Rep.,  (1975), 91--105.

\bibitem{dgs2}
\leavevmode\vrule height 2pt depth -1.6pt width 23pt, {\em Spherical codes and
  designs}, Geometriae Dedicata, 6 (1977), 363--388.

\bibitem{fii}
{\sc I.~A. Farad\v{z}ev, A.~A. Ivanov, and A.~V. Ivanov}, {\em
  Distance-transitive graphs of valency {$5$}, {$6$} and {$7$}}, European J.
  Combin., 7 (1986), 303--319.

\bibitem{fs1}
{\sc C.~A. Fuchs and M.~Sasaki}, {\em Squeezing quantum information through a
  classical channel: measuring the ``quantumness'' of a set of quantum states},
  Quantum Inf. Comput., 3 (2003), 377--404.

\bibitem{gar1}
{\sc A.~Gardiner}, {\em Antipodal covering graphs}, J. Combinatorial Theory
  Ser. B, 16 (1974), 255--273.

\bibitem{ghw}
{\sc K.~S. Gibbons, M.~J. Hoffman, and W.~K. Wootters}, {\em Discrete phase
  space based on finite fields}, Phys. Rev. A (3), 70 (2004), 062101, 23.

\bibitem{gc1}
{\sc C.~Godsil and A.~Chan}, {\em Type-{II} matrices}.
\newblock 2004.

\bibitem{god8}
{\sc C.~D. Godsil}, {\em Polynomial spaces}, Discrete Math., 73 (1989), 71--88.

\bibitem{blue}
\leavevmode\vrule height 2pt depth -1.6pt width 23pt, {\em Algebraic
  {C}ombinatorics}, Chapman \& Hall, New York, 1993.

\bibitem{god2}
{\sc C.~D. Godsil and G.~F. Royle}, {\em Algebraic {G}raph {T}heory},
  Springer-Verlag, New York, 2001.

\bibitem{gg1}
{\sc S.~W. Golomb and G.~Gong}, {\em Signal design for good correlation},
  Cambridge University Press, Cambridge, 2005.

\bibitem{gra1}
{\sc M.~Grassl}, {\em On sic-povms and mubs in dimension 6},
  http://arxiv.org/quant-ph/0406175,  (2004).

\bibitem{haa1}
{\sc J.~Haantjes}, {\em Equilateral point-sets in elliptic two- and
  three-dimensional spaces}, Nieuw Arch. Wiskunde (2), 22 (1948), 355--362.

\bibitem{hkcss}
{\sc A.~R. Hammons, Jr., P.~V. Kumar, A.~R. Calderbank, N.~J.~A. Sloane, and
  P.~Sol{\'e}}, {\em The {${\bf Z}\sb 4$}-linearity of {K}erdock, {P}reparata,
  {G}oethals, and related codes}, IEEE Trans. Inform. Theory, 40 (1994),
  301--319.

\bibitem{hig1}
{\sc D.~G. Higman}, {\em Schur relations for weighted adjacency algebras}, in
  Symposia Mathematica, Vol. XIII (Convegno di Gruppi e loro Rappresentazioni,
  INDAM, Rome, 1972), Academic Press, London, 1974, 467--477.

\bibitem{hog1}
{\sc S.~G. Hoggar}, {\em {$64$} lines from a quaternionic polytope}, Geom.
  Dedicata, 69 (1998), 287--289.

\bibitem{hug1}
{\sc D.~R. Hughes}, {\em Partial difference sets}, Amer. J. Math., 78 (1956),
  650--674.

\bibitem{iva1}
{\sc I.~D. Ivanovi{\'c}}, {\em Geometrical description of quantal state
  determination}, J. Phys. A, 14 (1981), 3241--3245.

\bibitem{jon1}
{\sc V.~F.~R. Jones}, {\em On knot invariants related to some statistical
  mechanical models}, Pacific J. Math., 137 (1989), 311--334.

\bibitem{jun1}
{\sc D.~Jungnickel}, {\em On automorphism groups of divisible designs}, Canad.
  J. Math., 34 (1982), 257--297.

\bibitem{kan1}
{\sc W.~M. Kantor}, {\em On the inequivalence of generalized {P}reparata
  codes}, IEEE Trans. Inform. Theory, 29 (1983), 345--348.

\bibitem{kau1}
{\sc L.~H. Kauffman}, {\em Knots and diagrams}, in Lectures at KNOTS '96
  (Tokyo), vol.~15 of Ser. Knots Everything, World Sci. Publishing, River Edge,
  NJ, 1997, 123--194.

\bibitem{kr1}
{\sc A.~Klappenecker and M.~R{\"o}tteler}, {\em Constructions of mutually
  unbiased bases}, in Finite fields and applications, vol.~2948 of Lecture
  Notes in Comput. Sci., Springer, Berlin, 2004, 137--144.

\bibitem{kr2}
\leavevmode\vrule height 2pt depth -1.6pt width 23pt, {\em Mutually unbiased
  bases are complex projective $2$-designs},
  www.arxiv.org/abs/quant-ph/0502031,  (2005).

\bibitem{kon1}
{\sc H.~K{\"o}nig}, {\em Cubature formulas on spheres}, in Advances in
  multivariate approximation (Witten-Bommerholz, 1998), vol.~107 of Math. Res.,
  Wiley-VCH, Berlin, 1999, 201--211.
\newblock http://analysis.math.uni-kiel.de/koenig/ko4.ps.

\bibitem{koo2}
{\sc T.~Koornwinder}, {\em The addition formula for {J}acobi polynomials and
  spherical harmonics}, SIAM J. Appl. Math., 25 (1973), 236--246.

\bibitem{koo1}
{\sc T.~H. Koornwinder}, {\em The addition formula for {J}acobi polynomials.
  {I}. {S}ummary of results}, Nederl. Akad. Wetensch. Proc. Ser. A {\bf
  75}=Indag. Math., 34 (1972), 188--191.

\bibitem{kt1}
{\sc A.~I. Kostrikin and P.~H. Ti{\cfudot{e}}p}, {\em Orthogonal decompositions
  and integral lattices}, vol.~15 of de Gruyter Expositions in Mathematics,
  Walter de Gruyter \& Co., Berlin, 1994.

\bibitem{ler1}
{\sc R.~M. Lerner}, {\em Signals having good correlation functions}, IEEE
  WESCON Convention Record,  (1961).

\bibitem{lid1}
{\sc R.~Lidl and H.~Niederreiter}, {\em Finite {F}ields}, Cambridge University
  Press, Cambridge, second~ed., 1997.

\bibitem{mcd1}
{\sc B.~R. McDonald}, {\em Finite rings with identity}, Marcel Dekker Inc., New
  York, 1974.

\bibitem{mit1}
{\sc H.~H. Mitchell}, {\em The subgroups of the quaternary abelian linear
  group}, Trans. Amer. Math. Soc., 15 (1914), 379--396.

\bibitem{neu1}
{\sc A.~Neumaier}, {\em Combinatorial configurations in terms of distances},
  Eindhoven Unviersity of Technology, Memorandum 81-09 (1981).

\bibitem{nc1}
{\sc M.~A. Nielsen and I.~L. Chuang}, {\em Quantum computation and quantum
  information}, Cambridge University Press, Cambridge, 2000.

\bibitem{na1}
{\sc G.~M. Nikolopoulos and G.~Alber}, {\em Security bound of two-basis
  quantum-key-distribution protocols using qudits}, Physical Review A (Atomic,
  Molecular, and Optical Physics), 72 (2005), 032320.

\bibitem{pro1}
{\sc J.~G. Proakis}, {\em Digital Communications}, McGraw-Hill, New York, 1995.

\bibitem{rcw}
{\sc D.~K. Ray-Chaudhuri and R.~M. Wilson}, {\em On {$t$}-designs}, Osaka J.
  Math., 12 (1975), 737--744.

\bibitem{ren1}
{\sc J.~M. Renes, R.~Blume-Kohout, A.~J. Scott, and C.~M. Caves}, {\em
  Symmetric informationally complete quantum measurements}, J. Math. Phys., 45
  (2004).

\bibitem{rud1}
{\sc W.~Rudin}, {\em Function theory in the unit ball of {${\bf C}\sp{n}$}},
  vol.~241 of Grundlehren der Mathematischen Wissenschaften [Fundamental
  Principles of Mathematical Science], Springer-Verlag, New York, 1980.

\bibitem{sws}
{\sc A.~Scott, J.~Walgate, and B.~C. Sanders}, {\em Optimal fingerprinting
  strategies with one-sided error},  (2005).

\bibitem{sei1}
{\sc J.~J. Seidel}, {\em Geometry and combinatorics}, Academic Press Inc.,
  Boston, MA, 1991.

\bibitem{sh1}
{\sc T.~Strohmer and R.~W. Heath, Jr.}, {\em Grassmannian frames with
  applications to coding and communication}, Appl. Comput. Harmon. Anal., 14
  (2003), 257--275.

\bibitem{vs1}
{\sc N.~J. Vilenkin and R.~L. {\v{S}}apiro}, {\em Irreducible representations
  of the group {${\rm SU}(n)$} of class {I} relative to {${\rm SU}(n-1)$}},
  Izv. Vys\v s. U\v cebn. Zaved. Matematika, 1967 (1967), 9--20.

\bibitem{wel1}
{\sc L.~Welch}, {\em Lower bounds on the maximum cross correlation of signals},
  IEEE Transactions on Information Theory, 20 (1974), 397--399.

\bibitem{wb1}
{\sc P.~Wocjan and T.~Beth}, {\em New construction of mutually unbiased bases
  in square dimensions}, Quantum Inf. Comput., 5 (2005), 93--101.

\bibitem{wf1}
{\sc W.~K. Wootters and B.~D. Fields}, {\em Optimal state-determination by
  mutually unbiased measurements}, Ann. Physics, 191 (1989), 363--381.

\bibitem{xzg}
{\sc P.~Xia, S.~Zhou, and G.~B. Giannakis}, {\em Achieving the {W}elch bound
  with difference sets}, IEEE Trans. Inf. Theory, 51 (2005), 1900--1907.

\bibitem{zau1}
{\sc G.~Zauner}, {\em Quantendesigns-Grundzuge einer nichtkommutativen
  Designtheorie}, PhD thesis, University of Vienna, 1999.

\end{thebibliography}

%\printnomenclature[2cm]
\printindex

\end{document}